\documentclass[11pt]{article}
\usepackage{graphicx}
\usepackage{amsmath}
\usepackage{comment}
\allowdisplaybreaks
\usepackage{amssymb}
\usepackage{amsthm}
\usepackage{mathtools}
\usepackage{epsfig}
\usepackage{mathrsfs}
\usepackage[usenames,dvipsnames]{color}
\usepackage{setspace}
\usepackage{enumerate}
\usepackage{lipsum}
\usepackage{subcaption}
\usepackage{bm}
\usepackage{todonotes}
\usepackage[margin=2.54cm]{geometry}
\usepackage{algorithm, algorithmic}
\usepackage{multirow}
\usepackage{tikz}

\DeclareMathOperator{\st}{s.t.}

\usepackage{colortbl}
\usepackage{lipsum}
\newtheorem{theorem}{Theorem}

\newtheorem{lemma}{Lemma}
\newtheorem{proposition}{Proposition}
\newtheorem{assumption}{Assumption \!\!}

\newtheorem{remark}{Remark}

\usepackage{dsfont}

\newcommand{\Ep}{{\mathbb{E}}}
\newcommand{\U}{{\mathcal{U}}}
\newcommand{\K}{{\mathcal{K}}}
\newcommand{\F}{{\mathcal{F}}}
\newcommand{\CP}{{\mathcal{CP}}}
\newcommand{\COP}{{\mathcal{COP}}}

\newcommand{\Y}{{\mathcal{Y}}}

\newcommand{\D}{{\mathcal{D}}}

\newcommand{\SYM}{{\Bbb S}}
\newcommand{\PP}{\Bbb P}
\newcommand{\QQ}{\Bbb Q}
\newcommand\iid{i.i.d.}

\def\RR{ {\Bbb{R}}}

\newcommand\NS{\mbox{\tiny NS}}

\title{Data-Driven Distributionally Robust Appointment \\
Scheduling over Wasserstein Balls}

\author{%
Ruiwei Jiang \thanks{Department of Industrial and Operations Engineering,
University of Michigan, Ann Arbor, MI, 48109, USA. Email: {\tt ruiwei@umich.edu}.}%
\ \ \ \ \ \ \
Minseok Ryu \thanks{Department of Industrial and Operations Engineering,
University of Michigan, Ann Arbor, MI, 48109, USA. Email: {\tt msryu@umich.edu}.}%
\ \ \ \ \ \ \
Guanglin Xu\thanks{Institute for Mathematics and its Applications, University of Minnesota,
Minneapolis, MN, 55455, USA. Email: {\tt gxu@umn.edu}.}%
}

\date{}

\begin{document}
\maketitle

\begin{abstract}

\noindent We study a single-server appointment scheduling problem
with a fixed sequence of appointments, for which we must determine
the arrival time for each appointment.
We specifically examine two stochastic models. In the first model,
we assume that all appointees show up at the scheduled arrival times
yet their service durations are random. In the second model, we assume that
appointees have random no-show behaviors and their service
durations are random given that they show up at the appointments.
In both models, we assume that the probability distribution of the uncertain
parameters is unknown but can be partially observed via a set
of historical data, which we view as independent samples drawn from the unknown distribution. In view of the distributional ambiguity, we
propose a data-driven distributionally robust optimization (DRO) approach to determine
an appointment schedule such that the worst-case (i.e., maximum) expectation of
the system total cost is minimized. A key feature of this approach is that the optimal
value and the set of optimal schedules thus obtained provably converge to
those of the ``true'' model, i.e., the stochastic appointment scheduling model
with regard to the true probability distribution of the uncertain parameters.
While our DRO models are computationally intractable in general, we reformulate them to
copositive programs, which are amenable to tractable semidefinite programming
problems with high-quality approximations. Furthermore, under some mild conditions,
we recast these models as polynomial-sized linear programs. Through an extensive numerical study, we
demonstrate that our approach yields better out-of-sample performance than
two state-of-the-art methods.

\mbox{}

\end{abstract}

\begin{onehalfspace}

\section{Introduction}

Since the pioneering work of Bailey~\cite{Bailey.1952}, appointment systems have been extensively
studied in many customer service industries with the aim of increasing resource utilization, matching
supply and demand, and smoothing customer flows. The core operational activity in many
appointment systems is to schedule arrival times for appointments to minimize the system total cost associated with
appointment waiting, as well as the server's idleness and overtime. 
Serving as a central modeling component in a wide range of applications, appointment scheduling (AS) has been applied in
outpatient scheduling~\cite{Fetter.Thompson.1966, Ho.Lau.1992,
Kaandorp.Koole.2007, Kong.Lee.Teo.Zheng.2013}, surgery planning~\cite{Denton.Gupta.2003}, call-center
staffing~\cite{Gurvich.2010}, and cloud computing server operations~\cite{Shen.Wang.2014}.
In this paper, we study a single-server AS problem where the number and sequence of the appointments
are fixed. The main decision in this problem is to schedule the arrival time for each appointment, which
helps design the scheduling template of the appointment system.

There are several sources of variability that make AS problems challenging to solve, including
random service durations, random appointee no-shows, unpunctuality, and emergency
interruptions. In this study, we focus on random service durations and random appointee no-shows.
Due to random service durations, an appointment can complete before or after the scheduled
starting time of the subsequent appointment and result in the server's idleness or the waiting of
the subsequent appointment, respectively. In addition, if the last appointment is completed
after the pre-determined time limit, then the server has to work overtime, which is often costly
and unpleasant for service providers. Random appointee no-shows arise in many appointment
systems, e.g., outpatient clinics. No-shows often results in idleness of resources (e.g., equipment and
personnel) and so the loss of opportunities for serving other appointments. In fact, it has been reported that random no-shows have more impact on the performance
of an appointment system than random durations (see, e.g.,~\cite{Ho.Lau.1992}). Hence, it is recommended that we adapt AS models to the anticipated
no-show behaviors~\cite{Cayirli.Veral.2003}.


In the literature, {\em stochastic programming} (SP)~\cite{Kall.1994, Shapiro.2014} approaches are
often proposed to tackle AS problems given that the true distribution of the service durations and no-shows
is fully known~\cite{Berg.Denton.Erdogan.Rohleder.Huschka.2014, Erdogan.Denton.2013, Gupta.Denton.2008}.
While SP exhibits superior modeling power, it has two inherent shortcomings.
First, SP suffers from the notorious curse of dimensionality. Indeed, the computation of expectations necessitates
the evaluation of multi-dimensional integrals, which is in general intractable. Second, it is challenging and sometimes
impossible to accurately estimate the true distribution. For example, the raw data of the
uncertain parameters can typically be explained by many strikingly different distributions.
{\color{black}Using a biased estimate of the distribution, the SP approach can yield overfitted decisions, which display an optimistic bias and can lead to post-decision disappointment.} For example, if one simply uses the empirical distribution based on the raw data, the obtained solution often results in unpleasant out-of-sample performance.

In view of the distributional ambiguity, one can construct a so-called {\em ambiguity set} to contain all possible
distributions that may govern the generation of the observed data samples. With the ambiguity set, one can formulate
a distributionally robust optimization (DRO) problem with the goal of minimizing the worst-case (i.e., maximum)
expected system total cost of the appointment system, where the expectations are taken with respect to the distributions
from the ambiguity set. The majority of DRO approaches considered in the literature
are based on moment ambiguity sets, which consist of all distributions sharing certain moments,
e.g., the first and second moments.
Although a moment ambiguity set often leads to a tractable optimization problem~\cite{Bertsimas.Popescu.2015, Jiang.Shen.Zhang.2017, Kong.Lee.Teo.Zheng.2013, Mak.Rong.Zhang.2015}, it typically does not converge to the true distribution even in the situation where more data can be obtained.
This gives rise to the central questions of this study:~Is there an alternative DRO approach that extracts more information of the underlying true distribution from available (possibly small-size) historical data? If such an approach exists, is the resulting AS model solvable in polynomial time?


In this paper, we endeavor to give affirmative answers to these two research questions.
In particular, we propose to construct the ambiguity set using a Wasserstein ball centered
at the empirical distribution based on the historical data~\cite{ Esfahani.Kuhn.2017, Pflug.Wozabal.2007, Wozabal.2012}. Accordingly, we consider two Wasserstein-based distributionally
robust appointment scheduling (W-DRAS) models. In the first model, we examine the situation
where all appointees show up at their scheduled starting times but their service durations are random. In the second model, we incorporate both random no-shows and random service durations of the appointees (if they show up). Results of the modern measure concentration theory guarantee that the Wasserstein ball has asymptotic consistency, which ensures that the optimal value as well as the optimal solutions of the W-DRAS models converge to their SP counterparts with regard to the true distribution, as the data size tends to infinity.
While the resulting optimization problems are, in general, intractable,
we reformulate them to copositive programs, which admit tractable semidefinite programs that have high-quality approximations. Furthermore, under some mild conditions, we reformulate the W-DRAS problems
to polynomial-sized linear programs, which can be efficiently solved by many off-the-shelf optimization solvers.

\subsection{Literature review}

Various methodologies have been applied to formulate and solve the AS problem, including queueing theory (see, e.g.,~\cite{Brahimi.Worthington.1991}), approximation algorithm (see, e.g.,~\cite{Kaandorp.Koole.2007}), and optimization. We refer the readers to~\cite{Ahmadi-Javid.Jalali.Klassen.2017, Cardoen.Demeulemeester.Belien.2010, Cayirli.Veral.2003, Gupta.Denton.2008} for comprehensive summaries of these studies. In this section, we conduct a brief literature review on the most relevant literature to the proposed approach.

The SP models of the AS problem assume that the probability distribution of the uncertain parameters is fully known and seek a schedule to minimize the expectation of the system total cost. Begen and Queyranne~\cite{Begen.Queyranne.2011} show for the first time that the AS problem, where the random service
durations follow a joint discrete probability distribution, could be solved in polynomial time under some mild conditions. Ge et al.~\cite{Ge.Wan.Wang.Zhang.2013} extend the result
of~\cite{Begen.Queyranne.2011} to the case where the cost is modeled
as piecewise linear convex functions of the waiting time and idleness.

When it comes to general probability distributions,
exact calculation of the multi-dimensional integral poses computational challenges.
To this end, sample average approximation (SAA) methods are often used to
approximately solve AS problems. Denton and Gupta~\cite{Denton.Gupta.2003}
formulate the AS problem as a two-stage stochastic linear program and propose
 a sequential bounding approach to determine upper
bounds on the optimal value. In a more recent work, Begen~et.~al.~\cite{Begen.Levi.Queyranne.2012} propose a
sampling-based approach, whereby one can construct
an empirical distribution over a set of historical data and quantify the related computation complexity to obtain a
near-optimal solution in terms of sample size. However,
SAA approaches often lead to optimistic bias, which motivates us to consider a DRO approach in this paper.

Modeling the no-show behavior in appointment scheduling systems is even more challenging due to the discrete nature of the no-show parameters. For example, the no-show of an appointee is often modeled by a Bernoulli random variable which equals one if the appointee shows up and zero if the appointee does not show up.
Ho and Lau~\cite{Ho.Lau.1992} take a first step to develop a heuristic approach that suggests to double book the first two appointments and subsequently schedule the remaining appointments. Following the
work in~\cite{Ho.Lau.1992}, a number of more advanced yet more sophisticated approximation and
heuristic approaches have been proposed to address the random no-show issue in various
settings~\cite{Chen.Robinson.2014, Denton.Gupta.2003, Erdogan.Denton.2013,
Hassin.Mendel.2008, LaGanga.Lawrence.2007}.

In reality, it is often difficult to fit the true distribution of the service durations and no-shows
for various reasons, e.g., lack of sufficient data~\cite{Macario.2009} and the existence of correlations~\cite{Cayirli.Veral.2003}. Robust approaches have been proposed to address this challenge based on partial information of the distribution. In particular, the classical robust optimization approach (see, e.g.,~\cite{Mittal.Schulz.Stiller.2014, Rowse.2015}) models the uncertain parameters based only on an uncertainty set (e.g., the support or a confidence set of the uncertainty). Differently, the DRO approach employs an ambiguity set that incorporates a family of distributions.
We refer the readers to the classical works~\cite{Bertsimas.Gupta.Kallus.2017,
Delage.Ye.2010, Esfahani.Kuhn.2017, Goh.Sim.2010, Scarf.1958} and references therein
for general DRO models and solution approaches.
Closely related to this research are the following papers. Kong et~el.~\cite{Kong.Lee.Teo.Zheng.2013}
propose a DRO model over a cross-moment ambiguity set consisting of all distributions with common mean and covariance of the random service durations.
They recast this model as a copositive program and solve its
semidefinite approximation to obtain upper bounds on the optimal value.
Differently, Mak et~el.~\cite{Mak.Rong.Zhang.2015} consider a marginal-moment ambiguity set consisting of all distributions with common marginal moments up to a finite order. They recast the corresponding DRO model as a semidefinite program for general marginal-moment ambiguity
sets and, in particular, a second-order cone program for the mean-variance ambiguity set and a linear program for
the mean-support ambiguity set. Recently, Jiang et~el.~\cite{Jiang.Shen.Zhang.2017} study a mean-support ambiguity set for both random no-shows and service durations. As the no-show parameters are discrete, they propose an integer
programming reformulation and develop a decomposition algorithm to solve the resulting mixed integer
programs.
Kong et~el.~\cite{Kong.Li.Liu.Teo.Yan.2015} consider a cross-moment ambiguity set with decision-dependent no-shows, i.e., the first and second moments of the no-shows depend on the appointment arrival times.
They propose a copositive programming reformulation for the corresponding DRO model and develop an algorithm to search for an optimal schedule by iteratively solving a series of semidefinite programs. In contrast to the aforementioned work that consider moment ambiguity sets, we propose a Wasserstein-based ambiguity set that enjoys asymptotic consistency (see Section~\ref{sec:numerical} for a demonstration based on a \emph{finite} data size).

The sequence of the appointments is assumed
to be fixed in this paper. In the literature,
determining the best sequence of appointments for a set of heterogeneous appointments is an interesting topic; see~\cite{Denton.Viapiano.Vogl.2007, Gupta.Denton.2008, He.Sim.Zhang.2015,
Mak.Rong.Zhang.2014, Mak.Rong.Zhang.2015}. In some situations, appointment systems are allowed to overbook a time slot with multiple appointments. To this end,~\cite{LaGanga.Lawrence.2012, Liu.Ziya.2014, Zacharias.Pinedo.2014}
propose various approaches to provide optimal overbooking policies. Multiple-server AS has also been
considered in the literature; see for example~\cite{Zacharias.Pinedo.2017}.

\subsection{Our contributions}

We highlight the main contributions of this paper as follows.
\vspace{-0.05in}
\begin{enumerate}[1.]
\item We propose a data-driven DRO approach for the AS problem with random service durations and random no-shows over Wasserstein balls. To the best of our knowledge, this is the first DRO approach applied to the AS problem that enjoys the asymptotic consistency.
\item We make technical contribution to the DRO literature. The W-DRAS model with random service durations is a two-stage DRO problem with random right-hand sides. When we further incorporate the no-shows, randomness arises from both the right-hand sides and the objective coefficients. While two-stage DRO problems are in general NP-hard (see~\cite{Bertsimas.Doan.Natarajan.Teo.2010}), we recast both W-DRAS models as copositive programs, which are amenable to tractable semidefinite programming approximations. More importantly, under mild conditions, we recast both W-DRAS models as convex programs that are solvable in polynomial time. In particular, if the Wasserstein ball is characterized by an $\ell_p$-norm, we show that W-DRAS admits a linear programming reformulation when $p = 1$ and a second-order conic reformulation when $p > 1$ and $p$ is rational.
\item We derive probability distributions that attain the worst-case expected total cost in the W-DRAS models. These distributions can be applied to stress test an appointment schedule generated from any decision-making processes. In addition, when we interpret W-DRAS as a two-person game between the AS scheduler and the nature, these distributions provide interesting insights on how the nature plays against a given appointment schedule.
\item We demonstrate the effectiveness of our approach over diverse test instances. In particular, our approach yields (i) near-optimal appointment schedules even with a small data size and (ii) better out-of-sample performance than two state-of-the-art methods, even when the distribution is misspecified. This demonstrates that the W-DRAS approach is effective in AS systems (i) with scarce data and/or (ii) in a quickly varying environment.
\end{enumerate}


\subsection{Paper structure and notation} \label{sec:wass}

The remainder of the paper is organized as follows. We study the W-DRAS problem with random service durations in Section~\ref{sec:duration_model} and extend the model to incorporate both random no-shows and random service durations in Section~\ref{sec:no_show}. We test the effectiveness of our approach
over diverse test instances in Section~\ref{sec:numerical}. Finally, we conclude and discuss future research directions in Section~\ref{sec:conclusions}. We relegate all technical proofs to the Appendix.

{\noindent \bf Notation:} For any $m \in \Bbb{N}$, we define $[m]$ as the set of running indices
$\{1, \ldots, m\}$. 
For $i, j \in \Bbb{N}$, we let $[i,j]_{\Bbb{Z}}$ represent the set of
running indices $\{i, i+1, \ldots, j\}$. We use boldface notation to denote vectors.
In particular, we denote by $\mathbf{e}$ the vector of all ones
and by $\mathbf{e}_i$ the $i$-th standard basis vector.
Finally, we denote by $\SYM^n$ the set of all symmetric matrices in $\mathbb R^{n\times n}$.

\section{Random Service Durations} \label{sec:duration_model}

In this section, we study the W-DRAS model with random service durations. We describe the model, including the Wasserstein ambiguity set, in Section \ref{sec:dras-duation} and reformulate it as a copositive program in Section \ref{sec:dras-duation-cop}. In Section \ref{sec:dras-duration-convex}, we recast the model as tractable convex programs under a rectangularity condition.

\subsection{\ref{equ:wdro} model and Wasserstein ambiguity set} \label{sec:dras-duation}

We consider $n$ appointments and determine a time allowance $s_i$, or equivalently an arrival time, for each appointment $i \in [n]$. We require that all appointments are scheduled to arrive by a fixed end-of-the-day time limit $T$, which gives rise to the following feasible region for $\bm s := (s_1, \ldots, s_n)^{\top}$:
\[
{\cal S} := \left\{ \bm s \in \RR^n : \bm s \geq 0, \ \sum\limits_{i=1}^n s_i \leq T \right\}.
\]
As each appointment $i$ takes up a random service duration $u_i$, one or multiple of the following three scenarios can happen: (i) an appointment cannot start on time due to a delay of completion of the previous appointment, (ii) the server is idle and waiting for the next appointment due to an early completion of the current appointment, and (iii) the server cannot finish serving all the appointments by $T$. Specifically, if we let $w_i$ represent the waiting time of appointment $i$, $v_i$ represent the server's idleness after serving appointment $i$, and $w_{n+1}$ represent the server's overtime beyond $T$, then we can recursively express the waiting times, as well as the server's idleness and overtime by $w_1 = 0$ and
\begin{equation}\label{equ:recursions}
\begin{array}{l}
w_i = \max \Big\{ 0, \  u_{i-1} + w_{i-1} - s_{i-1} \Big\},  \ \ \
v_{i-1} = \max \Big\{ 0, \ s_{i-1} - u_{i-1} - w_{i-1} \Big\}
\end{array}
\end{equation}
for all $i \in [2, n+1]_{\mathbb{Z}}$. Note that $w_1 = 0$ because the first appointment always starts on time. To evaluate the performance of the appointment schedule $\bm s$, we denote the unit costs of waiting, idleness, and overtime by $\bm c \in \RR_+^n$, $\bm d \in \RR_+^n$, and $C \in \RR_+$ respectively. In addition, we assume that $d_{i+1} - d_i \leq c_{i+1}$ for all $i \in [n-1]$. This assumption is standard in the literature (see, e.g.,~\cite{Denton.Gupta.2003, Ge.Wan.Wang.Zhang.2013, Kong.Lee.Teo.Zheng.2013, Mak.Rong.Zhang.2015}). If an appointment system fails to satisfy this assumption, the server would (intentionally) keep idle and make appointee $i+1$ wait after completing all previous appointments, which is unrealistic in practice. Under this assumption, $\bm w := (w_1, \ldots, w_{n+1})^{\top}$ and $\bm v := (v_1, \ldots, v_n)^{\top}$ can be obtained from the following linear program:
\begin{equation} \label{equ:qx}
\begin{array}{rl}
f(\bm s, \bm u) :=  \min\limits_{\bm w, \bm v} & \sum \limits_{i=1}^n (c_i w_i + d_i v_i) + Cw_{n+1} \\
                                                    \st & w_i - v_{i-1} = u_{i-1} + w_{i-1} - s_{i-1} \ \ \ \forall \, i \in [2, n+1]_{\mathbb{Z}} \\
                                                        & \bm w \geq 0, \ w_1 = 0, \ \bm v \geq 0,
\end{array}
\end{equation}
where $\bm u := (u_1, \ldots, u_n)^{\top}$ and $f(\bm s, \bm u)$ represents the system total cost, which is evaluated by a weighted sum of the waiting times, idleness, and overtime under schedule $\bm s$ and service durations $\bm u$. If the probability distribution of $\bm{u}$, denoted by $\PP_{\bm{u}}$, is fully known, classical stochastic AS approaches seek a schedule that minimizes the expected system total cost, i.e.,
\begin{equation} \label{equ:sas}
Z^{\star} := \min \limits_{\bm s \in {\cal S}} \Ep_{\PP_{\bm{ u}}}\big [f(\bm s, \bm{u})\big].
\end{equation}
In this paper, we assume that $\PP_{\bm{u}}$ is unknown and belongs to a Wasserstein ambiguity set. Specifically, suppose that two probability distributions $\QQ_1$ and $\QQ_2$ are defined over a common support set $\U \subseteq \mathbb{R}^n$ and, with $p\geq 1$, $\|\cdot\|_p$ represents the $p$-norm on $\mathbb{R}^n$. Then, the Wasserstein distance $d_p(\QQ_1,\QQ_2)$ between $\QQ_1$ and $\QQ_2$ is the minimum transportation cost of moving from $\QQ_1$ to $\QQ_2$, under the premise that the cost of moving from $\bm u_1$ to $\bm u_2$ amounts to $\|\bm u_1 - \bm u_2\|_p$. Mathematically,
\[
d_p(\QQ_1,\QQ_2) := \left( \inf \limits_{\Pi \in {\cal P}(\QQ_1,\QQ_2)} \Ep_{\Pi} \Big[ \|\bm u_1 - \bm u_2\|_p^p\Big] \right)^{1/p}
\]
where random vectors $\bm u_1$ and $\bm u_2$ follow $\QQ_1$ and $\QQ_2$ respectively, and ${\cal P}(\QQ_1,\QQ_2)$ represents the set of all joint distributions of $(\bm u_1, \bm u_2)$ with marginals $\QQ_1$ and $\QQ_2$.

In addition, we assume that $\PP_{\bm u}$ is only observed via a finite set of $N$ \iid~samples, denoted as $\{\widehat{\bm{u}}^1, \ldots, \widehat{\bm{u}}^N\}$. For example, these samples can come from the historical data of the service durations $\bm u$. Then, we consider the following $p$-Wasserstein ambiguity set
\[
\D_p(\widehat \PP^N_{\bm u}, \, \epsilon) := \left\{ \QQ_{\bm u} \in \mathcal{P}(\mathcal{U}) : \  d_p(\QQ_{\bm u}, \, \widehat \PP^N_{\bm u}) \le \epsilon \right\},
\]
where $\cal P(\U)$ represents the set of all probability distributions on $\U$, $\widehat \PP^N_{\bm u}$ represents the empirical distribution of $\bm u$ based on the $N$ \iid~samples, i.e., $\widehat \PP^N_{\bm u} = \frac1N \sum_{j=1}^N \delta_{\bm{\widehat u}^j}$, and $\epsilon > 0$ represents the radius of the ambiguity set. We seek a schedule that minimizes the expected total cost with regard to the worst-case distribution in the $p$-Wasserstein ambiguity set, i.e., we solve the following DRO problem:
\begin{equation} \tag{W-DRAS} \label{equ:wdro}
\widehat Z (N, \,\epsilon) := \min_{\bm s \in \cal S} \sup_{\QQ_{\bm u} \in \D_p(\widehat \PP^N_{\bm u}, \, \epsilon)} \Ep_{\QQ_{\bm u}} \big[f(\bm s, \, \bm u) \big].
\end{equation}
Many data-driven applications desire asymptotic consistency. In particular, an appointment scheduler may desire that, as the sample size $N$ increases to infinity, $\widehat Z(N, \,\epsilon)$ tends to $Z^{\star}$, which is the optimal value of the ``true'' model \eqref{equ:sas}, i.e., the model with perfect knowledge of $\mathbb{P}_{\bm u}$. Accordingly, an optimal appointment schedule obtained from~\eqref{equ:wdro} tends to the optimal schedule obtained from \eqref{equ:sas}. In addition, if $\widehat Z(N, \,\epsilon) \geq Z^{\star}$ almost surely, then \eqref{equ:wdro} provides a safe (upper bound) guarantee on the expected total cost with any \emph{finite} data size $N$. We close this section by formally establishing the asymptotic consistency and the finite-data guarantee of \eqref{equ:wdro}. To this end, we make the following assumption on the support set $\mathcal{U}$.
\begin{assumption} \label{ass:u_set}
The support set $\U$ is nonempty, compact, and convex.
\end{assumption}
This assumption is mild and it holds in most practical situations. The following concentration inequality on the $p$-Wasserstein ambiguity set is adapted from~\cite{fournier2015rate}.
\begin{lemma}[Adapted from Theorem 2 in~\cite{fournier2015rate}] \label{lem:concentration}
Suppose that Assumption \ref{ass:u_set} holds. Then there exist nonnegative constants $c_1$ and $c_2$ such that, for all $N \geq 1$ and $\beta \in (0, \min\{1, c_1\})$,
$$
\mathbb{P}^N_{\bm u}\left\{d_p(\mathbb{P}_{\bm u}, \widehat \PP^N_{\bm u}) \leq \epsilon_N(\beta)\right\} \geq 1 - \beta,
$$
where $\mathbb{P}^N_{\bm u}$ represents the product measure of $N$ copies of $\mathbb{P}_{\bm u}$ and
$$
\epsilon_N(\beta) \ = \ \left[\frac{\log(c_1 \beta^{-1})}{c_2 N}\right]^{\frac{1}{\max\{n, 3p\}}}.
$$
\end{lemma}
Lemma \ref{lem:concentration} assures that the $p$-Wasserstein ambiguity set incorporates the true distribution $\PP_{\bm u}$ with high confidence. This leads to the following two theorems, whose proofs rely on~\cite{Esfahani.Kuhn.2017} and are provided in the Appendix for completeness.
\begin{theorem}[Asymptotic consistency, adapted from Theorem 3.6 in~\cite{Esfahani.Kuhn.2017}] \label{thm:a-c}
Suppose that Assumption \ref{ass:u_set} holds. Consider a sequence of confidence levels $\{\beta_N\}_{N \in \mathbb{N}}$ such that $\sum_{N=1}^{\infty} \beta_N < \infty$ and $\lim_{N \rightarrow \infty} \epsilon_N(\beta_N) = 0$\footnote{For example, we can set $\beta_N = \exp\{-\sqrt{N}\}$.}, and let $\widehat{\bm s}(N, \,\epsilon_N(\beta_N))$ represent an optimal solution to \eqref{equ:wdro} with the ambiguity set $\D_p(\widehat \PP^N_{\bm u}, \, \epsilon_N(\beta_N))$. Then, $\mathbb{P}^{\infty}_{\bm u}$-almost surely we have $\widehat Z(N, \,\epsilon_N(\beta_N)) \rightarrow Z^{\star}$ as $N\rightarrow \infty$. In addition, any accumulation point of $\{\widehat{\bm s}(N, \,\epsilon_N(\beta_N))\}_{N \in \mathbb{N}}$ is an optimal solution of \eqref{equ:sas} $\mathbb{P}^{\infty}_{\bm u}$-almost surely.
\end{theorem}
\begin{theorem}[Finite-data guarantee, adapted from Theorem 3.5 in~\cite{Esfahani.Kuhn.2017}] \label{thm:f-d}
For any $\beta \in (0, 1)$, let $\widehat{\bm s}(N, \,\epsilon_N(\beta))$ represent an optimal solution of~\eqref{equ:wdro} with the ambiguity set $\D_p(\widehat \PP^N_{\bm u}, \, \epsilon_N(\beta))$. Then, $\mathbb{P}^N_{\bm u}\{\mathbb{E}_{\mathbb{P}_{\bm u}}[f(\widehat{\bm s}(N, \,\epsilon_N(\beta)), \, \bm u)] \leq \widehat Z(N, \,\epsilon_N(\beta))\} \geq 1 - \beta$.
\end{theorem}
The radius $\epsilon_N(\beta)$ suggested in the above theoretical results is often conservative, i.e., the actual radius $\epsilon$ that achieves the $(1 - \beta)$ confidence of $\D_p(\widehat \PP^N_{\bm u}, \, \epsilon)$ may be much smaller than $\epsilon_N(\beta)$. In this paper, we calibrate the radius of the $p$-Wasserstein ambiguity set by using the cross validation method. This yields encouraging convergence results and out-of-sample performance of \eqref{equ:wdro} (see the detail in Section \ref{sec:numerical}).

Problem \eqref{equ:wdro} is generically difficult to solve as it involves infinitely many probability distributions. In the following two subsections, we shall recast \eqref{equ:wdro} as a (deterministic and finite) copositive program and identify conditions under which we have more tractable reformulations.

\subsection{Copositive programming reformulations} \label{sec:dras-duation-cop}

In this section, we propose a copositive reformulation for~\eqref{equ:wdro}. As the first step, we represent $f(\bm s, \bm u)$ by the following dual problem of \eqref{equ:qx}:
\begin{equation} \label{equ:qx-max}
\begin{array}{rll}
f(\bm s, \bm u) = \displaystyle\max_{y \in \Y} & \sum \limits_{i=1}^n (u_i - s_i) y_i,
\end{array}
\end{equation}
where dual variables $\bm y$ are associated with the first constraint in \eqref{equ:qx} and the polyhedral feasible set $\Y$ is described as
\begin{equation}
\Y := \left\{ \bm y \in \RR^n :
\begin{array}{ll}
-y_i \leq d_i &   \ \ \ \forall \, i \in [n]   \\
y_{i-1} - y_i \leq c_i &   \ \ \ \forall \, i \in [2,n]_{\Bbb Z} \\
y_n \leq C &
\end{array}
\right\}. \label{equ:y-def}
\end{equation}
The strong duality between \eqref{equ:qx} and \eqref{equ:qx-max} holds because $\Y$ is nonempty and compact, as summarized in the following lemma.
\begin{lemma} \label{lem:y_set}
$\Y$ is nonempty and compact.
\end{lemma}

Second, to obtain a copositive reformulation, we consider the inner maximization problem of \eqref{equ:wdro}
\begin{equation}
\sup_{\QQ_{\bm u} \in \D_p(\widehat \PP_{\bm u}^N, \, \epsilon)} \Ep_{\QQ_{\bm u}} \big[ f(\bm s, \bm u)\big] \label{equ:bound}
\end{equation}
for fixed $\bm s \in \cal S$. In the following proposition, we present an equivalent dual formulation of~\eqref{equ:bound} via duality theory.
\begin{proposition} \label{prop:ROFormat}
The optimal value of formulation \eqref{equ:bound} equals
that of the following formulation:
\begin{align}
\inf_{\rho, \, \bm{\theta}} \ & \ \epsilon^p \rho + \dfrac1N \sum\limits_{j=1}^{N} \theta_j \nonumber \\
\st \ & \ \sup \limits_{\bm u \in \mathcal{U}} \ \left\{f(\bm s, \, \bm u) - \rho \|\bm u - \bm{\widehat u}^j\|_p^p\right\} \le \theta_j \ \ \ \forall \, j \in [N] \label{equ:vdp_nonconvex} \\
& \ \rho \ge 0, \ \bm \theta \in \RR^N. \nonumber
\end{align}
\end{proposition}
While this dual formulation is in the format of a classical robust optimization problem, it is
inherently difficult. Indeed, reformulating the semi-infinite constraint~\eqref{equ:vdp_nonconvex}
entails solving $N$ non-convex optimization problems (note that both $f(\bm s, \bm u)$ and $\|\bm u - \bm{\widehat u}^j\|_p^p$ are convex in $\bm u$), which are in general computationally intractable.
The presence of the norm $\|\cdot\|_p^p$ with a general $p \geq 1$ introduces additional computational difficulties.

In this section, we propose to use copositive programming reformulation techniques to tackle the intractable constraint for the cases of $p=1,2$.
For ease of exposition, we define sets ${\cal F}^j_1$ for all $j \in [N]$ and ${\cal F}_2$ by using $\U$ and $\Y$:
\[
\mathcal{F}^j_1 := \left \{
\begin{array}{l}
(\bm{u^+}, \bm{u^-}, \bm{y}) \in \RR_+^n \times \RR_+^n \times \RR^n : \bm{u^+} - \bm{u^-} +  \bm{\widehat u}^j \in \U, \, \bm{y} \in \Y
\end{array}
\right \},
\]
\[
\mathcal{F}_2 := \left\{ (\bm u, \bm y) \in \RR^n \times \RR^n :
\begin{array}{l}
 \bm u \in {\cal U}, \ \bm y \in {\cal Y}
\end{array}
 \right\},
\]
and then construct their perspective sets $\K_1^j$ and $\K_2$ respectively:
\begin{equation}
\K_1^j  :=  \text{closure} \Big( \Big \{ (t, \bm u^+, \bm u^-, \bm y) : (\bm u^+/t, \bm u^-/t, \bm{y} /t) \in \F_1^j, \ t > 0 \Big \} \Big ),
\end{equation}
\begin{equation}
\K_2  :=  \textup{closure} \Big ( \Big\{ (t, \bm u, \bm y)  :   (\bm u/t, \bm y/t)  \in \mathcal{F}_2, \ t > 0 \Big\} \Big ),
\end{equation}
where $\textup{closure}(\K)$ denotes the closure of the set $\K$.
In fact, $\K_1^j$ and $\K_2$ are closed and convex cones by Assumption~\ref{ass:u_set}.
Furthermore, for a closed and convex cone $\K$, we define $\COP(\K)$ as
the set of all copositive matrices with respect to $\K$, i.e.,
$\COP(\K) := \left\{  \bm{M} \in \SYM^n : \bm{x}^\top \bm{M}\bm{x} \geq 0  \  \forall \, \bm{x} \in \K \right\}$. The dual cone of $\COP(\K)$ is denoted by $\CP(\K)$, which is the set of all completely positive matrices with respect to $\K$.
We refer the reader to~\cite{Burer.2009, Burer.2011, Burer.2015, Dur.2010}
for a thorough discussion about copositive programming.

Now, we are ready to present the copositive programming reformulation of~\eqref{equ:wdro} for the cases of $p = 1,2$ in the definition of $\D_p(\widehat \PP^N_{\bm u}, \epsilon)$.
For fixed $\rho$ and $\bm s$, define \[
\bm{H}_j^1(\rho, \, \bm s) :=
\begin{bmatrix}
0 & -\frac{1}{2}\rho \bm{e}^\top & -\frac{1}{2}\rho \bm{e}^\top & \frac{1}{2}(\bm{\widehat u}^j - \bm s)^\top  \\
 -\frac{1}{2}\rho \bm{e} & 0 & 0 & \frac{1}{2} \bm{I} \\
 -\frac{1}{2}\rho \bm{e} & 0 & 0 & - \frac{1}{2} \bm{I} \\
\frac{1}{2}(\bm{\widehat u}^j - \bm s) & \frac{1}{2} \bm{I} & -\frac{1}{2} \bm{I} & 0
\end{bmatrix} \ \ \ \forall \, j \in [N],
\]
\[
\bm H_j^2(\rho, \, \bm s) :=
\begin{bmatrix}
-\rho \|\bm{\widehat u}^j\|^2 & \rho (\bm{\widehat u}^j)^\top & - \frac12 \bm s^\top \\
\rho \bm{\widehat u}^j & -\rho \bm I & \frac12\bm I \\
-\frac12 \bm s & \frac12 \bm I & 0
\end{bmatrix}  \ \ \ \forall \, j \in [N],
\]
where $\bm{I} \in \SYM^n$ and $\bm{e} \in \RR^n$ denote by the identity matrix and all-ones vector respectively.
\begin{theorem} \label{thm:dro-cop}
Suppose that Assumption~\ref{ass:u_set} holds. Then, when $p=1$, \eqref{equ:wdro} yields the same optimal value and the same set of optimal solutions as the following linear conic program:
\begin{equation}\label{equ:dro-cop-p1}
\begin{array}{rl}
\widehat Z(N, \, \epsilon) =  \inf  & \epsilon \rho + \dfrac1N \sum\limits_{j=1}^{N} \beta_j  \\
\st &  \rho \in \RR, \ \bm \beta \in \RR^N \\
& \beta_j \, \bm{e}_1 \bm{e}_1^\top - \bm{H}_j^1(\rho,\, \bm s) \in \COP(\K_1^j)  \ \ \ \forall \, j \in [N] \\
& \rho \ge 0, \ \bm s \in {\cal S},
\end{array}
\end{equation}
where $\bm e_1$ represents the first standard basis vector in $\RR^{3n+1}$.

In addition, when $p=2$, \eqref{equ:wdro} yields the same optimal value and the same set of optimal solutions as the following linear conic program:
\begin{equation}\label{equ:dro-cop-p2}
\begin{array}{rl}
\widehat Z(N, \, \epsilon) =  \inf  & \epsilon^2 \rho + \dfrac1N \sum\limits_{j=1}^{N} \beta_j \\
\st &  \rho \in \RR, \ \bm \beta \in \RR^N \\
& \beta_j \, \bm{e}_1 \bm{e}_1^\top - \bm{H}^2_j(\rho,\, \bm s) \in \COP(\K_2)  \ \ \ \forall \, j \in [N] \\
& \rho \ge 0, \ \bm s \in {\cal S},
\end{array}
\end{equation}
where $\bm e_1$ represents the first standard basis in $\RR^{2n+1}$.
\end{theorem}

\begin{remark}
The copositive programming reformulations in \eqref{equ:dro-cop-p1} and \eqref{equ:dro-cop-p2} are amenable to semidefinite programming solution schemes. Specifically, there exists a hierarchy of increasingly tight semidefinite-based inner approximations that converge to $\COP(\K)$ for a general closed and convex cone $\K$~\cite{Bomze.Klerk.2002, Klerk.Pasechnik.2002, Lasserre.2009, Parrilo.2000, Xu.Hanasusanto.2018}. Replacing the cone $\COP(\K)$ with these inner approximations gives rise to conservative semidefinite programs that can be solved using standard off-the-shelf solvers (such as MOSEK~\cite{Mosek} and SDPT3~\cite{toh1999sdpt3}).
\end{remark}

\subsection{Tractable reformulations} \label{sec:dras-duration-convex}

In this subsection, we consider a setting that admits tractable reformulations of~\eqref{equ:wdro} for general
$p \geq 1$. In particular, we recast problem~\eqref{equ:wdro} as a linear program when $p=1$ and as a second-order cone program when $p > 1$ and $p$ is rational. To this end, we make the following assumption on the support set $\mathcal{U}$.
\begin{assumption}[Rectangular support of service durations] \label{ass:tractable-support}
Assume that $\mathcal{U}$ is defined as
\[
\mathcal{U} := \Big\{ \bm u \in \RR^n : \bm u^{\mbox{\tiny L}} \leq \bm u \leq  \bm u^{\mbox{\tiny U}}  \Big\},
\]
for $0 \leq u^{\mbox{\tiny L}}_i  <   u^{\mbox{\tiny U}}_i < \infty$ for all $i \in [n]$.
\end{assumption}
Assumption \ref{ass:tractable-support} is mild because any compact $\mathcal{U} \subset \mathbb{R}^n_+$ can be relaxed to be rectangular. In addition, the asymptotic consistency and the finite-data guarantee of \eqref{equ:wdro} (i.e., Theorems \ref{thm:a-c}--\ref{thm:f-d}) hold valid under Assumption \ref{ass:tractable-support}. We note that rectangular $\mathcal{U}$ does not imply that $\{u_i\}_{i\in[n]}$ are probabilistically independent. First, following Proposition \ref{prop:ROFormat}, we rewrite~\eqref{equ:wdro} as
\begin{equation} \label{equ:tractable-ref}
\begin{array}{lll}
\widehat Z(N, \, \epsilon) = & \displaystyle \inf_{\rho, \, \bm s} & \epsilon^p \rho  + \dfrac1N \sum\limits_{j=1}^{N} \sup \limits_{\bm u \in \mathcal{U}} \left\{f( \bm s, \, \bm u) - \rho \|\bm u - \bm{\widehat u}^j\|^p_p\right\} \\
                & \st & \rho \ge 0, \, \bm s \in {\cal S}.
\end{array}
\end{equation}
Recall that formulation \eqref{equ:tractable-ref} is potentially prohibitive to compute because it entails solving $N$ non-convex optimization problems, in which $f( \bm s, \, \bm u) - \rho \|\bm u - \bm{\widehat u}^j\|^p_p$ is neither convex nor concave in $\bm u$. Fortunately, Assumption \ref{ass:tractable-support} enables us to recast these problems as linear programs for fixed $\rho$ and $\bm s$, as summarized in the following proposition.
\begin{proposition} \label{prop:tractable}
Suppose that Assumption~\ref{ass:tractable-support} holds. Given $p \geq 1$, $\rho \geq 0$, and $\bm s \in {\cal S}$, we denote
\[
\omega_j(\rho, \bm s) := \sup \limits_{\bm u \in \mathcal{U}} \Big \{ f( \bm s, \, \bm u) - \rho \|\bm u -\bm{\widehat u}^j\|^p_p\Big \}.
\]
Then, we have
\begin{subequations}
\begin{align}
\omega_j(\rho, \bm s) \ = \ \max_t \ & \ \sum_{k=1}^{n} \sum_{\ell=k}^{n+1} \left( \sum_{i=k}^{\min\{\ell, n\}} z_{i\ell j} \right) t_{k\ell} \label{equ:lp-obj} \\
\mbox{s.t.} \ & \ \sum_{k=1}^{i} \sum_{\ell=i}^{n+1} t_{k\ell} = 1 \ \ \ \forall \, i \in [n] \label{equ:lp-con} \\
                  \ & \ t_{k\ell} \geq 0 \ \ \ \forall \, k \in [n], \ \forall \, \ell \in [k, n+1]_{\Bbb Z}. \label{equ:lp-con-nonnegative}
\end{align}
In formulation \eqref{equ:lp-obj}--\eqref{equ:lp-con-nonnegative}, $z_{i\ell j}$ is defined as $z_{i\ell j} = - s_i\pi_{i\ell} + \sup_{u^{\mbox{\tiny L}}_i \leq u_i \leq u^{\mbox{\tiny U}}_i} \{\pi_{i\ell} u_i - \rho |u_i - \widehat{u}^j_i|^p \}$ for all $j \in [N]$, $i \in [n]$, and $\ell \in [i, n+1]_{\mathbb{Z}}$, where
\begin{equation}
\pi_{i\ell} := \left\{\begin{array}{ll} -d_{\ell} + \sum_{q=i+1}^{\ell} c_q & \mbox{if $i \in [n]$ and $\ell \in [i, n]_{\mathbb{Z}}$} \\
C + \sum_{q=i+1}^{n} c_q & \mbox{if $i \in [n]$ and $\ell = n+1$.}
\end{array}\right. \label{equ:h-pi} \\[0.25cm]
\end{equation}
\end{subequations}
\end{proposition}

Proposition \ref{prop:tractable} contrasts with the general computational tractability of two-stage DRO with right-hand side uncertainty (note that random variables $\bm u$ appear in the right-hand side of the linear program \eqref{equ:qx} that defines \eqref{equ:wdro}). In general, evaluating the objective function value of \eqref{equ:wdro} with a \emph{fixed} appointment schedule $\bm s$ entails solving an exponential-size convex program (see Remark 5.5 in~\cite{Esfahani.Kuhn.2017}). In contrast, Proposition \ref{prop:tractable} indicates that~\eqref{equ:wdro} can be solved in polynomial time to find an \emph{optimal} $\bm s$. Indeed,~\eqref{equ:wdro} is a convex program because the function $\omega_j(\rho, \bm s)$ is jointly convex in $\rho$ and $\bm s$. Additionally, Proposition \ref{prop:tractable} implies that the epigraph of $\omega_j(\rho, \bm s)$, denoted as $E := \{(\theta, \rho, \bm s): \theta \geq \omega_j(\rho, \bm s)\}$, can be separated in polynomial time. That is, given a point $(\bar{\theta}, \bar{\rho}, \bm{\bar{s}}) \in \mathbb{R}^{n+2}$, one can either verify that $(\bar{\theta}, \bar{\rho}, \bm{\bar{s}}) \in E$ or generate a hyperplane that separates $(\bar{\theta}, \bar{\rho}, \bm{\bar{s}})$ from $E$ in polynomial time. Then, following the equivalence between separation and convex optimization established in the seminal work~\cite{grotschel1981ellipsoid}, we conclude that~\eqref{equ:wdro} can be solved in polynomial time. This contrast in computational tractability arises because we study a particular DRO model in appointment scheduling and we take advantage of the structure of~\eqref{equ:wdro}.

We proceed to consider special $p$ values. In particular, the following theorem recasts~\eqref{equ:wdro} as a linear program and a second-order cone program for $p = 1$ and $p=2$, respectively. In both cases, the solution of~\eqref{equ:wdro} does not reply on specialized algorithms (such as separation) and can be obtained via off-the-shelf solvers (such as Gurobi and CPLEX).
\begin{theorem} \label{thm:tractable}
Suppose that Assumption~\ref{ass:tractable-support} holds.
Then, when $p = 1$, \eqref{equ:wdro} yields the same optimal value and the same
set of optimal solutions as the following linear program:
\begin{equation} \label{equ:lpdro}
\begin{array} {lll}
\displaystyle\min_{\rho, \bm{s}, \bm{\gamma}, \bm{z}} \ & \ \epsilon \rho + \dfrac{1}{N} \sum\limits_{j=1}^N \sum\limits_{i=1}^{n} \gamma_{ij} & \\
\mbox{s.t.} \ & \ \sum\limits_{k=i}^{\min\{\ell, \, n\}} \gamma_{kj} \geq \sum\limits_{k=i}^{\min\{\ell, \, n\}} z_{k\ell j} &   \forall \, i \in [n], \ \forall \, \ell \in [i, n+1]_{\Bbb Z}, \ \forall \, j \in [N] \\
& \ z_{i\ell j} + \pi_{i\ell} s_i + |u^{\mbox{\tiny L}}_i - \widehat{u}^j_i| \rho \geq \pi_{i\ell} u^{\mbox{\tiny L}}_i &  \forall  \,  i \in [n],  \ \forall \,  \ell \in [i, n+1]_{\Bbb Z}, \  \forall \, j \in [N] \\
& \ z_{i\ell j} + \pi_{i\ell} s_i \geq \pi_{i\ell} \widehat{u}^j_i  & \forall  \, i \in [n], \ \forall \,  \ell \in [i, n+1]_{\Bbb Z}, \  \forall \, j \in [N] \\
& \ z_{i\ell j} + \pi_{i\ell} s_i + |u^{\mbox{\tiny U}}_i - \widehat{u}^j_i| \rho \geq \pi_{i\ell} u^{\mbox{\tiny U}}_i &  \forall  \, i \in [n], \ \forall \,  \ell \in [i, n+1]_{\Bbb Z}, \ \forall \,  j \in [N] \\
& \ \rho \geq 0, \ \bm s \in {\cal S}.
\end{array}
\end{equation}
In addition, when $p = 2$, \eqref{equ:wdro} yields the same optimal value and the same set of optimal solutions as the following second-order cone program:
\begin{equation} \label{equ:socpdro}
\begin{array} {lll}
\displaystyle\min_{\rho, \bm{s}, \bm{\gamma}, \bm{z}, \bm{\beta}, \bm{r}} \ & \ \epsilon^2 \rho + \dfrac{1}{N} \sum\limits_{j=1}^N \sum\limits_{i=1}^{n} \gamma_{ij} & \\
\mbox{s.t.} \ & \ \sum\limits_{k=i}^{\min\{\ell, \, n\}} \gamma_{kj} \geq \sum\limits_{k=i}^{\min\{\ell, \, n\}} z_{k\ell j} &   \forall \, i \in [n], \ \forall \, \ell \in [i, n+1]_{\Bbb Z}, \ \forall \, j \in [N] \\
& \ z_{i\ell j} + \pi_{i\ell} s_i - (u^{\mbox{\tiny U}}_i - u^{\mbox{\tiny L}}_i) \beta^{\mbox{\tiny U}}_{i\ell j} & \\
& \ - (\widehat{u}^j_i - u^{\mbox{\tiny L}}_i) \beta^{\mbox{\tiny L}}_{i\ell j} - r_{i\ell j} \geq \pi_{i\ell} \widehat{u}^j_i & \forall  \,  i \in [n],  \ \forall \,  \ell \in [i, n+1]_{\Bbb Z}, \  \forall \, j \in [N] \\[0.25cm]
& \ \Biggl\|\begin{bmatrix} \pi_{i\ell} + \beta^{\mbox{\tiny L}}_{i\ell j} - \beta^{\mbox{\tiny U}}_{i\ell j} \\[0.1cm] r_{i\ell j} - \rho \end{bmatrix}\Biggr\|_2 \leq r_{i\ell j} + \rho  & \forall  \, i \in [n], \ \forall \,  \ell \in [i, n+1]_{\Bbb Z}, \  \forall \, j \in [N] \\[0.25cm]
& \ \rho \geq 0, \ \bm s \in {\cal S}. &
\end{array}
\end{equation}
\end{theorem}
We note that both reformulations \eqref{equ:lpdro} and \eqref{equ:socpdro} involve $\mathcal{O}(Nn^2)$ decision variables and $\mathcal{O}(Nn^2)$ constraints. More generally, we show that~\eqref{equ:wdro} admits a second-order conic reformulation for all $p > 1$, as long as $p$ is rational. But as $p$ typically takes integer values (e.g., $p = 1, 2$) in real-world applications, we relegate this more general result to Theorem~\ref{thm:p_rational} in the Appendix.

In addition to tractable computation of an optimal appointment schedule, Theorem \ref{thm:tractable} suggests an approach to stress testing \emph{any} appointment schedule $\bar{s} \in \mathcal{S}$. Specifically, the following theorem derives a worst-case probability distribution $\mathbb{Q}^\star_{\bm u}$ of the random service durations $\bm u$ that attains $\sup_{\mathbb{Q}_{\bm u} \in \mathcal{D}_1(\widehat{\mathbb{P}}^N_{\bm u}, \epsilon)} \mathbb{E}_{\mathbb{Q}_{\bm u}}[f(\bar{\bm{s}}, \bm u)]$. This distribution can be applied, for example, to assess the quality of an appointment schedule generated by any decision-making processes.\footnote{Although the derivation of worst-case distributions in Theorem~\ref{thm:wc_dist} is based on the formulation when $p=1$, similar conclusions can be obtained for the case when $p>1$ and $p$ is rational.}
\begin{theorem} \label{thm:wc_dist}
For fixed $\bar{\bm s} \in \mathcal{S}$ and $\epsilon \geq 0$, $\displaystyle\sup_{\mathbb{Q}_{\bm u} \in \mathcal{D}_1(\widehat{\mathbb{P}}^N_{\bm u}, \epsilon)} \mathbb{E}_{\mathbb{Q}_{\bm u}}[f(\bar{\bm s}, \bm u)]$ equals the optimal value of the following linear program:
\begin{equation} \label{equ:lpdro-x-wc}
\begin{array}{ll}
\max\limits_{\bm p, \bm q, \bm r} & \dfrac1N  \sum\limits_{j=1}^N \sum\limits_{i=1}^{n}\sum\limits_{\ell=i}^{n+1}  \pi_{i\ell}  \left [ (u_i^{\tiny L} - \widehat u^j_i) q_{i\ell j} + (u_i^{\tiny U} - \widehat u^j_i)r_{i\ell j} + \left(\sum\limits_{k=1}^i p_{k\ell j}\right)(\widehat u^j_i - \bar{s}_i) \right] \\
\st & \dfrac1N  \sum\limits_{j=1}^N \sum\limits_{i=1}^{n}\sum\limits_{\ell=i}^{n+1}  \left[  (\widehat u^j_i - u_i^{\tiny L}) q_{i\ell j} + (u_i^{\tiny U} - \widehat u^j_i) r_{i\ell j}\right] \leq \epsilon \\
 & \sum \limits_{\ell = i}^{n+1} \sum\limits_{k=1}^i p_{k\ell j} = 1  \ \ \ \forall \,  i \in [n], \  \forall \,  j \in  [N] \\
 &  \sum\limits_{k=1}^i p_{k\ell j}  - q_{i\ell j} - r_{i\ell j} \geq 0\ \ \  \forall \, i \in [n], \  \forall \,  \ell \in [i, n+1]_{\Bbb Z}, \ \forall \,   j \in  [N]  \\
 & p_{i\ell j} \geq 0, \ q_{i\ell j} \geq 0, \ r_{i\ell j} \geq 0 \ \  \  \forall \,  i \in [n], \  \forall \,  \ell \in [i, n+1]_{\Bbb Z}, \ \forall \,   j \in  [N].
\end{array}
\end{equation}
Let $\{p^\star_{k\ell j}, q^\star_{i\ell j}, r^\star_{i\ell j}\}$ be an optimal solution of \eqref{equ:lpdro-x-wc} and define
\[
\mathcal{T} = \left\{\bm{t}\in \{0, 1\}^{(n+1)(n+2)/2}: \ \sum_{k=1}^i \sum_{\ell=i}^{n+1} t_{k\ell} = 1, \ \forall i \in [n+1] \right\}.
\]
Then, there exists a distribution $\mathbb{P}^j_{\bm t}$ on $\mathcal{T}$ such that $\mathbb{P}^j_{\bm t}\{t_{k\ell} = 1\} = p^\star_{k\ell j}$ for all $j \in [N]$, $k \in [n+1]$, and $\ell \in [k, n+1]_{\mathbb{Z}}$. Furthermore, define the probability distribution
$$
\mathbb{Q}^\star_{\bm u} = \frac{1}{N}\sum_{j=1}^N \sum_{\bm{\tau} \in \mathcal{T}} \mathbb{P}^j_{\bm t}\{\bm t = \bm \tau\} \delta_{\bm{u}^j(\bm \tau)},
$$
where, for all $i \in [n]$ and $j \in [N]$,
$$
u_i^j(\bm \tau) = \sum_{\ell=i}^{n+1} \left( \sum_{k=1}^i \tau_{k\ell} \right) u_{i\ell j} \ \ \mbox{and} \ \ u_{i\ell j} = \widehat u^j_i + \frac{q^\star_{i\ell j}(u_i^{\tiny L} - \widehat u^j_i )} {\sum_{k=1}^i p^\star_{k\ell j}} + \frac{r^\star_{i\ell j}(u_i^{\tiny U} - \widehat u^j_i )} {\sum_{k=1}^ip^\star_{k\ell j}},
$$
where we adopt an extended arithmetic given by $0/0 = 0$. Then, $\mathbb{Q}^\star_{\bm u}$ belongs to the Wasserstein ambiguity set $\mathcal{D}_1(\widehat{\mathbb{P}}^N_{\bm u}, \epsilon)$ and $\mathbb{E}_{\mathbb{Q}^\star_{\bm u}}[f(\bar{\bm s}, \bm u)] = \displaystyle\sup_{\mathbb{Q}_{\bm u} \in \mathcal{D}_1(\widehat{\mathbb{P}}^N_{\bm u}, \epsilon)} \mathbb{E}_{\mathbb{Q}_{\bm u}}[f(\bar{\bm s}, \bm u)]$.
\end{theorem}
\begin{remark}
Intuitively, the \eqref{equ:wdro} model can be viewed as a two-person game between the AS scheduler and the nature, who picks a $\mathbb{Q}^\star_{\bm u}$ that maximizes the expected total cost after the schedule $\bm s$ is determined. Theorem \ref{thm:wc_dist} gives us a clear picture on how the nature makes its pick. Specifically, $\mathbb{Q}^\star_{\bm u}$ is a mixture of $N$ distributions, with each pertaining to a data sample. That is, $\mathbb{Q}^\star_{\bm u} = (1/N)\sum_{j=1}^N \mathbb{Q}^j_{\bm u}$, where $\mathbb{Q}^j_{\bm u} := \sum_{\bm{\tau} \in \mathcal{T}} \mathbb{P}^j_{\bm t}\{\bm t = \bm \tau\} \delta_{\bm{u}^j(\bm \tau)}$ for all $j \in [N]$. Note that $\mathcal{T}$ is the collection of all possible partitions of the set $[n+1]$ into sub-intervals. Hence, what the nature does under $\mathbb{Q}^j_{\bm u}$ is to: (i) randomly pick a partition $\bm \tau \in \mathcal{T}$ following distribution $\mathbb{P}^j_{\bm t}$ and (ii) for each $i \in [n]$, if $i$ belongs to the sub-interval $[k, \ell]_{\mathbb{Z}}$ (i.e., if $\tau_{k\ell} = 1$ with $k \leq i \leq \ell$) then the $i^{\mbox{\tiny th}}$ appointment has a service duration $u_{i\ell j}$.
\end{remark}
\begin{remark}
In the situation where the radius of the Wasserstein ball is set to zero, the worst-case distribution $\mathbb{Q}^\star_{\bm u}$ reduces to the empirical distribution. Indeed, setting $\epsilon =0$ enforces all $q_{i\ell j}$ and $r_{i \ell j}$ to be zero in \eqref{equ:lpdro-x-wc}. It follows that $u_{i\ell j} = \widehat{u}^j_i$ and so $u^j_i(\bm \tau) = \widehat{u}^j_i$ for all $\bm \tau \in \mathcal{T}$. Therefore, we have $\mathbb{Q}^\star_{\bm u} = \frac{1}{N}\sum_{j=1}^N \delta_{\widehat{\bm u}^j}$.
\end{remark}

\section{Random No-Shows and Service Durations} \label{sec:no_show}

In many AS systems, appointments have random no-shows, i.e., the appointee cancels her appointment too late such that the scheduler
cannot make a substitute. In such a situation, the approach studied in
Section~\ref{sec:duration_model} is no longer applicable. In Section \ref{sec:w-ns}, we extend \eqref{equ:wdro} and the Wasserstein ambiguity set to incorporate both random no-shows and service durations. We reformulate this extended model as copositive programs in Section \ref{sec:ns-cop} and as tractable convex programs under a mild condition in Section \ref{sec:ns-tractable}.

\subsection{Extended model for random no-shows} \label{sec:w-ns}
We model the random no-show of appointment $i$ by using a Bernoulli random variable $\lambda_i$ such that $\lambda_i=1$ if appointee $i$ shows up and $\lambda_i=0$ otherwise. We denote the support of $\bm \lambda$ by $\Lambda \subseteq \{0, 1\}^n$. If $\Lambda = \{0, 1\}^n$ then it includes all possible scenarios of no-shows. Unfortunately, it has been observed (e.g., in~\cite{Jiang.Shen.Zhang.2017}) that such a $\Lambda$ often results in poor out-of-sample performance. Intuitively, this is because the set $\{0, 1\}^n$ includes many unlikely scenarios (e.g., a majority of appointees do not show up), rendering the resulting schedule over-conservative. In this paper, we propose to consider a less conservative, budget-constrained support set
\[
\Lambda := \left\{ \bm \lambda \in \{0,1\}^n : \sum\limits_{i=1}^n (1-\lambda_i) \leq K \right\},
\]
where the integer parameter $K \in [n]$ denotes the ``budget'' of no-shows and controls the conservativeness of $\Lambda$. Intuitively, $\Lambda$ contains only scenarios with no more than $K$ no-shows out of the $n$ appointments. For example, if $K=0$ then all $n$ appointees show up for their appointments, yielding the least conservative support set; and if $K= n$ then $\Lambda = \{0, 1\}^n$, yielding the most conservative case. The parameter $K$ can be determined based on the scheduler's knowledge, risk attitude, and/or the historical data of no-shows (see~\eqref{equ:setK} below).

Observing that an AS system spends no time on a no-show appointment, we let $\bm \mu := (\mu_1, \ldots, \mu_n)^\top$ represent the \emph{actual} service durations of appointments and, for ease of exposition, let $\bm \xi := (\bm{\mu}^\top, \bm{\lambda}^\top)^\top$. Then, under a rectangularity condition, we define the support set of $\bm \xi$ as
\[
\Xi := \Big\{ (\bm \mu, \bm \lambda) \in \RR^n \times \Lambda:
\ u_i^{\mbox{\tiny L}} \lambda_i \leq \mu_i \leq u_i^{\mbox{\tiny U}} \lambda_i \ \ \forall \, i \in [n] \Big\}.
\]
We note that the constraint $u_i^{\mbox{\tiny L}} \lambda_i \leq \mu_i \leq u_i^{\mbox{\tiny U}} \lambda_i$ ensures that (i) if $\lambda_i = 1$ (i.e., if appointee $i$ shows up) then $\mu_i \in [u_i^{\mbox{\tiny L}}, u_i^{\mbox{\tiny U}}]$ and (ii) if $\lambda_i = 0$ (i.e., if appointee $i$ does not show up) then $\mu_i = 0$ (i.e., the actual service duration is zero). Under the standard assumption that $d_{i+1} - d_i \leq c_{i+1}$ for all $i \in [n-1]$ (see Section \ref{sec:dras-duation} for elaboration of this assumption), the total cost of the appointment system for given $\bm s$ and $\bm \xi$ can be obtained from the following linear program:
\begin{equation} \label{equ:qxsl}
\begin{array}{rl}
g(\bm s, \bm \xi) :=  \min\limits_{\bm w, \bm v} & \sum \limits_{i=1}^n (c_i\lambda_i w_i + d_iv_i) + Cw_{n+1} \\
                                                    \st & w_i - v_{i-1} = \mu_{i-1} + w_{i-1} - s_{i-1} \ \ \ \forall \, i \in [2, n+1]_{\mathbb{Z}} \\
                                                         & \bm w \geq 0, \ w_1 = 0, \ \bm v \geq 0.
\end{array}
\end{equation}
Here, the cost of waiting time $c_i\lambda_iw_i$ is modeled from the perspective of appointments, i.e., this cost is waived if appointee $i$ does not show up. In addition, we note that the variables $\bm u$ do not explicitly appear in the above definitions of $\Xi$ and $g(\bm s, \bm \xi)$. In this section, we interpret $u_i$ as the service duration \emph{if appointee $i$ shows up}, i.e., $\bm u$ are conditional random variables depending on $\bm \lambda$. As a result, $\bm u$ may not even be observable when no-shows take place, while $\bm \xi$ are always observable, for example, from the historical data of service durations and no-shows. Similar to Section \ref{sec:duration_model}, we assume that we observe a finite set of $N$ i.i.d. samples of $\bm \xi$, denoted as $\{\widehat{\bm{\xi}}^1, \ldots, \widehat{\bm{\xi}}^N\}$. Then, we consider the following $p$-Wasserstein ambiguity set
\[
\D_p(\widehat \PP^N_{\bm \xi}, \, \epsilon) := \left\{ \QQ_{\bm \xi} \in \mathcal{P}(\Xi) : \ d_p(\QQ_{\bm \xi}, \, \widehat \PP^N_{\bm \xi}) \le \epsilon \right\},
\]
where $\widehat \PP^N_{\bm \xi}$ represents the empirical distribution of $\bm \xi$ based on the $N$ \iid~samples, i.e., $\widehat \PP^N_{\bm \xi} = \frac1N \sum_{j=1}^N \delta_{\bm{\widehat \xi}^j}$. Additionally, we can determine $K$, the budget of no-shows in $\Lambda$, based on these samples. For example, we can set
\begin{equation}\label{equ:setK}
K := \max_{j \in [N]} \left\{\sum_{i=1}^n (1-\widehat{\lambda}^j_i)\right\}.
\end{equation}

With the above Wasserstein ambiguity set, we formulate the following DRO model to seek a schedule that minimizes the expected total cost with regard to the worst-case distribution in $\D_p(\widehat \PP^N_{\bm \xi}, \, \epsilon)$:
\begin{equation} \tag{W-NS} \label{equ:dras-ns}
\widehat{Z}_{\mbox{\tiny NS}}(N, \,\epsilon) := \min \limits_{\bm s \in {\cal S}} \sup\limits_{\QQ_{\bm \xi} \in \D_p(\widehat \PP^N_{\bm \xi}, \, \epsilon)} \Ep_{\QQ_{\bm \xi}} [g( \bm s, \, \bm \xi)].
\end{equation}
We close this section by noting that similar asymptotic consistency and finite-data guarantee as in Section \ref{sec:duration_model} (see Theorems \ref{thm:a-c}--\ref{thm:f-d}) also hold for \eqref{equ:dras-ns}. In particular, as the data size $N$ increases to infinity, $\widehat{Z}_{\mbox{\tiny NS}}(N, \,\epsilon)$ converges to the optimal value of the stochastic AS model, where perfect information of the probability distribution of $\bm \xi$ is known. Accordingly, a \eqref{equ:dras-ns} optimal appointment schedule converges to the optimal schedule obtained from this stochastic model. In addition, with high confidence, \eqref{equ:dras-ns} provides an upper bound on the optimal value of the stochastic model with any finite data size $N$. We skip the formal statement of these two results to avoid repetition.

\subsection{Copositive programming reformulations} \label{sec:ns-cop}
In this section, we propose a copositive reformulation for~\eqref{equ:dras-ns}. As the first step, we represent $g(\bm s, \bm \xi)$ by the following dual linear program of \eqref{equ:qxsl}:
\begin{equation} \label{equ:qxsl-max}
\begin{array}{rl}
g(\bm s, \bm \xi) \ = \ \displaystyle\max_{\bm y \in \mathcal{Y}(\bm \lambda)} & \sum \limits_{i=1}^n (\mu_i - s_i) y_i \\
\end{array}
\end{equation}
where dual variables $\bm y$ are associated with the first constraint in \eqref{equ:qxsl} and the polyhedral feasible set $\mathcal{Y}(\bm \lambda)$ is described as
\[
\Y(\bm \lambda) := \left\{ \bm y \in \RR^n :
\begin{array}{l}
-y_i \leq d_i     \ \ \ \forall \, i \in [n]  \\
y_{i-1} - y_i \leq c_i\lambda_i  \ \ \ \forall \, i \in [2, n]_{\mathbb{Z}}  \\
y_n \leq C
\end{array}
\right\}.
\]
The strong duality between \eqref{equ:qxsl} and \eqref{equ:qxsl-max} holds because $\mathcal{Y}(\bm \lambda)$ is nonempty and compact for any $\bm \lambda \in \Lambda$. We formally state this fact in the following lemma and omit its proof due to its similarity to that of Lemma \ref{lem:y_set}.
\begin{lemma} \label{lem:ylambda}
For any $\bm \lambda \in \Lambda$, $\Y(\bm \lambda)$ is nonempty, compact, and convex.
\end{lemma}

Second, we derive a deterministic formulation to compute the worst-case expectation in \eqref{equ:dras-ns}, $\sup_{\QQ_{\bm \xi} \in \D_p(\widehat \PP^N_{\bm \xi}, \, \epsilon)} \Ep_{\QQ_{\bm \xi}} [g( \bm s, \, \bm \xi)]$, for fixed $\bm s \in \mathcal{S}$. We state this formulation in the following proposition and omit the proof due to the similarity to that of Proposition \ref{prop:ROFormat}.
\begin{proposition} \label{prop:wc-dras-ns}
For fixed $\bm s \in \mathcal{S}$, $\displaystyle\sup_{\QQ_{\bm \xi} \in \D_p(\widehat \PP^N_{\bm \xi}, \, \epsilon)} \Ep_{\QQ_{\bm \xi}} [g( \bm s, \, \bm \xi)]$ equals the optimal value of the following formulation:
\begin{align}
\inf \ & \ \epsilon^p \rho + \dfrac1N \sum\limits_{j=1}^{N} \sup_{\bm \xi \in \Xi} \left\{g(\bm s, \, \bm \xi) - \rho \|\bm \xi - \widehat{\bm{\xi}}^j\|_p^p \right\} \label{equ:wc-dras-ns-nonconvex} \\
\st \ & \ \rho \ge 0. \nonumber
\end{align}
\end{proposition}
The deterministic formulation in Proposition \ref{prop:wc-dras-ns} is computationally intractable particularly due to the maximization problem in the objective function \eqref{equ:wc-dras-ns-nonconvex}. As compared to its counterpart in Section \ref{sec:duration_model} (see \eqref{equ:vdp_nonconvex}), the problem in \eqref{equ:wc-dras-ns-nonconvex} involves both binary decision variables $\bm \lambda$ and continuous decision variables $\bm \mu$, which significantly increases the computational difficulty. Nevertheless, in the following theorem we derive a copositive reformulation of \eqref{equ:dras-ns} for $p \in \{1,2\}$. Similarly, we define the sets
${\cal F}_{\NS,1}^j$ for all $j \in [N]$ and ${\cal F}_{\NS,2}$:

\[
{\cal F}_{\NS,1}^j  := \left\{ (\bm \mu^+, \bm \mu^-,  \bm \lambda^+,   \bm \lambda^-,  \bm y) \in \RR^{5n} :
\begin{array}{l}
\bm \mu^+ \in \RR_+^n, \, \bm \mu^- \in \RR_+^n, \,  \bm \lambda^+ \in \RR_+^n, \, \bm \lambda^- \in \RR_+^n, \, \bm y \in \RR^n \\
 0 \le \bm{\lambda}^+ -\bm{\lambda}^- + \widehat{\bm \lambda}^j \le \bm e, \ \bm \lambda^+ \le \bm e, \ \bm \lambda^- \le \bm e   \\
  \sum\limits_{i=1}^n \left[1- (\lambda_i^+ - \lambda_i^- + \widehat{\lambda}_i^j) \right] \leq K  \\
 u_i^{\mbox{\tiny L}} (\lambda_i^+ - \lambda_i^- + \widehat{\lambda}_i^j) \leq \mu_i^+ - \mu_i^- +\widehat{\mu}_i^j   \ \ \forall \, i \in [n]  \\
  \mu_i^+ - \mu_i^- +\widehat{\mu}_i^j  \leq u_i^{\mbox{\tiny U}}(\lambda_i^+ - \lambda_i^- + \widehat{\lambda}_i^j) \ \ \forall \, i \in [n] \\
y_n \leq C, \ -y_i \leq d_i   \ \ \ \forall \, i \in [n]  \\
y_{i-1} - y_i \leq c_i (\lambda_i^+ - \lambda_i^- + \widehat{\lambda}_i^j)  \ \ \ \forall \, i \in [2, n]_{\mathbb{Z}}
\end{array}
\right\},
\]
\[
{\cal F}_{\NS,2} := \left\{ (\bm \mu, \bm \lambda, \bm y) \in \RR^{3n} :
\begin{array}{l}
\bm \mu \in \RR^n, \ \bm \lambda \in \RR^n, \ \bm y \in {\cal Y}(\bm \lambda) \\
 \sum\limits_{i=1}^n (1-\lambda_i) \leq K, \  0 \le \bm \lambda \le \bm e \\
 u_i^{\mbox{\tiny L}} \lambda_i \leq \mu_i \leq u_i^{\mbox{\tiny U}} \lambda_i \ \ \forall i \in [n]
\end{array}
\right\},
\]
and then construct their perspective sets $\K_{\NS,1}^j$ and $\K_{\NS,2}$ respectively:
\begin{equation*}
\K_{\NS,1}^j  :=  \text{closure} \Big( \Big \{ (t, \bm \mu^+, \bm \mu^-, \bm \lambda^+, \bm \lambda^-, \bm y) : (\bm \mu^+/t, \bm \mu^-/t, \bm \lambda^+/t, \bm\lambda^-/t, \bm{y} /t) \in \F_{\NS,1}^j, \ t > 0 \Big \} \Big ),
\end{equation*}
\begin{equation*}
\K_{\NS,2}  :=  \textup{closure} \Big ( \Big\{ (t, \bm \mu, \bm \lambda, \bm y)  :   (\bm \mu/t, \bm \lambda/t, \bm y/t)  \in \mathcal{F}_{\NS,2}, \ t > 0 \Big\} \Big ).
\end{equation*}

Now, we are ready to present copositive programming reformulations of \eqref{equ:dras-ns} for the cases of $p = 1, 2$. For ease of exposition, we define

\[
{\bm G}^1_j(\rho, \bm s) :=
\begin{bmatrix}
0  & -\frac12\rho\bm e^\top & -\frac12\rho\bm e^\top & -\frac12\rho\bm e^\top & -\frac12\rho\bm e^\top & \frac12(\widehat{\bm \mu}^j - \bm s)^\top \\
-\frac12\rho\bm e & 0 & 0 & 0 & 0 & \frac12 \bm I \\
-\frac12\rho\bm e & 0 & 0 & 0 & 0 & -\frac12 \bm I \\
-\frac12\rho\bm e & 0 & 0 & 0 & 0 & 0 \\
-\frac12\rho\bm e & 0 & 0 & 0 & 0 & 0 \\
\frac12(\widehat{\bm \mu}^j - \bm s) & \frac12 \bm I & -\frac12 \bm I & 0 & 0 & 0
\end{bmatrix},
\]
\[
{\bm J}_i := \begin{bmatrix} 0 \\ 0  \\  0 \\ \bm e_i \\  0 \\ 0  \end{bmatrix} \begin{bmatrix}0 \\ 0  \\  0 \\ \bm e_i \\  0 \\ 0 \end{bmatrix}^\top
- \dfrac12 \begin{bmatrix} 0 \\ 0  \\  0 \\ \bm e_i \\  0 \\ 0 \end{bmatrix}\begin{bmatrix} 1 \\  0  \\  0 \\ 0 \\  0 \\ 0 \end{bmatrix}^\top
- \dfrac12 \begin{bmatrix} 1 \\ 0  \\  0 \\ 0 \\  0 \\ 0 \end{bmatrix} \begin{bmatrix} 0 \\ 0  \\  0 \\ \bm e_i \\  0 \\ 0 \end{bmatrix}^\top,
{\bm M}_i := \begin{bmatrix} 0 \\ 0 \\ 0  \\  0 \\ \bm e_i \\  0   \end{bmatrix} \begin{bmatrix} 0 \\ 0 \\ 0  \\  0 \\ \bm e_i \\  0  \end{bmatrix}^\top
- \dfrac12 \begin{bmatrix} 0 \\ 0 \\ 0  \\  0 \\ \bm e_i \\  0  \end{bmatrix}\begin{bmatrix} 1 \\  0 \\ 0  \\  0 \\ 0 \\  0  \end{bmatrix}^\top
- \dfrac12 \begin{bmatrix} 1 \\ 0 \\ 0  \\  0 \\ 0 \\  0  \end{bmatrix} \begin{bmatrix} 0 \\ 0 \\ 0  \\  0 \\ \bm e_i \\  0  \end{bmatrix}^\top,
\]
\[
{\bm G}^2_j(\rho, \bm s) :=
\begin{bmatrix}
-\rho (\|\bm{\widehat \mu}_j\|^2 + \|\bm{\widehat \lambda}_j\|^2) & \rho \bm{\widehat \mu}_j^\top & \rho \bm{\widehat \lambda}_j^\top  &  - \frac12 \bm s^\top \\
\rho \bm{\widehat \mu}_j & -\rho \bm I &  0  &  \frac12\bm I \\
\rho \bm{\widehat \lambda}_j &  0 & -\rho \bm I  &  0 \\
-\frac12 \bm s & \frac12 \bm I &  0 &  0
\end{bmatrix},
\]
\[
{\bm N}_i := \begin{bmatrix} 0 \\  0 \\ \bm e_i \\  0 \end{bmatrix} \begin{bmatrix}0 \\   0 \\ \bm  e_i  \\  0 \end{bmatrix}^\top
- \dfrac12 \begin{bmatrix} 0 \\  0 \\ \bm e_i  \\  0 \end{bmatrix}\begin{bmatrix} 1 \\   0 \\  0 \\  0 \end{bmatrix}^\top
- \dfrac12 \begin{bmatrix} 1 \\  0 \\  0 \\  0 \end{bmatrix} \begin{bmatrix} 0 \\  0 \\ \bm e_i \\  0 \end{bmatrix}^\top.
\]

\begin{theorem} \label{thm:dras-ns}
When $p=1$, ~\eqref{equ:dras-ns} yields the same optimal value and the same optimal solutions for the following copositive program:
\begin{equation} \label{equ:dro_cop2_p1}
\begin{array}{lll}
\widehat{Z}_{\mbox{\tiny NS}}(N, \,\epsilon) = & \displaystyle\inf & \epsilon \rho + \dfrac1N \sum\limits_{j=1}^{N} \ \beta_j \\
& \st & \ \rho \in \RR, \  \bm \psi \in \RR^{n \times N}, \bm \phi \in \RR^{n \times N} \\
& &  \beta_j \bm e_1 \bm e_1^\top + \sum\limits_{i=1}^n \psi_{ij} {\bm M}_i + \sum\limits_{i=1}^n \phi_{ij} {\bm J}_i - {\bm G}^1_j(\rho, \, \bm  s) \in \COP(\mathcal{K}_{\NS,1}^j)  \ \ \ \forall \, j \in [N] \\
& & \rho \ge 0, \  \bm s \in {\cal S}.
\end{array}
\end{equation}
In addition, when $p=2$, \eqref{equ:dras-ns} yields the same optimal value and the same optimal solutions for the following copositive program:
\begin{equation} \label{equ:dro_cop2}
\begin{array}{lll}
\widehat{Z}_{\mbox{\tiny NS}}(N, \,\epsilon) = & \displaystyle\inf & \epsilon^2 \rho + \dfrac1N \sum\limits_{j=1}^{N} \ \beta_j \\
& \st & \rho \in \RR, \ \bm \beta \in \RR^N, \ \bm \psi \in \RR^{n \times N} \\
& & \beta_j \bm e_1 \bm e_1^\top + \sum\limits_{i=1}^n \psi_{ij} {\bm N}_i - {\bm G}^2_j(\rho, \, \bm  s) \in \COP(\mathcal{K}_{\NS,2})  \ \ \ \forall \, j \in [N] \\
& & \rho \ge 0, \  \bm s \in {\cal S}.
\end{array}
\end{equation}
\end{theorem}


\subsection{Tractable reformulations} \label{sec:ns-tractable}
In this subsection, we consider a setting that admits a tractable reformulation of \eqref{equ:dras-ns} for general $p \geq 1$. Different from Section \ref{sec:duration_model} that considers random service durations only, \eqref{equ:dras-ns} incorporates no-shows and the accompanying Bernoulli random variables. To address the amplified computational challenge, we adopt a different approach based on dynamic programming and network flow techniques. This leads to a (polynomial-size) linear programming reformulation of \eqref{equ:dras-ns} when $p = 1$ and a (polynomial-size) second-order cone programming reformulation when $p > 1$ and $p$ is rational. To this end, we make the following assumption on the costs of waiting and idleness of the AS system.
\begin{assumption}[Homogeneous Costs] \label{ass:no-show-tractable-homogeneous-costs}
The costs of appointment waiting and server idleness are homogeneous, i.e.,
$c_0 := c_1 = c_2 = \cdots c_n$ and $d_0 := d_1 = d_2 = \cdots = d_n$.
\end{assumption}
Assumption~\ref{ass:no-show-tractable-homogeneous-costs} is non-stringent because (1) the server idleness is always associated with the same server, and (2) although the waiting times are associated with different appointments, the scheduler should consider a homogeneous cost to ensure fairness among all appointments. Under this assumption, we can assume $c_0 = 1$ without loss of generality.

We first recast and identify an optimality condition (OC) of the maximization problem $\sup_{\bm \xi \in \Xi} \ \{g(\bm s, \, \bm \xi) - \rho \|\bm \xi - \bm{\widehat{\xi}}^j\|^p_p\}$ in formulation \eqref{equ:wc-dras-ns-nonconvex}.
\begin{proposition} \label{prop:ns-oc}
Denote $\omega'_j(\rho, \bm s) = \sup_{\bm \xi \in \Xi} \ \{g(\bm s, \, \bm \xi) - \rho \|\bm \xi - \bm{\widehat{\xi}}^j\|^p_p\}$ and suppose that Assumptions \ref{ass:tractable-support}--\ref{ass:no-show-tractable-homogeneous-costs} hold. Then, for all $p \geq 1$, $j \in [N]$, $\rho \geq 0$, and $\bm s \in \mathcal{S}$,
\begin{equation}\label{equ:ns-ref}
\omega'_j(\rho, \bm s) \ = \ \sup_{\bm \lambda \in \Lambda, \ \bm y \in \mathcal{Y}(\bm \lambda)} \ \left\{\sum_{i=1}^n f_{ij}(\lambda_i, y_i)\right\},
\end{equation}
where
\begin{equation}\label{equ:z-def}
f_{ij}(\lambda_i, y_i) := \displaystyle\sup_{u^{\mbox{\tiny L}}_i \lambda_i \leq \mu_i \leq u^{\mbox{\tiny U}}_i \lambda_i} \left\{y_i(\mu_i - s_i) - \rho|\mu_i - \widehat{\mu}^j_i|^p - \rho|\lambda_i - \widehat{\lambda}^j_i|^p\right\}
\end{equation}
for all $i \in [n]$. In addition, without loss of optimality, variables $\bm \lambda$ and $\bm y$ satisfy the following optimality condition:
\begin{equation}\tag{OC}\label{equ:oc}
\left\{\begin{array}{ll}
y_n = C \mbox{ or } y_n = -d_0 & \\
y_i = -d_0 \mbox{ or } y_i = y_{i+1} + \lambda_{i+1} & \forall i \in [n-1] \\
y_i \in \mathcal{Y}_i & \forall i \in [n],
\end{array}
\right.
\end{equation}
where $\mathcal{Y}_i := [-d_0, -d_0 + n - i]_{\mathbb{Z}} \cup [C, C + n - i]_{\mathbb{Z}}$.
\end{proposition}
\eqref{equ:oc} shrinks the search space of problem \eqref{equ:ns-ref} from a union of polytope to a finite set of points. Specifically, for each $i \in [n]$, variable $y_i$ has $2(n-i+1)$ possible choices because $\mathcal{Y}_i$ consists of $2(n-i+1)$ elements. More importantly, given the value of $y_{i+1}$, $y_i$ can take only two values: $-d_0$ or $y_{i+1} + \lambda_{i+1}$. This allows us to recast $\omega'_j(\rho, \bm s)$ as a dynamic program (DP), that is, we solve \eqref{equ:ns-ref} by sequentially determining $(\lambda_1, y_1)$, $(\lambda_2, y_2)$, and so on. To this end, for each $i \in [n]$, we define the state of this DP as $(\bar{\lambda}_i, y_i) \in [0, K]_{\mathbb{Z}} \times \mathcal{Y}_i$, where $\bar{\lambda}_i := \sum_{k=1}^i (1 - \lambda_k)$ records the total number of no-shows among the first $i$ appointments\footnote{Note that $\lambda_1 = 1 - \bar{\lambda}_1$ and $\lambda_i = \bar{\lambda}_{i-1} - \bar{\lambda}_i + 1$ for all $i \in [2, n]_{\mathbb{Z}}$. In addition, although $\bar{\lambda}_i \in [0, \min\{i, K\}]_{\mathbb{Z}}$, we consider $\bar{\lambda}_i \in [0, K]_{\mathbb{Z}}$ in this DP for notational brevity.}, and define the value function of this DP through
\begin{align*}
& V_{nj}(\bar{\lambda}_n, y_n) \ = \ 0 \ \ \forall \, (\bar{\lambda}_n, y_n) \in [0, K]_{\mathbb{Z}} \times \mathcal{Y}_n, \\
& V_{(i-1)j}(\bar{\lambda}_{i-1}, y_{i-1}) \ = \ \sup_{\bar{\lambda}_i, y_i} \ \ f_{ij}(\bar{\lambda}_{i-1} - \bar{\lambda}_i + 1, y_i) + V_{ij}(\bar{\lambda}_i, y_i) \\
& \hspace{3.8cm} \mbox{s.t.} \ \ \bar{\lambda}_i \in \{\bar{\lambda}_{i-1}, \bar{\lambda}_{i-1} + 1\}, \ \ y_i \in \{- d_0, \ y_{i-1} - (\bar{\lambda}_{i-1} - \bar{\lambda}_i + 1)\} \\
& \forall \, i \in [2, n]_{\mathbb{Z}}, \ \ \forall \, (\bar{\lambda}_{i-1}, y_{i-1}) \in [0, K]_{\mathbb{Z}} \times \mathcal{Y}_{i-1}.
\end{align*}
It follows that $\omega'_j(\rho, \bm s) = \sup_{(\bar{\lambda}_1, y_1) \in \{0, 1\}\times \mathcal{Y}_1}\{ f_{1j}(1 - \bar{\lambda}_1, y_1) + V_{1j}(\bar{\lambda}_1, y_1) \}$.

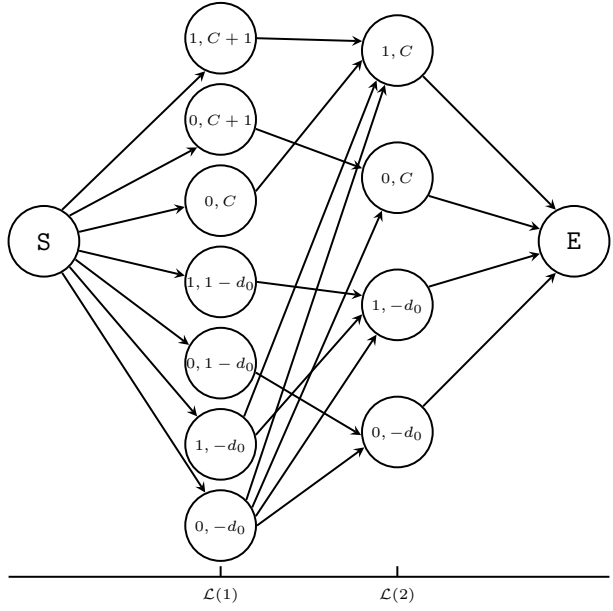
\begin{figure}[bth]
\begin{center}
\resizebox{0.5\textwidth}{!}{
\begin{tikzpicture}[scale=0.5, roundnode/.style={circle, draw=black, fill=white, thick, minimum size=1cm}]
    \def \discrepancy {0.3}
    \def \radius {1}
    \node[roundnode] (start) at (-5, 0) {\texttt{S}};
    \node[roundnode] (1C1) at (0, 5 * \radius + 5 * \discrepancy / 2) {};
    \draw (0, 5 * \radius + 5 * \discrepancy / 2 - 0.5) node[anchor=south] {\tiny $1, C+1$};
    \node[roundnode] (0C1) at (0, 3 * \radius + 3 * \discrepancy / 2) {};
    \draw (0, 3 * \radius + 3 * \discrepancy / 2 - 0.5) node[anchor=south] {\tiny $0, C+1$};
    \node[roundnode] (0C0) at (0, 1 * \radius + 1 * \discrepancy / 2) {};
    \draw (0, 1 * \radius + 1 * \discrepancy / 2 - 0.5) node[anchor=south] {\tiny $0, C$};
    \node[roundnode] (1d1) at (0, - 1 * \radius - 1 * \discrepancy / 2) {};
    \draw (0, - 1 * \radius - 1 * \discrepancy / 2 - 0.5) node[anchor=south] {\tiny $1, 1-d_0$};
    \node[roundnode] (0d1) at (0, - 3 * \radius - 3 * \discrepancy / 2) {};
    \draw (0, - 3 * \radius - 3 * \discrepancy / 2 - 0.5) node[anchor=south] {\tiny $0, 1-d_0$};
    \node[roundnode] (1d0) at (0, - 5 * \radius - 5 * \discrepancy / 2) {};
    \draw (0, - 5 * \radius - 5 * \discrepancy / 2 - 0.5) node[anchor=south] {\tiny $1, -d_0$};
    \node[roundnode] (0d0) at (0, - 7 * \radius - 7 * \discrepancy / 2) {};
    \draw (0, - 7 * \radius - 7 * \discrepancy / 2 - 0.5) node[anchor=south] {\tiny $0, -d_0$};
    \node[roundnode] (1C) at (5, 4.5 * \radius + 3 * \discrepancy) {};
    \draw (5, 4.5 * \radius + 3 * \discrepancy - 0.5) node[anchor=south] {\tiny $1, C$};
    \node[roundnode] (0C) at (5, 1.5 * \radius + 1 * \discrepancy) {};
    \draw (5, 1.5 * \radius + 1 * \discrepancy - 0.5) node[anchor=south] {\tiny $0, C$};
    \node[roundnode] (1d) at (5, - 1.5 * \radius - 1 * \discrepancy) {};
    \draw (5, - 1.5 * \radius - 1 * \discrepancy - 0.5) node[anchor=south] {\tiny $1, -d_0$};
    \node[roundnode] (0d) at (5, - 4.5 * \radius - 3 * \discrepancy) {};
    \draw (5, - 4.5 * \radius - 3 * \discrepancy - 0.5) node[anchor=south] {\tiny $0, -d_0$};
    \node[roundnode] (tail) at (10, 0) {};
    \draw (10, -0.5) node[anchor=south] {\texttt{E}};
    \draw[thick, -stealth] (start.60) -- (1C1.245);
    \draw[thick, -stealth] (start.45) -- (0C1.230);
    \draw[thick, -stealth] (start.15) -- (0C0.190);
    \draw[thick, -stealth] (start.345) -- (1d1.170);
    \draw[thick, -stealth] (start.330) -- (0d1.150);
    \draw[thick, -stealth] (start.315) -- (1d0.130);
    \draw[thick, -stealth] (start.300) -- (0d0.115);
    \draw[thick, -stealth] (1C1.0) -- (1C.165);
    \draw[thick, -stealth] (0C1.345) -- (0C.165);
    \draw[thick, -stealth] (0C0.15) -- (1C.195);
    \draw[thick, -stealth] (1d1.0) -- (1d.165);
    \draw[thick, -stealth] (0d1.345) -- (0d.185);
    \draw[thick, -stealth] (1d0.15) -- (1d.195);
    \draw[thick, -stealth] (0d0.0) -- (0d.205);
    \draw[thick, -stealth] (1d0.50) -- (1C.235);
    \draw[thick, -stealth] (0d0.45) -- (1C.250);
    \draw[thick, -stealth] (0d0.30) -- (0C.245);
    \draw[thick, -stealth] (0d0.15) -- (1d.235);
    \draw[thick, -stealth] (1C.315) -- (tail.120);
    \draw[thick, -stealth] (0C.330) -- (tail.160);
    \draw[thick, -stealth] (1d.30) -- (tail.200);
    \draw[thick, -stealth] (0d.45) -- (tail.240);
    \draw[thick] (-6, - 8 * \radius - 5 * \discrepancy) -- (11, - 8 * \radius - 5 * \discrepancy);
    \draw[thick] (0, - 8 * \radius - 5 * \discrepancy) -- (0, - 8 * \radius - 5 * \discrepancy + 0.3);
    \draw (0, - 8 * \radius - 5 * \discrepancy - 1) node[anchor=south] {\tiny ${\cal L}(1)$};
    \draw[thick] (5, - 8 * \radius - 5 * \discrepancy) -- (5, - 8 * \radius - 5 * \discrepancy + 0.3);
    \draw (5, - 8 * \radius - 5 * \discrepancy - 1) node[anchor=south] {\tiny ${\cal L}(2)$};
\end{tikzpicture}
}
\caption[]{An Example of the Network $(\mathcal{G}, \mathcal{E})$ with $n = 2$ and $K=1$.
The nodes in layers 1 and 2 are arranged along the coordinates ${\cal L}(1)$ and
${\cal L}(2)$, respectively.} \label{fig:network-2}
\end{center}
\end{figure}

To provide more intuition, we map the states and state transitions of this DP onto an $n$-layer network. In layer $i$, $\forall \, i \in [n]$, we construct a set ${\cal L}(i)$ of nodes with ${\cal L}(i) := \left\{(\bar{\lambda}_i, y_i): \bar{\lambda}_i \in [0, K]_{\mathbb{Z}}, \ y_i \in \mathcal{Y}_i \right\}$. In addition, we construct directed arcs between two consecutive layers in the network. For $i \in [2, n]_{\mathbb{Z}}$, there is an arc from node $(\bar{\lambda}_{i-1}, y_{i-1}) \in {\cal L}({i-1})$ to node $(\bar{\lambda}_i, y_i) \in {\cal L}(i)$ if $\bar{\lambda}_i \in \{\bar{\lambda}_{i-1}, \bar{\lambda}_{i-1} + 1\}$ and $y_i \in \{-d_0, \ y_{i-1} - (\bar{\lambda}_{i-1} - \bar{\lambda}_i + 1)\}$. Finally, we construct a (dummy) starting node \texttt{S} and a (dummy) ending node \texttt{E}, as well as arcs from \texttt{S} to all nodes $(\bar{\lambda}_1, y_1) \in {\cal L}(1)$ with $\bar{\lambda}_1 \in \{0, 1\}$ and from all nodes in ${\cal L}(n)$ to \texttt{E}. We depict an example of this network in Figure~\ref{fig:network-2}. It can be observed that each feasible solution of the DP corresponds to an \texttt{S}-\texttt{E} path in the network, and vice versa. This indicates that, by associating appropriate length to each arc, solving the longest-path problem in the network produces an optimal solution to the DP. We summarize this observation in the following proposition.
\begin{proposition} \label{prop:tractable-no-show-bounded-support}
Suppose that Assumptions \ref{ass:tractable-support}--\ref{ass:no-show-tractable-homogeneous-costs} hold and consider a directed network $(\mathcal{G}, \mathcal{E})$ such that $\mathcal{G} = \bigcup_{i=1}^n {\cal L}(i) \cup \{\texttt{S}, \texttt{E}\}$ and $\mathcal{E} = \{( (\bar{\lambda}_{i-1}, y_{i-1}),
\ (\bar{\lambda}_i, y_i) ) \in {\cal L}({i-1}) \times {\cal L}(i): \ \bar{\lambda}_i \in \{\bar{\lambda}_{i-1}, \bar{\lambda}_{i-1} + 1\}, \ y_i \in \{- d_0, \ y_{i-1} - (\bar{\lambda}_{i-1} - \bar{\lambda}_i + 1)\}, \ i \in [2, n]_{\mathbb{Z}}\}
\cup \{(\texttt{S}, \ (\bar{\lambda}_1, y_1)): \ \bar{\lambda}_1 \in \{0, 1\}, \ (\bar{\lambda}_1, y_1) \in {\cal L}(1)\}
\cup \{((\bar{\lambda}_n, y_n), \ \texttt{E}): \ (\bar{\lambda}_n, y_n) \in {\cal L}(n)\}$.
For each arc $(k, \ell) \in \mathcal{E}$, let $g_{k\ell j}$ represent its length such that
$$
g_{k\ell j} = \left\{\begin{array}{ll} f_{1j}(1 - \bar{\lambda}_1, \ y_1) & \mbox{if $k = \texttt{S}$ and $\ell = (\bar{\lambda}_1, y_1)$} \\
f_{ij}(\bar{\lambda}_{i-1} - \bar{\lambda}_i + 1, \ y_i) & \mbox{if $k = (\bar{\lambda}_{i-1}, y_{i-1})$ and $\ell = (\bar{\lambda}_i, y_i)$ for $i \in [2, n]_{\mathbb{Z}}$} \\
0 & \mbox{if $\ell = \texttt{E}$,}
\end{array}\right.
$$
where $f_{ij}(\cdot, \cdot)$ is defined in \eqref{equ:z-def}. Then, we have
\begin{subequations}
\begin{align}
\omega'_j(\rho, \bm s) \ = \ \max_{\bm z} \ & \ \sum_{(k, \ell) \in \mathcal{E}} g_{k\ell j} z_{k\ell} \label{equ:lp-no-show-bounded-obj} \\
\mbox{s.t.} \ & \ \sum_{\ell: (k, \ell) \in \mathcal{E}} z_{k\ell} - \sum_{\ell: (\ell, k) \in \mathcal{E}} z_{\ell k} = \left\{\begin{array}{ll} 1 & \mbox{if $k = \texttt{S}$} \\
0 & \mbox{if $k \neq \texttt{S}, \texttt{E}$} \\
-1 & \mbox{if $k = \texttt{E}$}
\end{array}\right. \ \ \forall \, k \in \mathcal{G}   \label{equ:lp-no-show-bounded-con-1} \\
& \ z_{k\ell} \geq 0 \ \ \ \forall \, (k, \ell) \in {\cal E}. \label{equ:lp-no-show-bounded-con-2}
\end{align}
\end{subequations}
\end{proposition}
We note that $|\mathcal{G}| = \mathcal{O}(Kn^2)$ and $|\mathcal{E}| = \mathcal{O}(Kn^2)$. This indicates that $\omega'_j(\rho, \bm s)$ can be computed in polynomial time by solving the linear program \eqref{equ:lp-no-show-bounded-obj}--\eqref{equ:lp-no-show-bounded-con-2}. Following the equivalence between separation and convex optimization (see~\cite{grotschel1981ellipsoid}), Proposition \ref{prop:tractable-no-show-bounded-support} implies that~\eqref{equ:dras-ns} can be solved in polynomial time to find an optimal appointment schedule under both random no-shows and service durations.

As a generalization of Theorem \ref{thm:tractable}, the following theorem recasts~\eqref{equ:dras-ns} as a linear program and a second-order cone program for $p = 1$ and $p=2$, respectively. Similar second-order conic reformulations can be obtained for any rational $p \geq 1$.
\begin{theorem} \label{thm:tractable-ns-bounded}
Suppose that Assumptions \ref{ass:tractable-support}--\ref{ass:no-show-tractable-homogeneous-costs} hold and define
\begin{align*}
& \mathcal{E}^0_1 := \left\{(\texttt{S}, \ell) \in \mathcal{E}: \ \ell = (1, y_i) \right\}, \ \ \mathcal{E}^1_1 := \left\{(\texttt{S}, \ell) \in \mathcal{E}: \ \ell = (0, y_i)\right\}, \\
& \mathcal{E}^0_i := \{(k, \ell) \in \mathcal{E}: \ k = (\bar{\lambda}_{i-1}, y_{i-1}), \ell = (\bar{\lambda}_i, y_i), \mbox{ and } \bar{\lambda}_i = \bar{\lambda}_{i-1} + 1 \} \ \ \forall \, i \in [2, n]_{\mathbb{Z}}, \\
& \mathcal{E}^1_i := \{(k, \ell) \in \mathcal{E}: \ k = (\bar{\lambda}_{i-1}, y_{i-1}), \ell = (\bar{\lambda}_i, y_i), \mbox{ and } \bar{\lambda}_i = \bar{\lambda}_{i-1} \} \ \ \forall \, i \in [2, n]_{\mathbb{Z}}, \\
& \mbox{and} \ \ \mathcal{E}_{\texttt{E}} := \{(k, \ell) \in \mathcal{E}: \ \ell = \texttt{E}\}.
\end{align*}
Then, when $p = 1$, \eqref{equ:dras-ns} yields the same optimal value and the same set of optimal appointment schedules as the following linear program:
\begin{align} \label{equ:ns-ref-1}
\begin{split}
\min_{\rho, \bm s, \bm \alpha} \ & \ \epsilon \rho + \frac1N \sum_{j=1}^N \left( \alpha^j_{\texttt{S}} - \alpha^j_{\texttt{E}} \right) \\
\mbox{s.t.} \ & \ \alpha^j_k - \alpha^j_{\ell} \geq 0 \ \ \forall \, (k, \ell) \in \mathcal{E}_{\texttt{E}} \ \ \forall \, j \in [N] \\
& \ \alpha^j_k - \alpha^j_{\ell} + y_i s_i + (\widehat{\mu}^j_i + \widehat{\lambda}^j_i) \rho \geq 0 \ \ \forall \, i \in [n], \ \forall \, (k, \ell) \in \mathcal{E}^0_i, \ \forall \,  j \in [N] \\
& \hspace{-0.15cm} \left.\begin{array}{l}
\alpha^j_k - \alpha^j_{\ell} + y_i s_i + (1 - \widehat{\lambda}^j_i + \widehat{\mu}^j_i - u^{\mbox{\tiny L}}_i) \rho \geq u^{\mbox{\tiny L}}_i y_i \\[0.2cm]
\alpha^j_k - \alpha^j_{\ell} + y_i s_i + (1 - \widehat{\lambda}^j_i) \rho \geq \widehat{\mu}^j_i y_i \\[0.2cm]
\alpha^j_k - \alpha^j_{\ell} + y_i s_i + (1 - \widehat{\lambda}^j_i + u^{\mbox{\tiny U}}_i - \widehat{\mu}^j_i) \rho \geq u^{\mbox{\tiny U}}_i y_i
\end{array}\right\}
\begin{array}{l}
\forall \, i \in [n], \ \forall \, (k, \ell) \in \mathcal{E}^1_i \\
\forall \,  j \in [N]
\end{array} \\
& \ \rho \geq 0, \ \bm s \in {\cal S}.
\end{split}
\end{align}
In addition, when $p=2$, \eqref{equ:dras-ns} yields the same optimal value and the same set of optimal appointment schedules as the following second-order cone program:
\begin{align*}
\min_{\substack{\rho, \bm s, \bm \alpha\\ \bm \beta, \bm \varphi}} \ & \ \epsilon^2 \rho + \frac1N \sum_{j=1}^N \left( \alpha^j_{\texttt{S}} - \alpha^j_{\texttt{E}} \right) \\
\mbox{s.t.} \ & \ \alpha^j_k - \alpha^j_{\ell} \geq 0 \ \ \forall \, (k, \ell) \in \mathcal{E}_{\texttt{E}} \ \ \forall \, j \in [N] \\
& \ \alpha^j_k - \alpha^j_{\ell} + y_i s_i + \left((\widehat{\mu}^j_i)^2 + \widehat{\lambda}^j_i\right) \rho \geq 0 \ \ \forall \, i \in [n], \ \forall \, (k, \ell) \in \mathcal{E}^0_i, \ \forall \,  j \in [N] \\
& \ \alpha^j_k - \alpha^j_{\ell} + y_i s_i + (1 - \widehat{\lambda}^j_i)\rho - \varphi_{k\ell j} - (\widehat{\mu}^j_i - u^{\mbox{\tiny L}}_i) \beta^{\mbox{\tiny L}}_{k\ell j} - (u^{\mbox{\tiny U}}_i - \widehat{\mu}^j_i) \beta^{\mbox{\tiny U}}_{k\ell j} \geq \widehat{\mu}^j_i y_i \\
& \ \forall \, i \in [n], \ \forall \, (k, \ell) \in \mathcal{E}^1_i, \ \forall \,  j \in [N] \\
& \ \left\|\begin{bmatrix} \beta^{\mbox{\tiny L}}_{k\ell j} - \beta^{\mbox{\tiny U}}_{k\ell j} + y_i \\ \varphi_{k\ell j} - \rho \end{bmatrix}\right\|_2 \leq \varphi_{k\ell j} + \rho \ \ \forall \, i \in [n], \ \forall \, (k, \ell) \in \mathcal{E}^1_i, \ \forall \,  j \in [N] \\
& \ \rho \geq 0, \ \bm s \in {\cal S}.
\end{align*}
\end{theorem}
We note that both reformulations in Theorem \ref{thm:tractable-ns-bounded} involve $\mathcal{O}(KNn^2)$ decision variables and $\mathcal{O}(KNn^2)$ constraints. We further derive a worst-case probability distribution $\mathbb{Q}^\star_{\bm \xi}$ of the random no-shows and service durations that attains $\sup_{\mathbb{Q}_{\bm \xi} \in \mathcal{D}_1(\widehat{\mathbb{P}}^N_{\bm \xi}, \epsilon)} \mathbb{E}_{\mathbb{Q}_{\bm \xi}}[g(\bar{\bm{s}}, \bm \xi)]$. This distribution can be applied to stress test an appointment schedule generated from any decision-making processes.
\begin{theorem} \label{thm:no-show-wc-dist}
For fixed $\bar{\bm s} \in \mathcal{S}$ and $\epsilon \geq 0$, $\displaystyle\sup_{\mathbb{Q}_{\bm \xi} \in \mathcal{D}_1(\widehat{\mathbb{P}}^N_{\bm \xi}, \epsilon)} \mathbb{E}_{\mathbb{Q}_{\bm \xi}}[g(\bar{\bm s}, \bm \xi)]$ equals the optimal value of the following linear program:
\begin{subequations} \label{equ:no-show-wc-lp}
\begin{align}
\max\limits_{\bm o, \bm p, \bm q, \bm w, \bm r} \ & \ \dfrac1N\sum\limits_{j=1}^N \sum\limits_{i=1}^{n} y_i \left[ \sum\limits_{(k, \ell) \in \mathcal{E}^1_i} \left( (u_i^{\tiny L} - \bar{s}_i) q_{k\ell j} + (\widehat{\mu}^j_i - \bar{s}_i) w_{k\ell j} + (u_i^{\tiny U} - \bar{s}_i) r_{i\ell j} \right) - \sum\limits_{(k, \ell) \in \mathcal{E}^0_i} \bar{s}_i p_{k\ell j} \right] \label{equ:no-show-wc-lp-obj} \\
\st \ & \ \dfrac1N  \sum\limits_{j=1}^N \sum\limits_{i=1}^{n} \Biggl[ \sum\limits_{(k, \ell) \in \mathcal{E}^0_i} (\widehat{\mu}^j_i + \widehat{\lambda}^j_i) p_{k\ell j} + \sum\limits_{(k, \ell) \in \mathcal{E}^1_i} \Biggl( \bigl(1 - \widehat{\lambda}^j_i + \widehat{\mu}^j_i - u^{\mbox{\tiny L}}_i \bigr) q_{k\ell j} \nonumber \\
& + \bigl(1 - \widehat{\lambda}^j_i\bigr) w_{k\ell j} + \bigl(1 - \widehat{\lambda}^j_i + u_i^{\tiny U} - \widehat{\mu}^j_i\bigr) r_{k\ell j} \Biggr) \Biggr] \leq \epsilon \label{equ:no-show-wc-lp-con-1} \\
& \sum\limits_{\ell:(k, \ell) \in \mathcal{E}^0} p_{k\ell j} + \sum\limits_{\ell:(k, \ell) \in \mathcal{E}^1} (q_{k\ell j} + w_{k\ell j} + r_{k\ell j}) - \sum\limits_{\ell:(\ell, k) \in \mathcal{E}^0} p_{\ell kj} \nonumber \\
& - \sum\limits_{\ell:(\ell, k) \in \mathcal{E}^1} (q_{\ell kj} + w_{\ell kj} + r_{\ell kj}) = \left\{ \begin{array}{ll} 1 & \mbox{if $k = \texttt{S}$}\\ 0 & \mbox{if $k \neq \texttt{S}$} \end{array} \right. \ \ \forall \, k \in \mathcal{E}\setminus\bigl(\mathcal{L}(n)\cup\texttt{E}\bigr), \ \forall \, j \in [N] \label{equ:no-show-wc-lp-con-2} \\
& o_{k\texttt{E} j} - \sum\limits_{\ell:(\ell, k) \in \mathcal{E}^0} p_{\ell kj} - \sum\limits_{\ell:(\ell, k) \in \mathcal{E}^1} (q_{\ell kj} + w_{\ell kj} + r_{\ell kj}) = 0 \ \ \forall \, k \in \mathcal{L}(n), \ \forall \, j \in [N] \label{equ:no-show-wc-lp-con-3} \\
& \sum\limits_{k:(k, \texttt{E}) \in \mathcal{E}_{\texttt{E}}} o_{k\texttt{E} j} = 1 \ \ \forall \, j \in [N] \label{equ:no-show-wc-lp-con-4} \\
& o_{k\ell j}, \ p_{k\ell j}, \ q_{k\ell j}, \ w_{k\ell j}, \ r_{k\ell j} \geq 0 \ \ \forall \, (k, \ell) \in \mathcal{E}, \ \forall \, j \in [N], \label{equ:no-show-wc-lp-con-5}
\end{align}
\end{subequations}
where $\mathcal{E}^0 := \cup_{i=1}^n \mathcal{E}^0_i$ and $\mathcal{E}^1 := \cup_{i=1}^n \mathcal{E}^1_i$. Let $\{o^\star_{k\ell j}, p^\star_{k\ell j}, q^\star_{k\ell j}, w^\star_{k\ell j}, r^\star_{k\ell j}\}$ be an optimal solution of the above linear program and define $\mathcal{P} = \{\bm{z}\in \{0, 1\}^{|\mathcal{E}|}: \ \mbox{\eqref{equ:lp-no-show-bounded-con-1}}\}$. Then, for all $j \in [N]$, there exists a distribution $\mathbb{P}^j_{\bm z}$ on $\mathcal{P}$ such that (i) $\mathbb{P}^j_{\bm z}\{z_{k\ell} = 1\} = p^\star_{k\ell j}$ for all $(k, \ell) \in \mathcal{E}^0$ and (ii) $\mathbb{P}^j_{\bm z}\{z_{k\ell} = 1\} = q^\star_{k\ell j} + w^\star_{k\ell j} + r^\star_{k\ell j}$ for all $(k, \ell) \in \mathcal{E}^1$. Furthermore, define the probability distribution
$$
\mathbb{Q}^\star_{\bm \xi} = \frac{1}{N}\sum_{j=1}^N \sum_{\bm{\zeta} \in \mathcal{P}} \mathbb{P}^j_{\bm z}\{\bm z = \bm \zeta\} \delta_{\bm{\xi}^j(\bm \zeta)},
$$
where, for all $i \in [n]$ and $j \in [N]$,
$$
\xi_i^j(\bm \zeta) = \sum_{(k, \ell) \in \mathcal{E}^0_i} \zeta_{k\ell} (0, 0) + \sum_{(k, \ell) \in \mathcal{E}^1_i} \zeta_{k\ell} (\mu_{k\ell j}, 1) \ \ \mbox{and} \ \ \mu_{k\ell j} = \widehat{\mu}^j_i + \frac{q^\star_{k\ell j}(u_i^{\tiny L} - \widehat{\mu}^j_i )} {q^\star_{k\ell j} + w^\star_{k\ell j} + r^\star_{k\ell j}} + \frac{r^\star_{k\ell j}(u_i^{\tiny U} - \widehat{\mu}^j_i )} {q^\star_{k\ell j} + w^\star_{k\ell j} + r^\star_{k\ell j}}.
$$
Here we adopt an extended arithmetic given by $0/0 = 0$. Then, $\mathbb{Q}^\star_{\bm \xi}$ belongs to the Wasserstein ambiguity set $\mathcal{D}_1(\widehat{\mathbb{P}}^N_{\bm \xi}, \epsilon)$ and $\mathbb{E}_{\mathbb{Q}^\star_{\bm \xi}}[g(\bar{\bm s}, \bm \xi)] = \displaystyle\sup_{\mathbb{Q}_{\bm \xi} \in \mathcal{D}_1(\widehat{\mathbb{P}}^N_{\bm \xi}, \epsilon)} \mathbb{E}_{\mathbb{Q}_{\bm \xi}}[g(\bar{\bm s}, \bm \xi)]$.
\end{theorem}
\begin{remark}
Once again, we can view the \eqref{equ:dras-ns} model as a two-person game between the AS scheduler and the nature. Theorem \ref{thm:no-show-wc-dist} provides an intuitive interpretation of how the nature picks the worst-case distribution $\mathbb{Q}^\star_{\bm \xi}$. Specifically, $\mathbb{Q}^\star_{\bm \xi}$ is a mixture of $N$ distributions, i.e., $\mathbb{Q}^\star_{\bm \xi} = (1/N)\sum_{j=1}^N \mathbb{Q}^j_{\bm \xi}$, where each $\mathbb{Q}^j_{\bm \xi} := \sum_{\bm{\zeta} \in \mathcal{P}} \mathbb{P}^j_{\bm z}\{\bm z = \bm \zeta\} \delta_{\bm{\xi}^j(\bm \zeta)}$ pertains to the $j^{\mbox{\tiny th}}$ data sample. Note that $\mathcal{P}$ consists of all the \texttt{S}--\texttt{E} paths in the network $(\mathcal{G}, \mathcal{E})$. Hence, what the nature does under $\mathbb{Q}^j_{\bm \xi}$ is to: (i) randomly pick an \texttt{S}--\texttt{E} path $\bm \zeta \in \mathcal{P}$ following distribution $\mathbb{P}^j_{\bm z}$ and (ii) for each $i \in [n]$, if the $i^{\mbox{\tiny th}}$ arc of the path belongs to $\mathcal{E}^0_i$ (i.e., if $(k, \ell) \in \mathcal{E}^0_i$) then appointment $i$ does not show up and accordingly the service duration equals zero; and if this arc belongs to $\mathcal{E}^1_i$ then appointment $i$ shows up and lasts for $\mu_{k\ell j}$ long.
\end{remark}

\section{Numerical Experiments} \label{sec:numerical}
In this section, we demonstrate the effectiveness of the proposed approach via numerical experiments. Throughout these experiments, we adopt $1$-Wasserstein ambiguity sets, i.e., $p = 1$. All computations are conducted on an 8-core 2.3 GHz Linux PC with 16 GB RAM. We implement the experiments in Python 2.7.12 and solve the linear programming problems by using Gurobi 8.0.1.

\subsection{Random service durations} \label{sec:exp_rs}
We begin with the \eqref{equ:wdro} model discussed in Section~\ref{sec:duration_model}, where all appointments show up and the service durations are random. We consider $n=10$ appointments and the unit waiting, idleness, and overtime costs are set to be $c_1 =\cdots = c_{10} = 2$, $d_1 = \cdots = d_{10} =1$, and $C=20$, respectively. Additionally, we consider each of the following three distributions to generate the data in our experiments:
\begin{itemize}
\item[LN:] The service duration $u_i$ of each appointment $i$ independently follows a lognormal distribution with mean $\mu_i$ and standard deviation $\sigma_i$, which are uniformly sampled from the intervals $[0.9, 1.1]$ and $[0.1, 0.9]$, respectively.
\vspace{-0.1in}
\item[UB:] Each appointment $i$ has a random service duration $u_i = 2\beta_i$, where $\beta_i$ independently follows the (U-shaped) beta distribution $B(0.5,0.5)$.
\vspace{-0.1in}
\item[NG:] Each appointment $i$ has a random service duration $u_i = \phi + \gamma_i$, where $\phi \geq 0$ follows a truncated normal distribution $N(1,0.5^2)$ and
$\gamma_i$ independently follows the gamma distribution $\Gamma(\alpha, 1/\alpha)$ with $\alpha$ uniformly sampled from the interval $[0.5,1]$. This is to mimic a situation in which the service duration is influenced by both the service provider (modeled by the shared random variable $\phi$) and the appointment characteristics (modeled by each individual $\gamma_i$).
\end{itemize}
It can be observed that the mean value of the service duration equals $1$ in LN and UB, and is larger than $2$ in NG. Accordingly, we set the time limit $T = 15$ in the LN and UB instances, and set $T = 30$ in the NG instances. Finally, we set ${\bm u}^{\tiny \mbox{L}}$ and ${\bm u}^{\tiny \mbox{U}}$ such that for all $i \in [n]$:
\[
u_i^{\tiny \mbox{L}} := \min_{j \in [N]}\{ \widehat{u}_i^j\} \ \ \ \text{and}  \ \ \ u_i^{\tiny \mbox{U}} := \max_{j \in [N]}\{ \widehat{u}_i^j\},
\]
where $\{\widehat{\bm u}^j\}_{j=1}^N$ are the generated data sets.

\subsubsection{Calibration of the Wasserstein ball radius}
Our first experiment is to investigate the impact of the Wasserstein ball radius $\epsilon$ on the out-of-sample performance of the \eqref{equ:wdro} optimal solution, denoted by $\bm{\widehat s}(\epsilon, N)$, with respect to the data size $N$. Specifically, we evaluate the out-of-sample performance by computing
\begin{equation*}
\Ep_{\PP_{\text{approx}}}\Big[f(\bm{\widehat s}(\epsilon, N), \bm u)\Big],
\end{equation*}
where $\PP_{\text{approx}}$ represents an empirical distribution over a set of $100,000$ samples independently drawn from the true distribution $\mathbb{P}_{\bm u}$. In addition, we consider a discrete set of values $\Omega :=\big\{ 0.01, 0.02, \cdots, 0.1, \cdots, 1, \cdots, 10 \big\}$ for the selection of $\epsilon$. For each $\epsilon \in \Omega$, we randomly sample $30$ data sets of size $N \in \{5, 50, 500\}$ from each of the three distributions LN, UB, and NG. We solve an instance of the model \eqref{equ:wdro} via its LP reformulation~\eqref{equ:lpdro} for each of the generated data sets and each of the candidate Wasserstein radius $\epsilon$.

Figure~\ref{fig:epsilon} visualizes the impact of $\epsilon$ on the out-of-sample performance of $\bm{\widehat s}(\epsilon, N)$, which are derived over the data sets generated from distribution LN. Specifically, Figure~\ref{fig:epsilon} illustrates the tubes between the 20th and 80th percentiles (shaded areas) and the mean values (solid lines) of the out-of-sample performance $Z(\bm{\widehat s}(\epsilon,N))$ as a function of $\epsilon$. The percentiles and mean values are estimated over the 30 independent simulation runs. We observe that the out-of-sample performance improves up to a critical $\epsilon$ value and then deteriorates. Hence, there exists a Wasserstein radius $\epsilon_{\text{best}}$ such that the corresponding optimal distributionally robust solutions have the lowest (i.e., best) out-of-sample performance. We note that same trends are observed from the data sets generated by distributions UB and NG.

In practice, however, a large data set is usually unavailable to construct $\PP_{\text{approx}}$ and thus seeking $\epsilon_{\text{best}}$ by computing the out-of-sample performance is not viable. In this paper, we implement a cross validation method that mimics the above out-of-sample evaluation procedure to approximate $\epsilon_{\text{best}}$ based on the $N$ in-sample data. More specifically, we randomly partition the data $\{\widehat{\bm u}^j\}_{j=1}^N$ into two parts: a training data set consisting of $0.8N$ data and a validation data set consisting of the remaining $0.2N$ data. Using only the training data, we solve \eqref{equ:wdro} to obtain optimal solutions $\bm{\widehat s}(\epsilon, 0.8N)$ for each $\epsilon \in \Omega$. Then, we evaluate these solutions by computing $\Ep_{\PP_{0.2N}}[f(\bm{\widehat s}(\epsilon, 0.8N), \bm u)]$, where $\PP_{0.2N}$ is the empirical distribution based only on the validation data, and we set $\widehat{\epsilon}^N_{\text{best}}$ to any $\epsilon$ that minimizes this quantity, i.e., $\widehat{\epsilon}^N_{\text{best}} \in \displaystyle\mbox{arg min}_{\epsilon \in \Omega}\{\Ep_{\PP_{0.2N}}[f(\bm{\widehat s}(\epsilon, 0.8N), \bm u)]\}$. Finally, we repeat this procedure for 30 random partitions and set $\epsilon$ to the average of the $\widehat{\epsilon}^N_{\text{best}}$ obtained from these 30 partitions.


\begin{figure} [t]
	\begin{center}
		\includegraphics[width=0.32\linewidth]{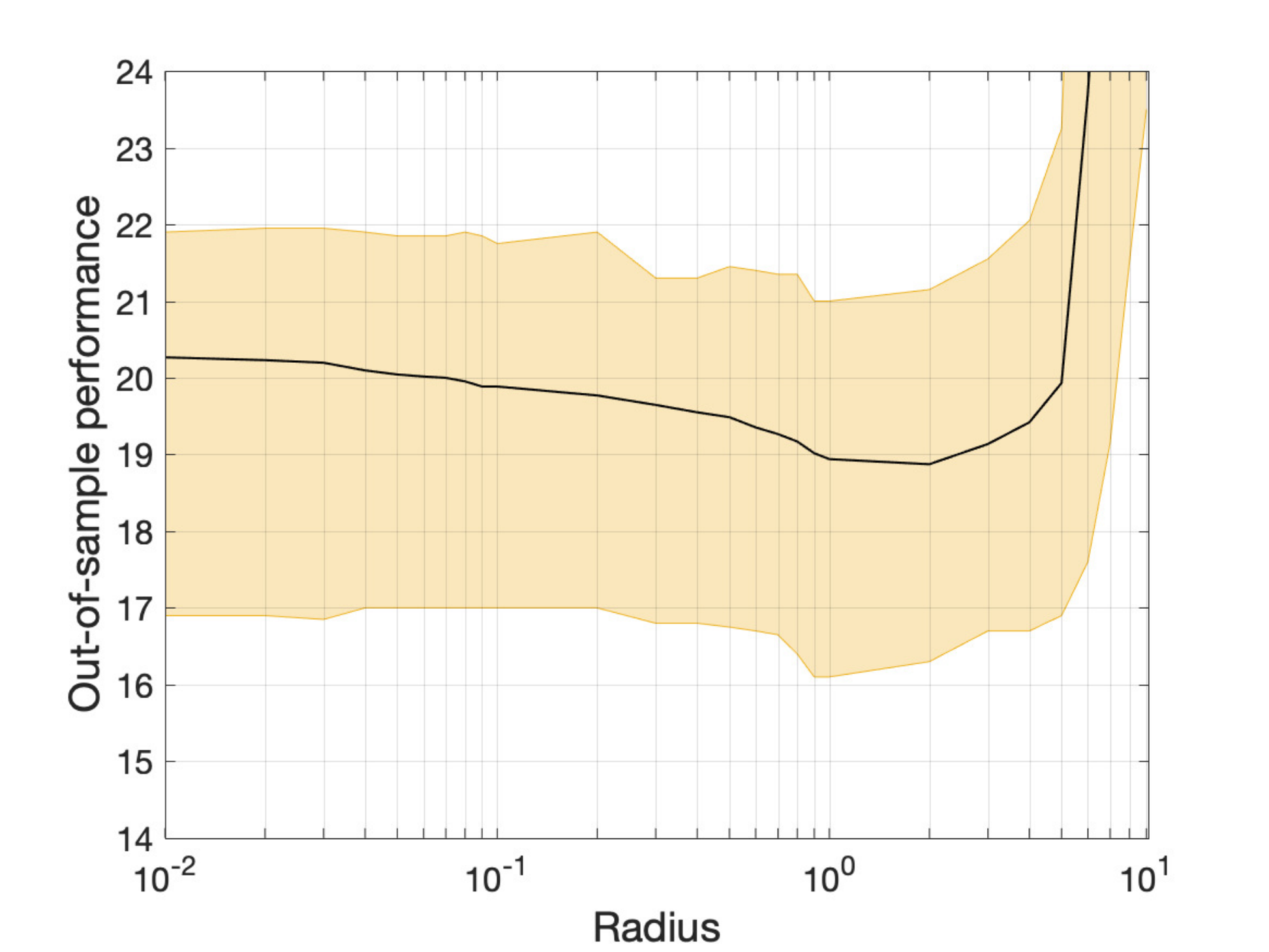}
		\includegraphics[width=0.32\linewidth]{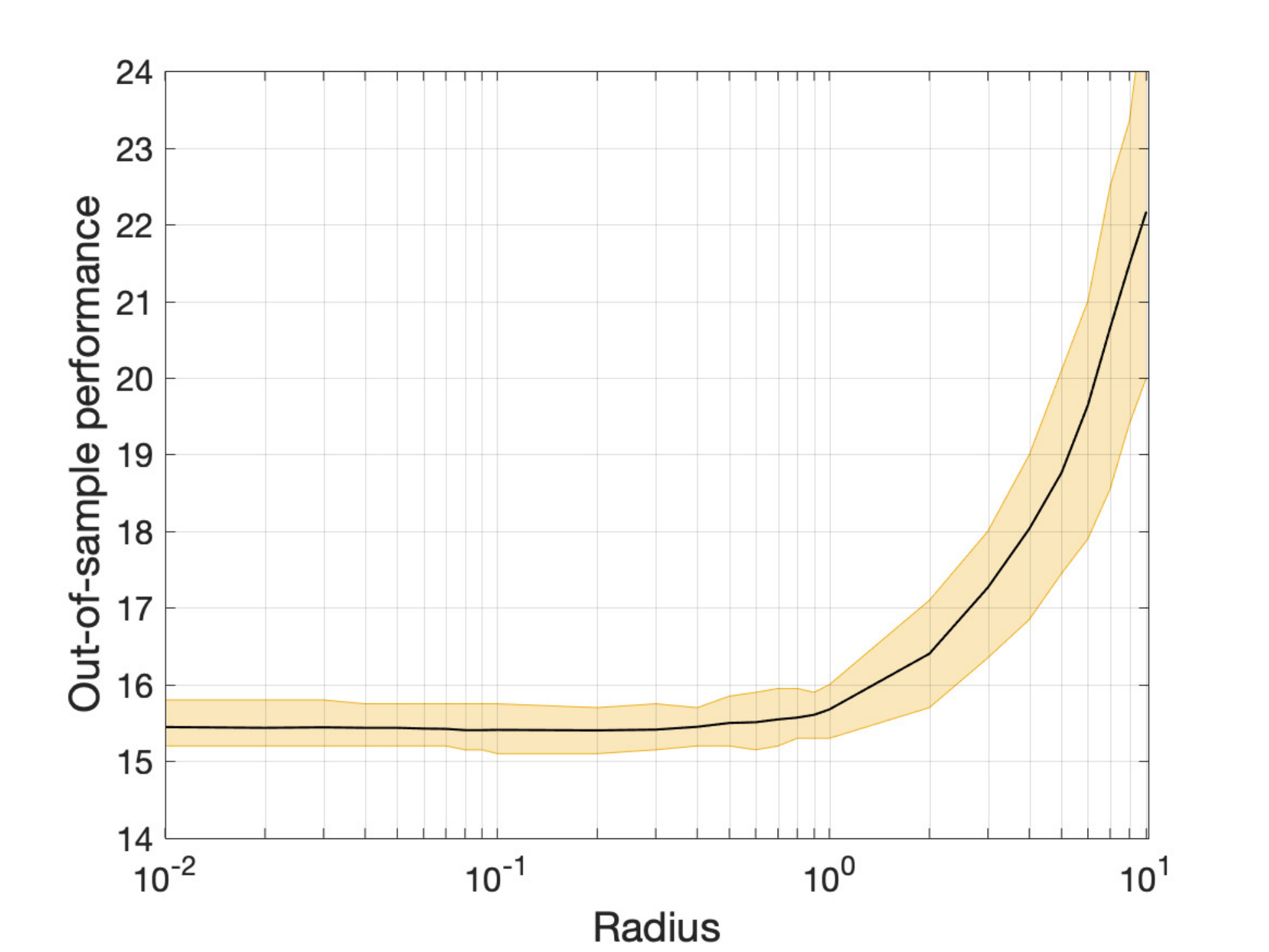}
		\includegraphics[width=0.32\linewidth]{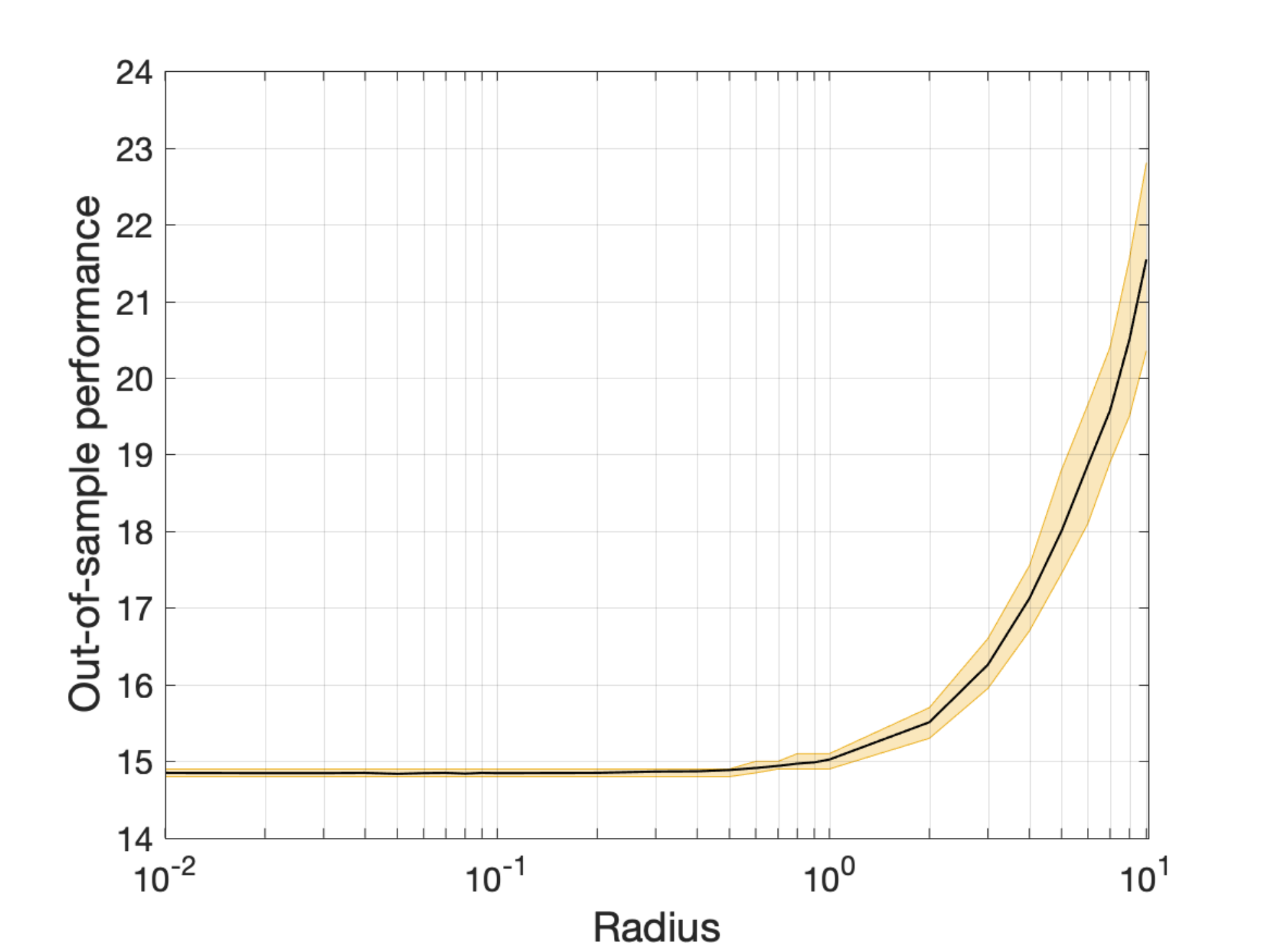}
		\caption{Out-of-sample performance as a function of the Wasserstein radius $\epsilon$.
		The data sets are generated from distribution LN. Sample size: $N=5$ (left), $N=50$ (middle), and $N=500$ (right).}
						\label{fig:epsilon}
	\end{center}
\end{figure}

\subsubsection{Out-of-sample performance}
We compare the out-of-sample performance of the \eqref{equ:wdro} approach with that of a cross-moment (CM) distributionally robust approach, in which the ambiguity set is characterized by the mean, variance, and correlation information~\cite{Kong.Lee.Teo.Zheng.2013}. This moment information is estimated from the data samples $\{\widehat{\bm u}^j\}_{j=1}^N$. We obtain optimal appointment schedules of the CM approach by solving the semidefinite programming approximations; see details in~\cite{Kong.Lee.Teo.Zheng.2013}. In addition, we compare with a sample average approximation (SAA) approach that solves model \eqref{equ:sas} with $\mathbb{P}_{\bm u}$ replaced by the empirical distribution $(1/N)\sum_{j=1}^N \delta_{\widehat{u}^j}$. We generate data sets of size $N \in \{5, 10, 50, 100, 500, 1000\}$ from each of the distributions LN, UB, and NG. For each of the generated data sets, we solve an instance of \eqref{equ:wdro} with an $\epsilon$ set in the cross validation method, an instance of CM approach, and an instance of SAA approach.

	\begin{figure}[t]
		\begin{subfigure}[t]{0.32\textwidth}
			\includegraphics[width=\linewidth]{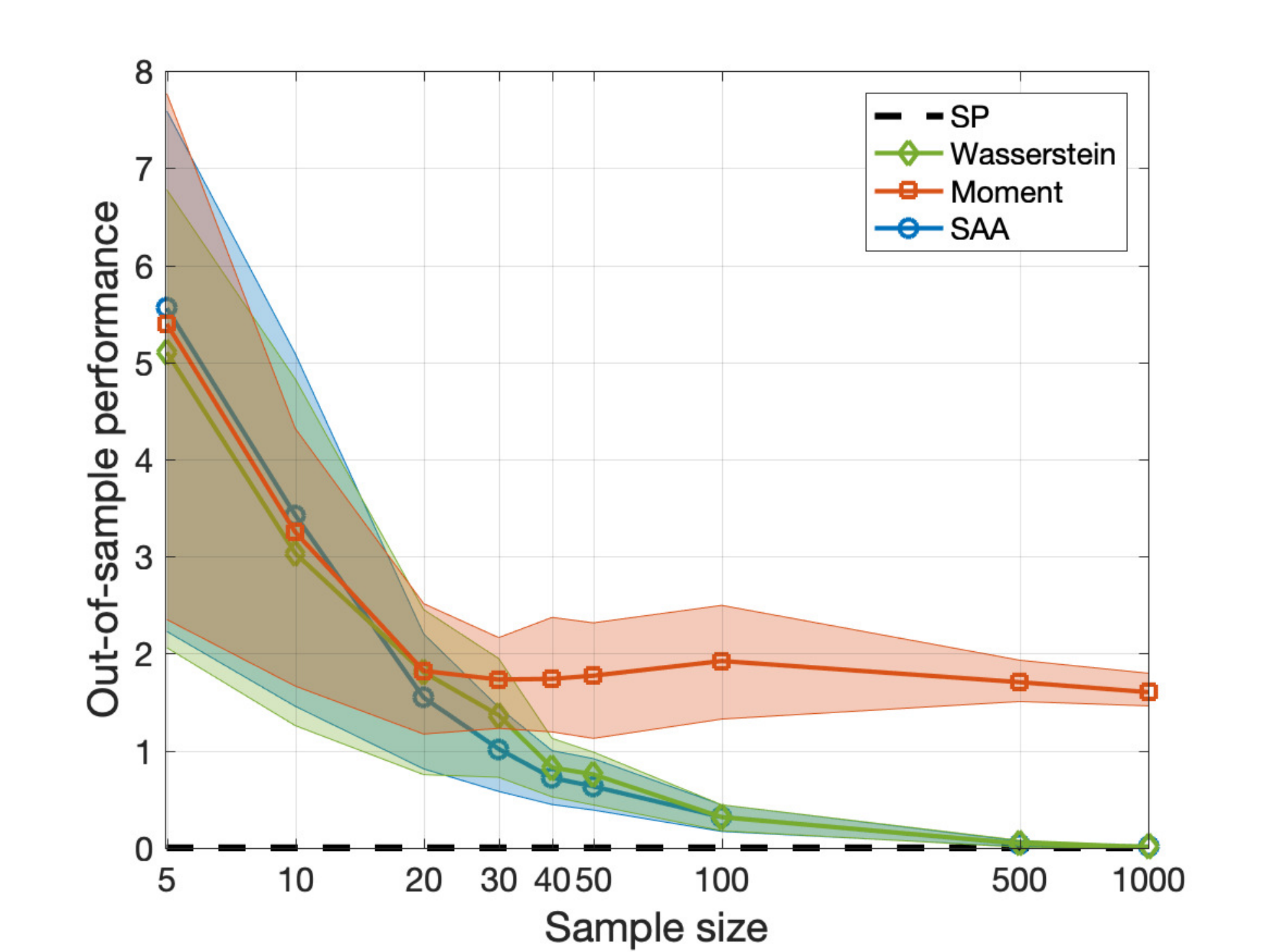}
			\caption{LN}
			\label{fig:lnosp}
		\end{subfigure}%
		\begin{subfigure}[t]{0.32\textwidth}
			\includegraphics[width=\linewidth]{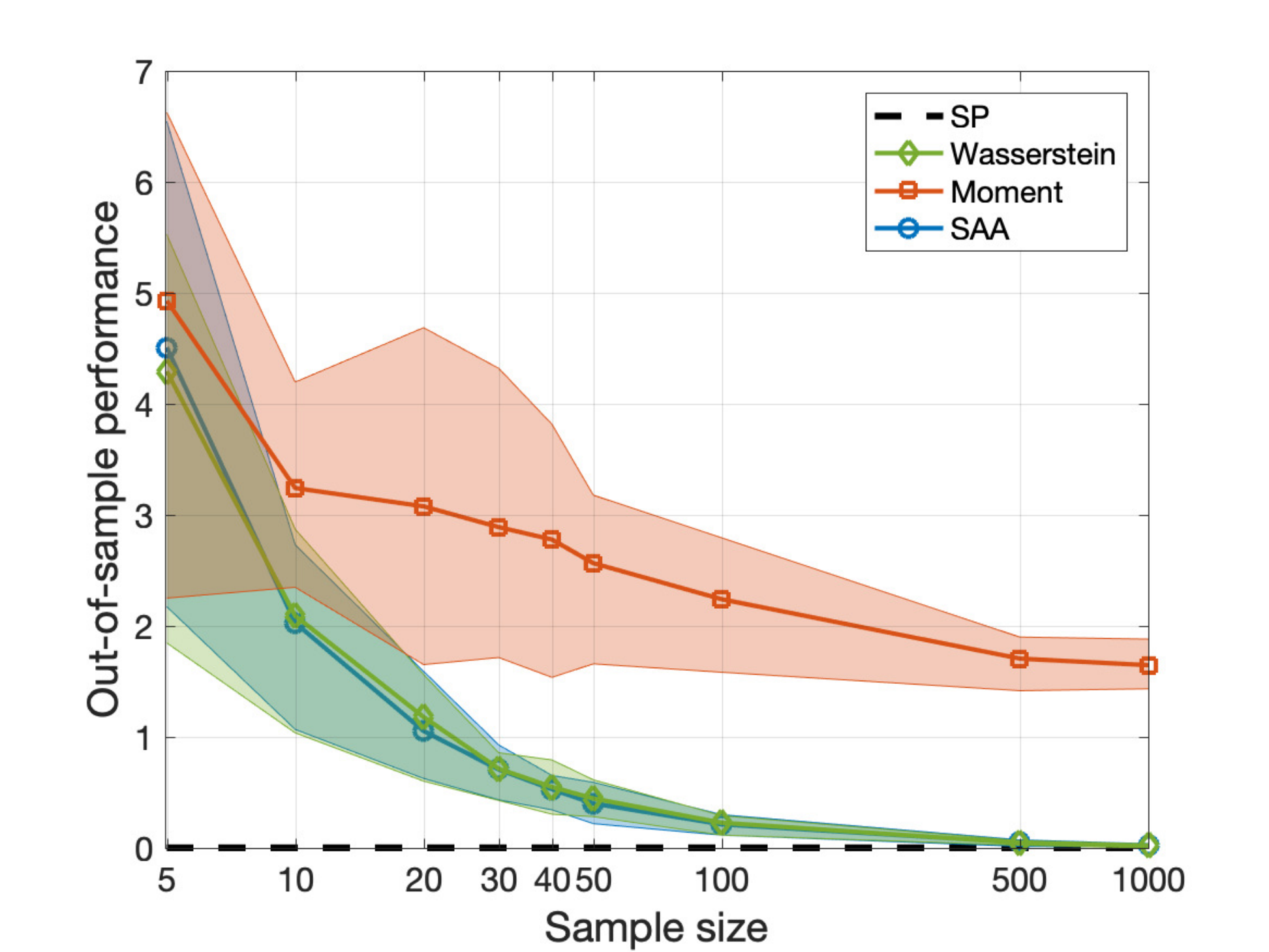}
			\caption{UB}
			\label{fig:btosp}
		\end{subfigure}%
		\begin{subfigure}[t]{0.32\textwidth}
			\includegraphics[width=\linewidth]{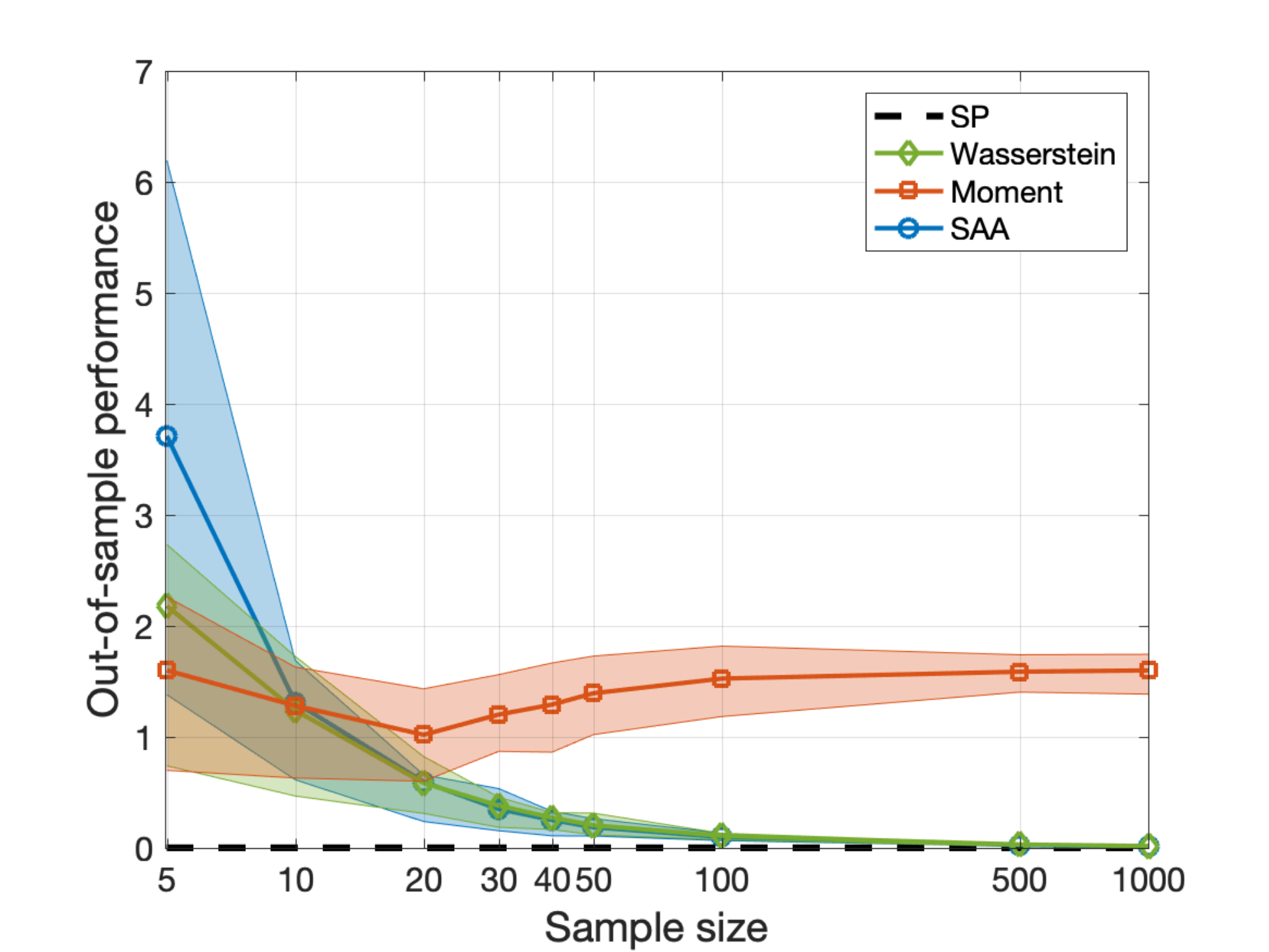}
			\caption{NG}
			\label{fig:gnosp}
		\end{subfigure}%
		\caption{Out-of-sample performance of optimal W-DRAS, CM, and SAA appointment schedules as a function of data size $N$}
		\label{fig:osp}
	\end{figure}
Figure~\ref{fig:osp} displays the tubes between the 20th and 80th percentiles (shaded areas) and the mean values (solid lines) of the out-of-sample performance as a function of the sample size $N$. The percentiles and mean values are estimated over $30$ independent simulation runs. Note that the out-of-sample performance presented in Figure~\ref{fig:osp} are estimated by using the optimal solutions that minimize the W-DRAS, CM, and SAA problems, respectively. The horizontal dashed line represents $Z^\star$, the optimal value of the stochastic appointment scheduling model \eqref{equ:sas}, in which $\mathbb{P}_{\bm u}$ is replaced with an empirical distribution based on $10,000$ scenarios.\footnote{Note that, for the convenience of making comparisons, we shifted the vertical coordinate of all points downwards by $Z^\star$ in Figure~\ref{fig:osp}. As a consequence, the horizontal dashed line for $Z^\star$ appears with zero vertical coordinate, and a point with vertical coordinate $1$, for example, represents an out-of-sample average cost higher than $Z^\star$ by $1$ unit. We applied similar shifting operation in Figures~\ref{fig:osp-small},~\ref{fig:mis_osp},~\ref{fig:lnosp}--\ref{fig:ngosp-small}, and~\ref{fig:ns_mis_osp}.} From Figure~\ref{fig:osp}, we observe that the out-of-sample performance of W-DRAS and SAA converge to $Z^\star$, while that of CM does not. This is consistent with the theoretical results in Theorem \ref{thm:a-c}, confirming that the \eqref{equ:wdro} approach enjoys the asymptotic consistency. In contrast, the CM approach relies on the first two moments of the service durations and so the asymptotic consistency cannot be guaranteed in general.

	\begin{figure}[htb]
		\begin{subfigure}[t]{0.32\textwidth}
			\includegraphics[width=\linewidth]{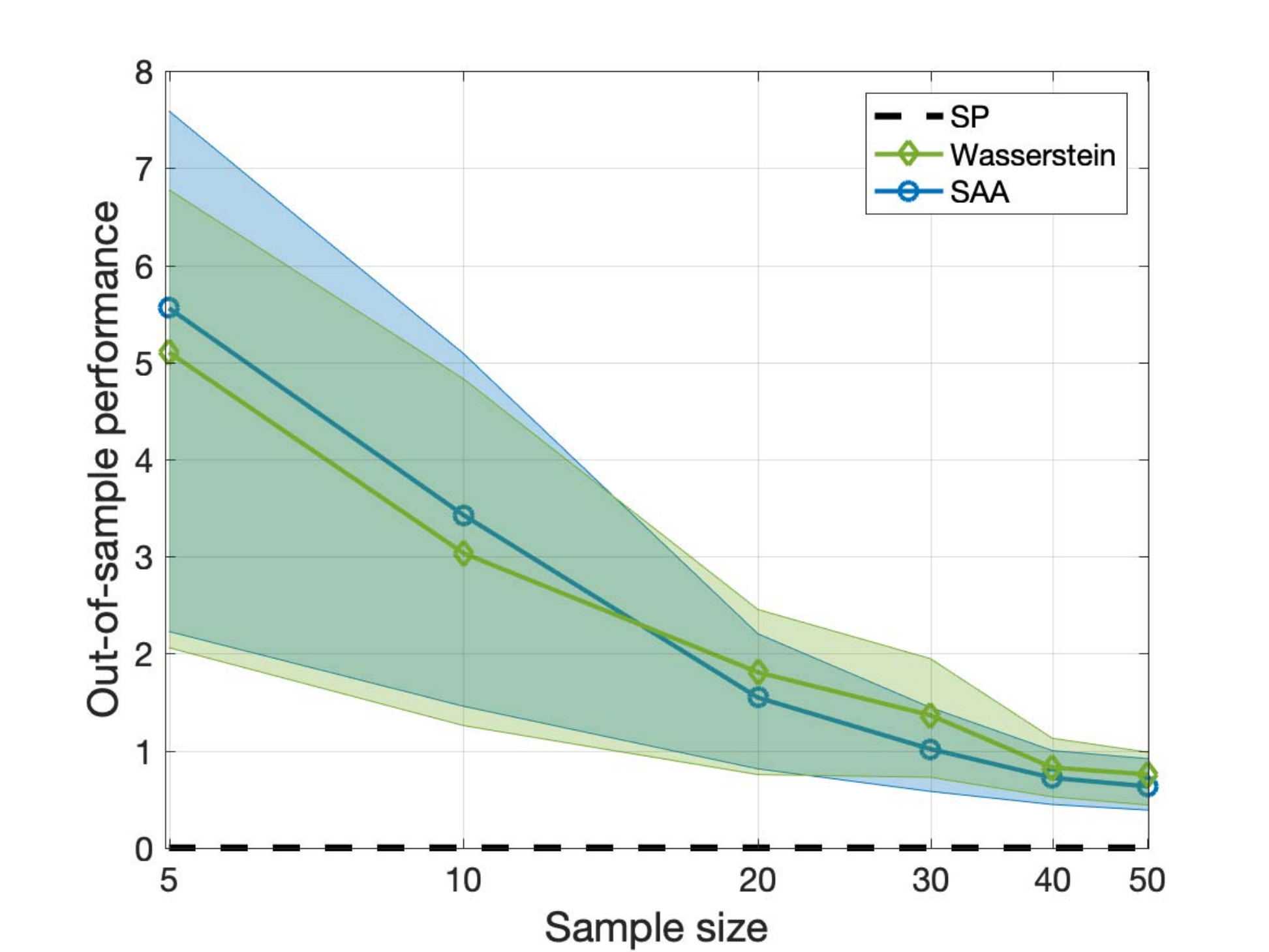}
			\caption{LN}
			\label{fig:lnosp-small}
		\end{subfigure}%
		\begin{subfigure}[t]{0.32\textwidth}
			\includegraphics[width=\linewidth]{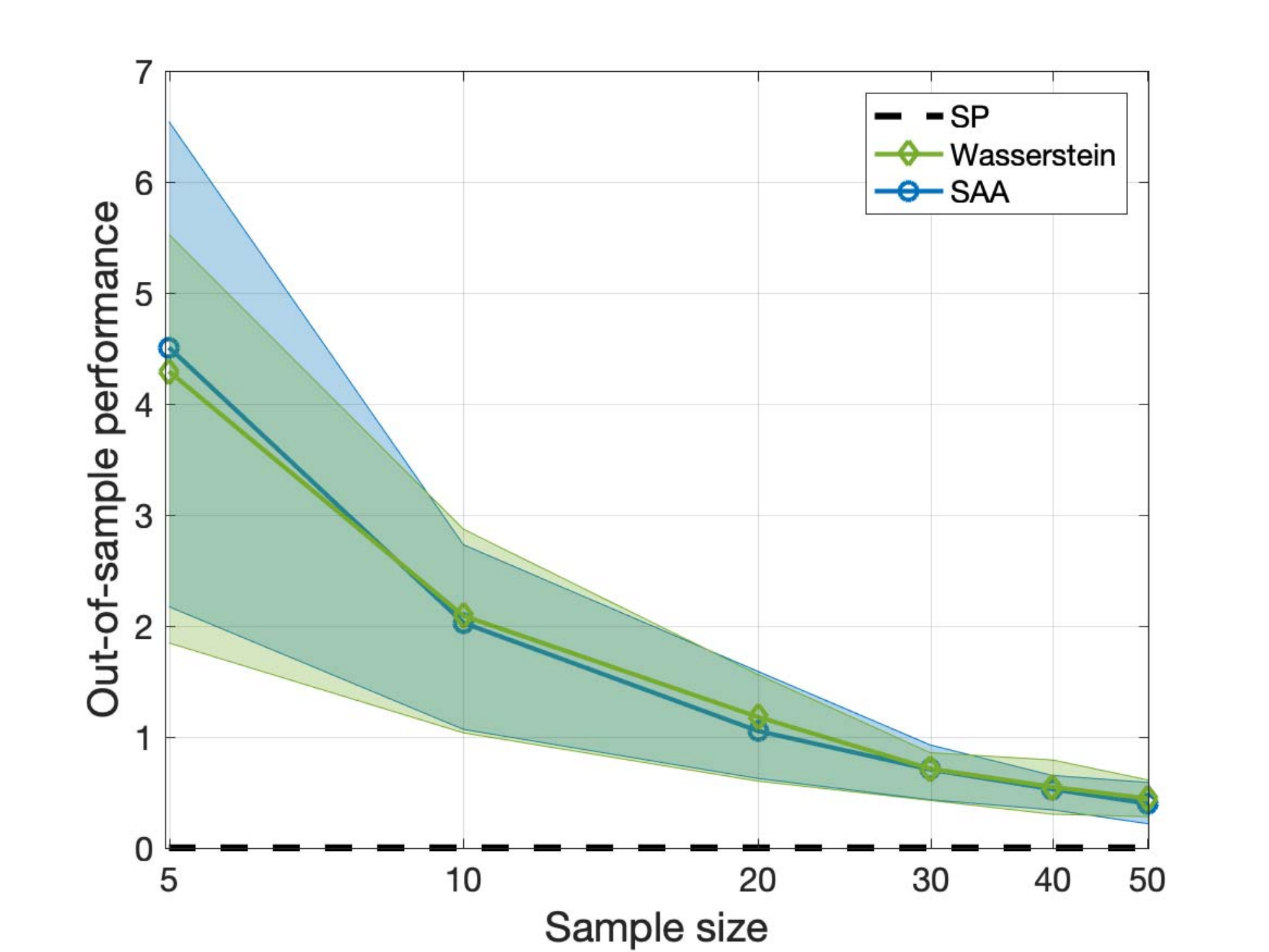}
			\caption{UB}
			\label{fig:btosp-small}
		\end{subfigure}%
		\begin{subfigure}[t]{0.32\textwidth}
			\includegraphics[width=\linewidth]{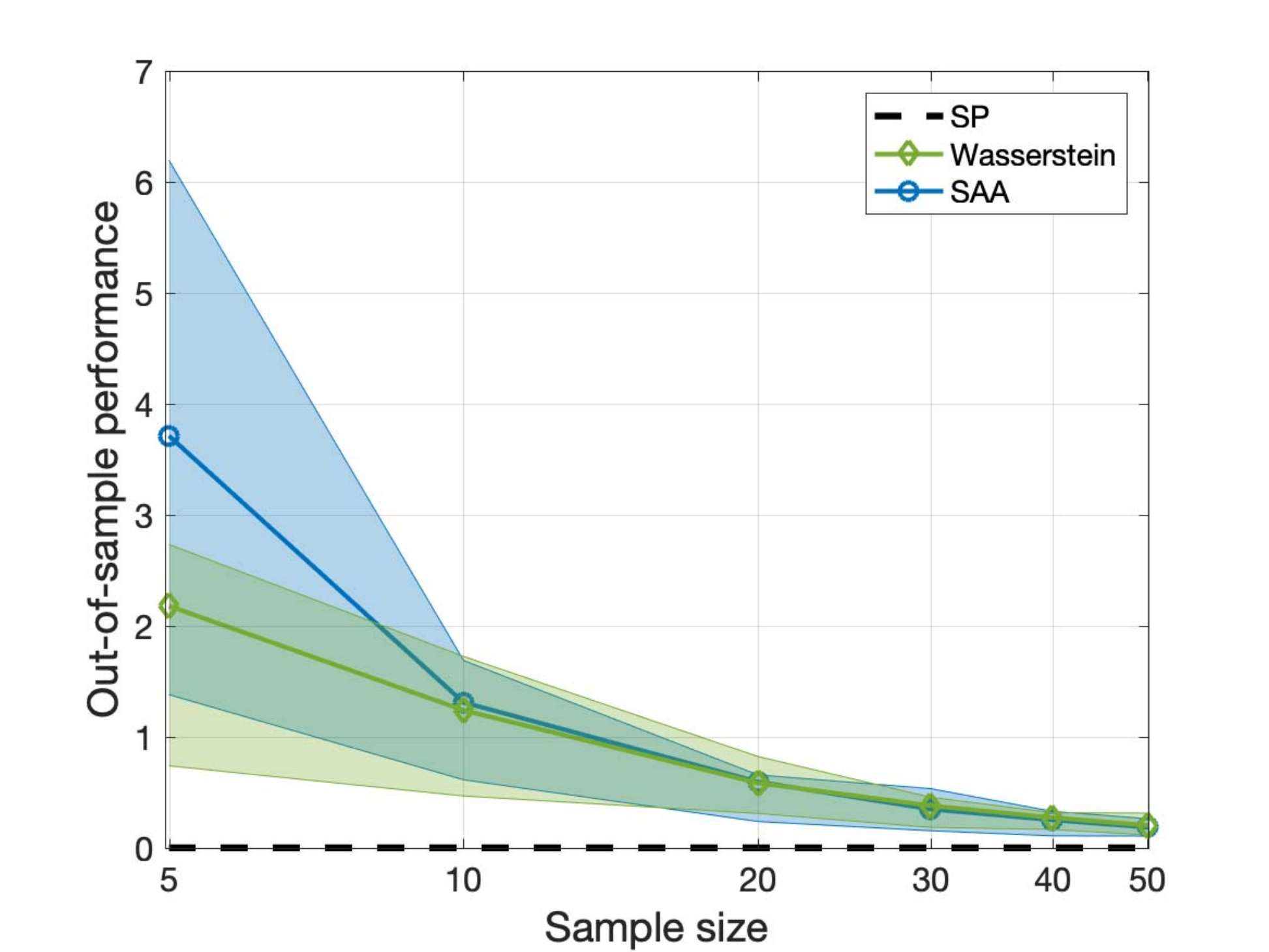}
			\caption{NG}
			\label{fig:gnosp-small}
		\end{subfigure}%
		\caption{Out-of-sample performance of the W-DRAS and SAA optimal appointment schedules with small data sizes}
		\label{fig:osp-small}
	\end{figure}
In addition, we compare the out-of-sample performance of the W-DRAS and SAA approaches in Figure~\ref{fig:osp-small}. From this figure, we observe that W-DRAS (slightly) outperforms SAA. Intuitively, this demonstrates that W-DRAS is capable to effectively learn distributional information even from a very limited amount of data (e.g., when $N = 5$ or $10$). As a consequence, the proposed W-DRAS approach is particularly effective in AS systems with scarce service duration data.

Figure~\ref{fig:reliability} displays the reliability of the three approaches, which is the empirical probability of the event that the optimal values of W-DRAS, CM, and SAA exceed the out-of-sample performance of the corresponding optimal solutions. The empirical probability is estimated over 30 independent simulation runs. From Figure~\ref{fig:reliability}, we observe that the reliability of W-DRAS and CM is consistently higher than that of SAA under all tested data sizes and across all tested generating distributions. For example, the reliability of CM increases to 100\% once $N$ exceeds 30, and that of W-DRAS is higher than 70\% in most instances. In contrast, the reliability of SAA is generally lower than 50\%, unless when $N$ becomes large (e.g., $N \geq 500$). This is consistent with the theoretical results in Theorem \ref{thm:f-d}, confirming that \eqref{equ:wdro} can provide a safe (upper bound) guarantee on the expected total cost even with a small data size. 

	\begin{figure}
		\begin{subfigure}[t]{0.32\textwidth}
			\includegraphics[width=\linewidth]{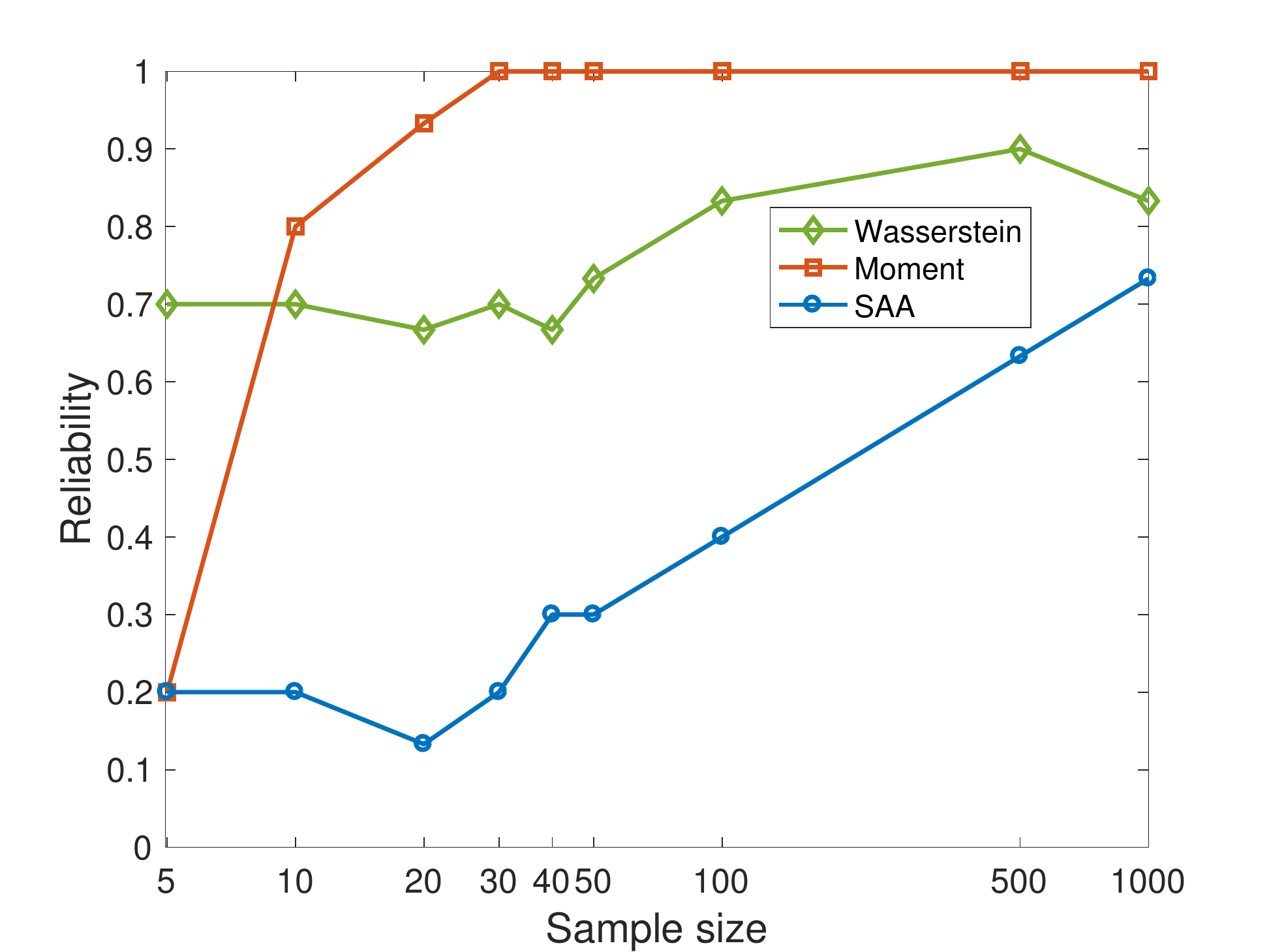}
			\caption{LN}
			\label{fig:lnrel}
		\end{subfigure}%
		\begin{subfigure}[t]{0.32\textwidth}
			\includegraphics[width=\linewidth]{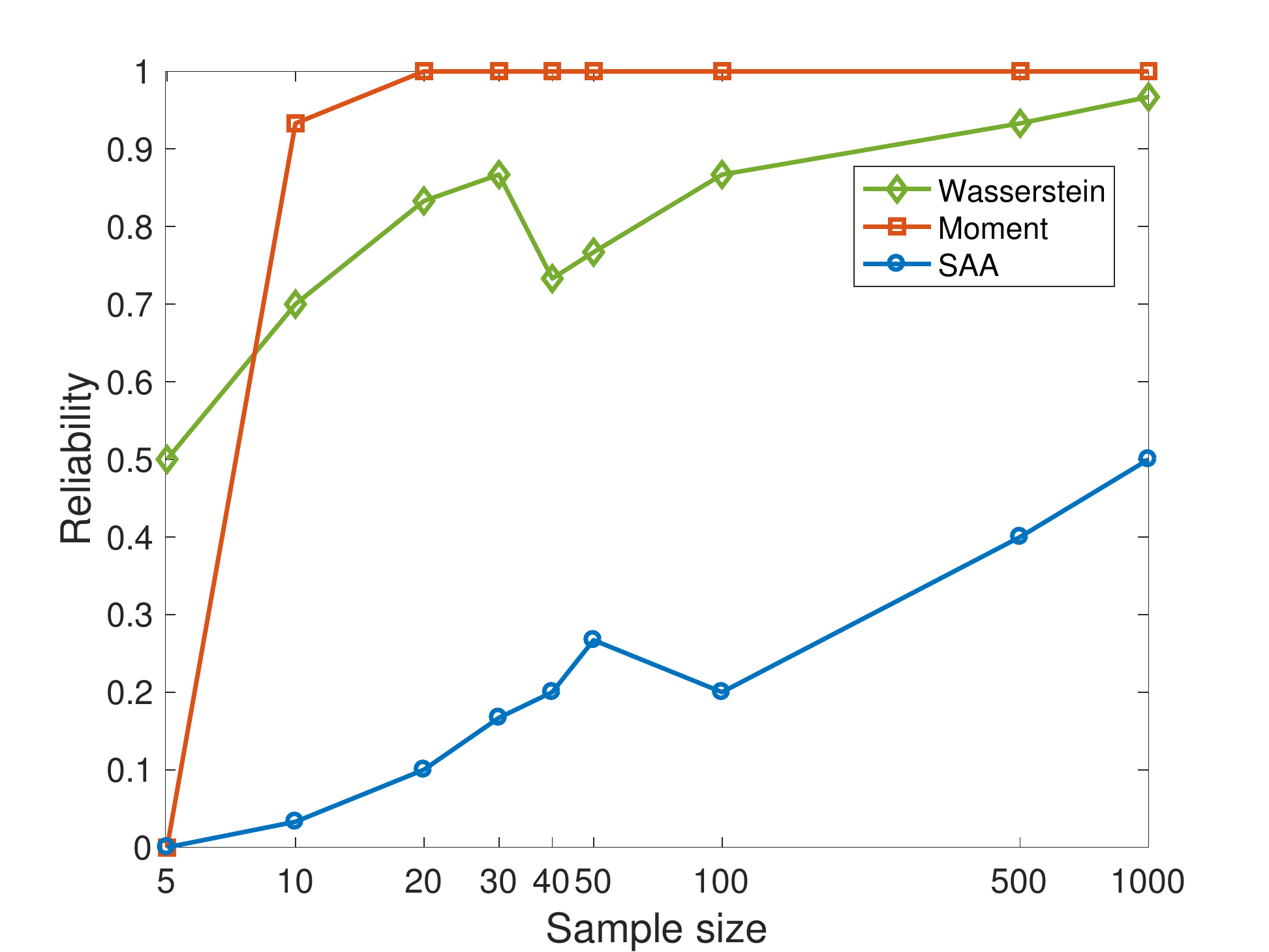}
			\caption{UB}
			\label{fig:btrel}
		\end{subfigure}%
		\begin{subfigure}[t]{0.32\textwidth}
			\includegraphics[width=\linewidth]{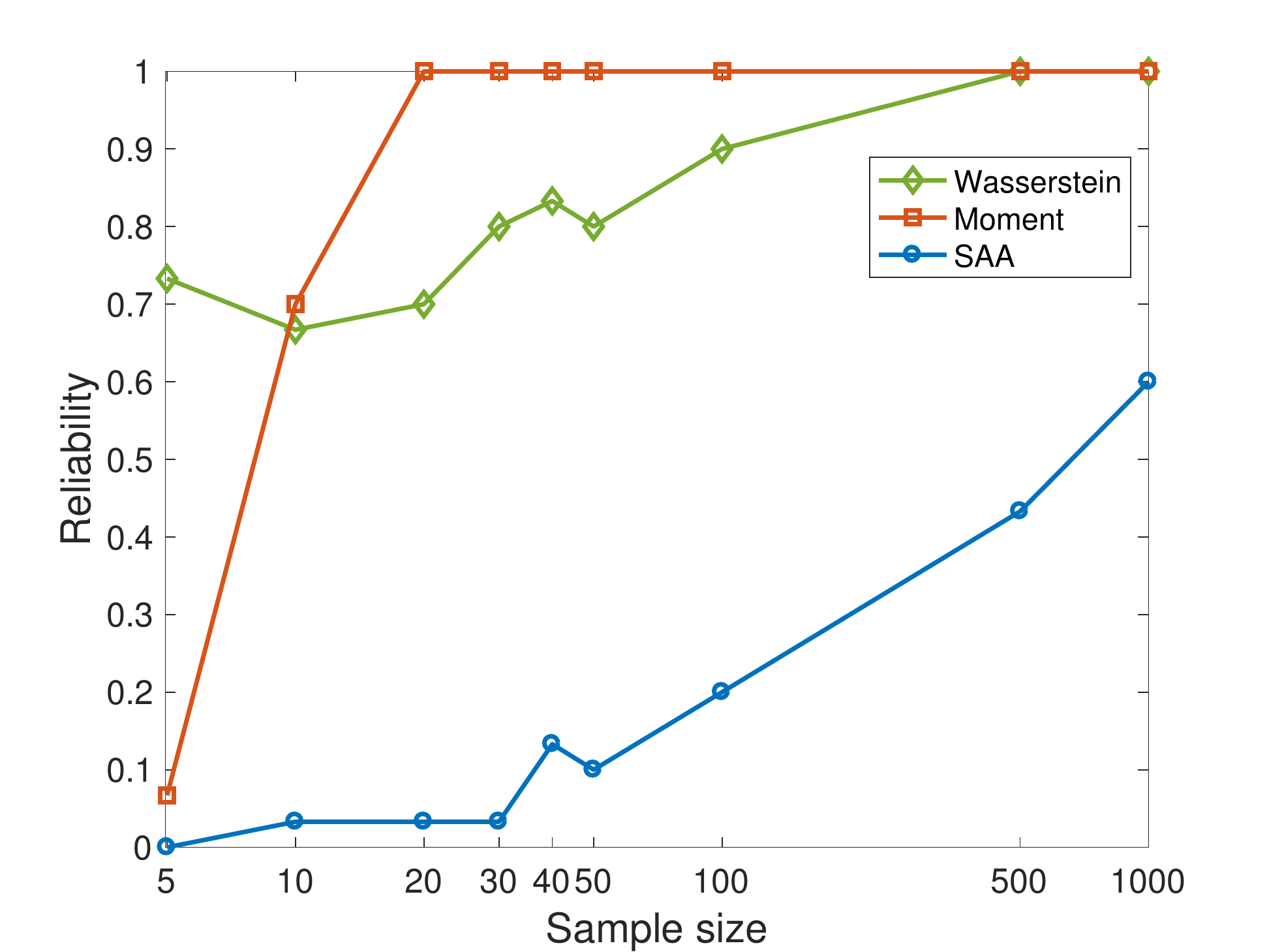}
			\caption{NG}
			\label{fig:gnrel}
		\end{subfigure}%
		\caption{Reliability of W-DRAS, CM, and SAA as a function of data size $N$}
		\label{fig:reliability}
	\end{figure}

\subsubsection{Misspecified distributions}
Another situation of interest is when the distribution of service durations in an AS system quickly varies due to, e.g., changes in service provider and/or appointee mix. As a consequence, the data we rely on to produce the appointment schedule may follow a misspecified distribution, i.e., one that is different from the true distribution. We conduct an experiment to examine the performance of the optimal W-DRAS, CM, and SAA appointment schedules by using data generated from a different distribution. Specifically, we use the same types of distributions as those generating the $N$ in-sample data, but increase or decrease their parameters by $\sigma\%$ with $\sigma$ uniformly sampled from $[5,10]$. Figure~\ref{fig:mis_osp} shows the performance of these appointment schedules under misspecified distributions. We observe that the W-DRAS and SAA approaches still outperform the CM approach even under misspecified distributions. In addition, the W-DRAS approach still (slightly) outperforms the SAA approach when the data size is small. This demonstrates that the proposed W-DRAS approach is particularly effective in AS systems in quickly varying environments.

	\begin{figure}
		\begin{subfigure}[b]{0.32\textwidth}
			\includegraphics[width=\linewidth]{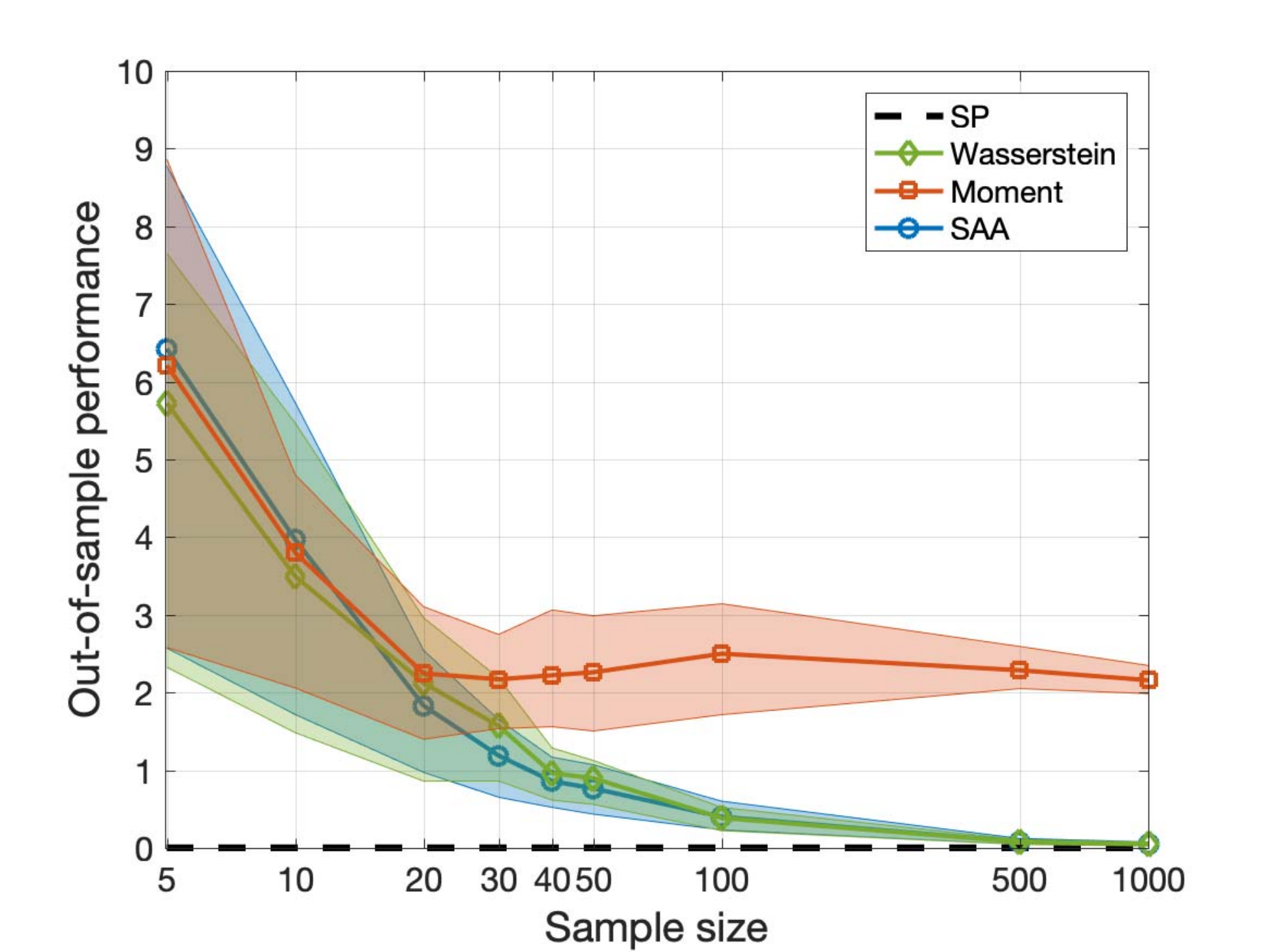}
			\caption{Misspecified LN}
		\end{subfigure}%
		\begin{subfigure}[b]{0.32\textwidth}
			\includegraphics[width=\linewidth]{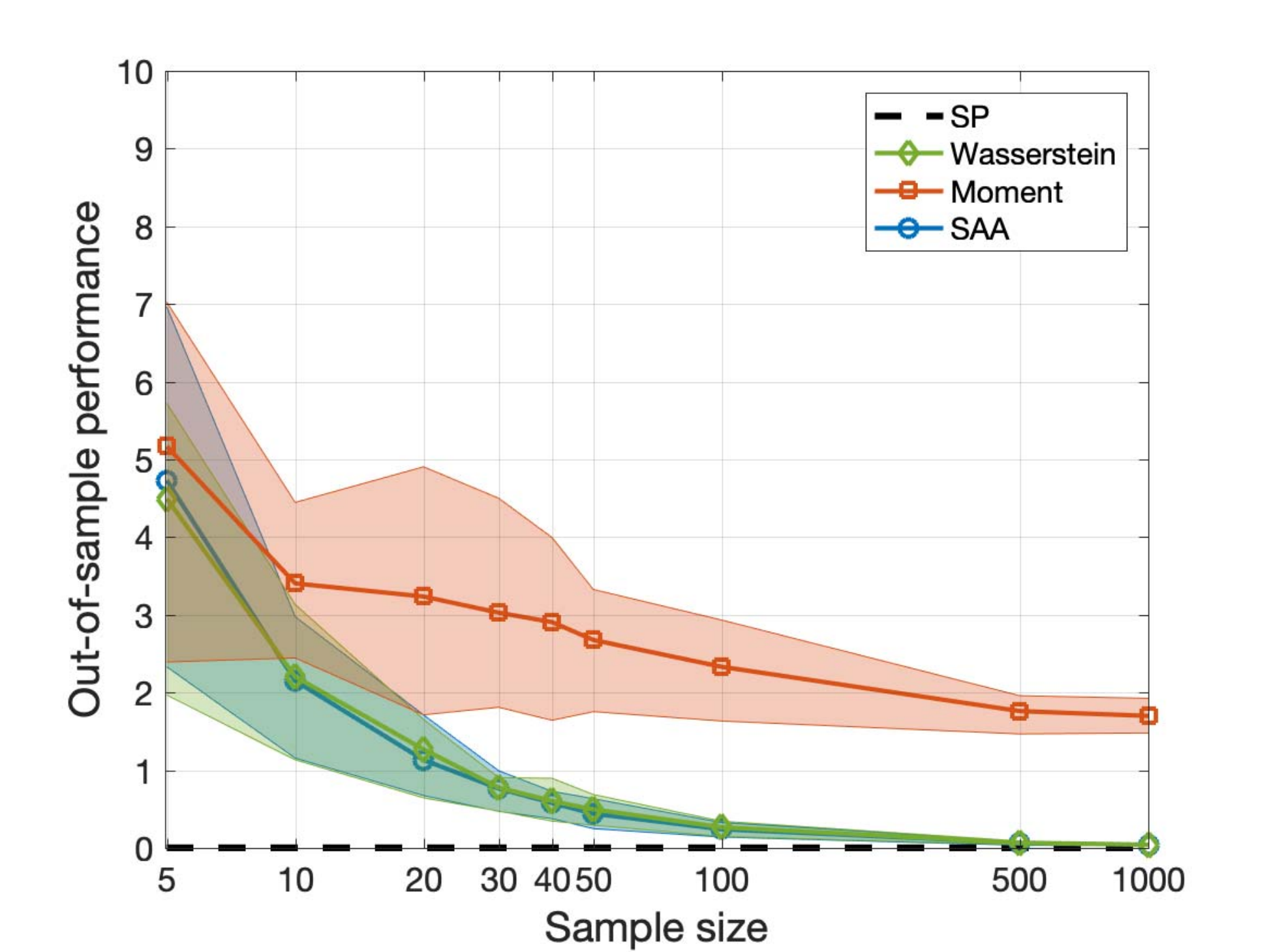}
			\caption{Misspecified UB}
		\end{subfigure}
		\quad
		\begin{subfigure}[b]{0.32\textwidth}
			\includegraphics[width=\linewidth]{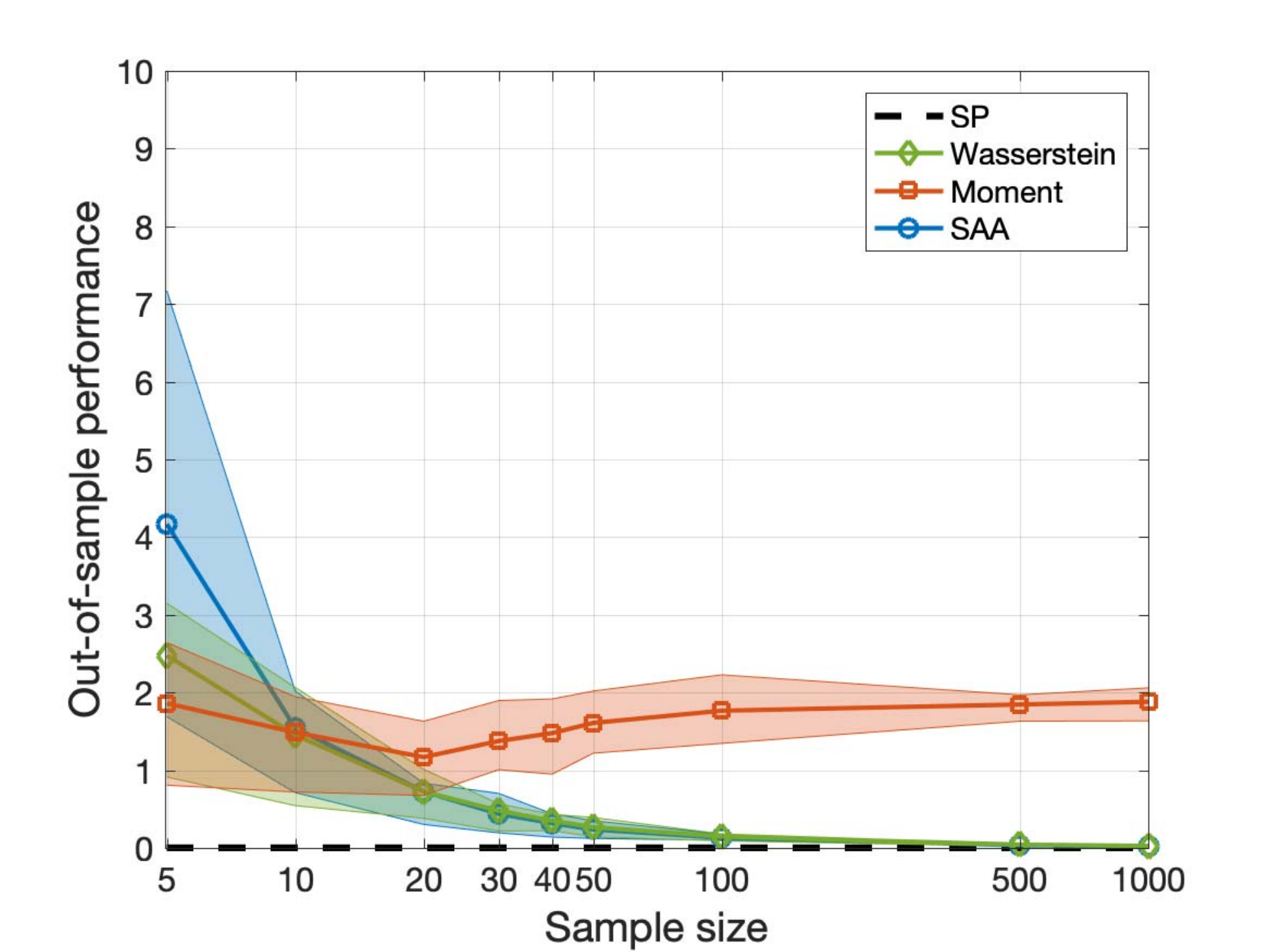}
			\caption{Misspecified NG}
		\end{subfigure}%
		\caption{Out-of-sample performance of the optimal W-DRAS, CM, and SAA appointment schedules under misspecified distributions}
		\label{fig:mis_osp}
	\end{figure}

\subsection{Random no-shows and service durations}
We conduct numerical experiments to test the \eqref{equ:dras-ns} model discussed in Section~\ref{sec:no_show}, where both no-shows and service durations are random. We consider the same distributions (LN, UB, and NG) for the service durations as in Section~\ref{sec:exp_rs}. Meanwhile, we employ a Bernoulli distribution with parameter $0.4$ for no-shows (i.e., each appointment does not show up with a probability of $0.4$). In addition, we implement the same cross-validation method to calibrate the Wasserstein ball radius.

We compare the out-of-sample performance of our W-NS approach with a marginal-moment (MM) distributionally robust approach, which characterizes the ambiguity set based on the mean and support information of the random no-shows and service durations (see~\cite{Jiang.Shen.Zhang.2017}). In addition, as in Section~\ref{sec:exp_rs}, we compare with the simple SAA approach.

    \begin{figure}
		\begin{subfigure}[t]{0.32\textwidth}
			\includegraphics[width=\linewidth]{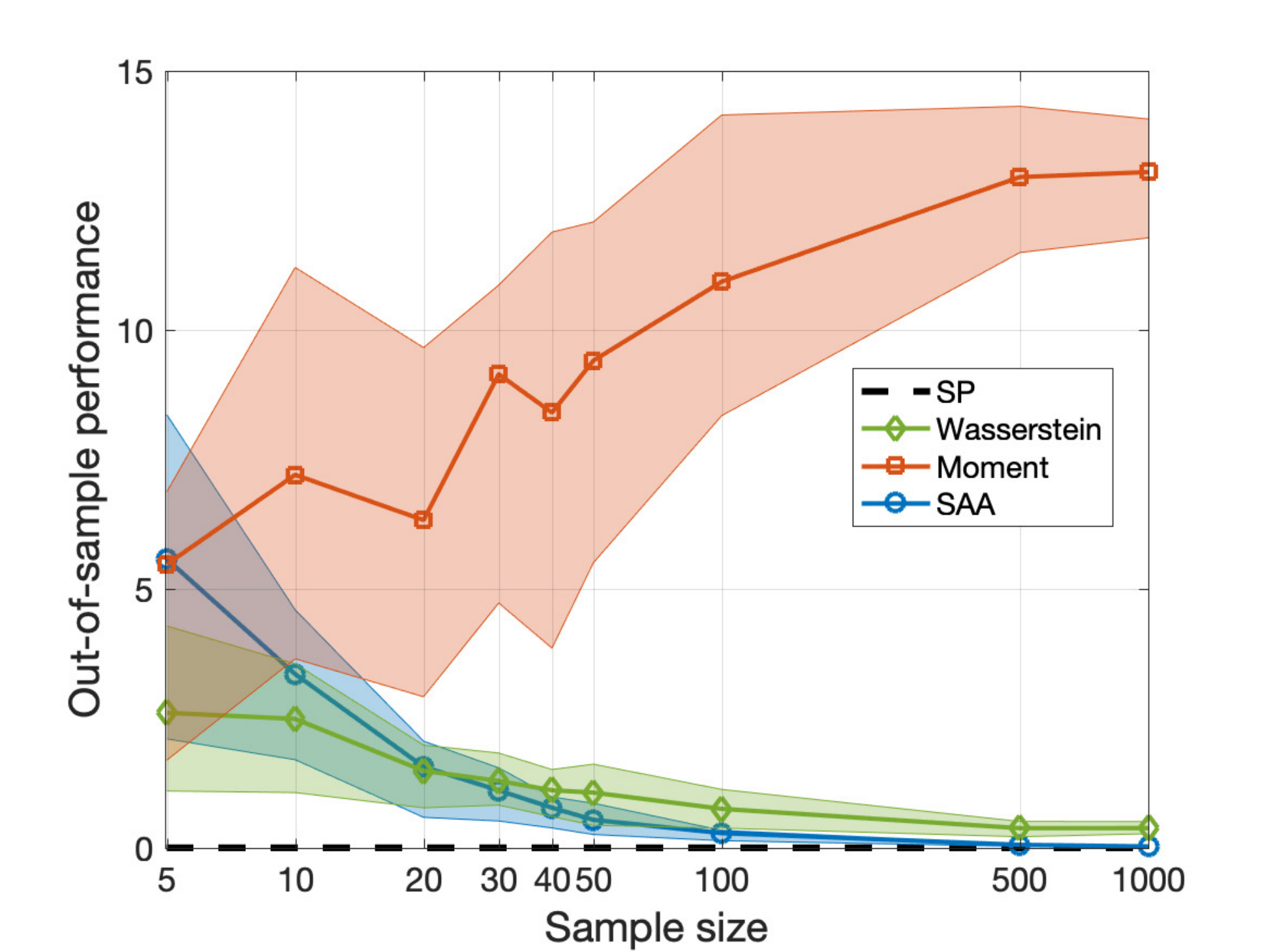}
			\caption{LN, Out-of-sample}
			\label{fig:lnosp}
		\end{subfigure}%
		\hfill
		\begin{subfigure}[t]{0.32\textwidth}
            \includegraphics[width=\linewidth]{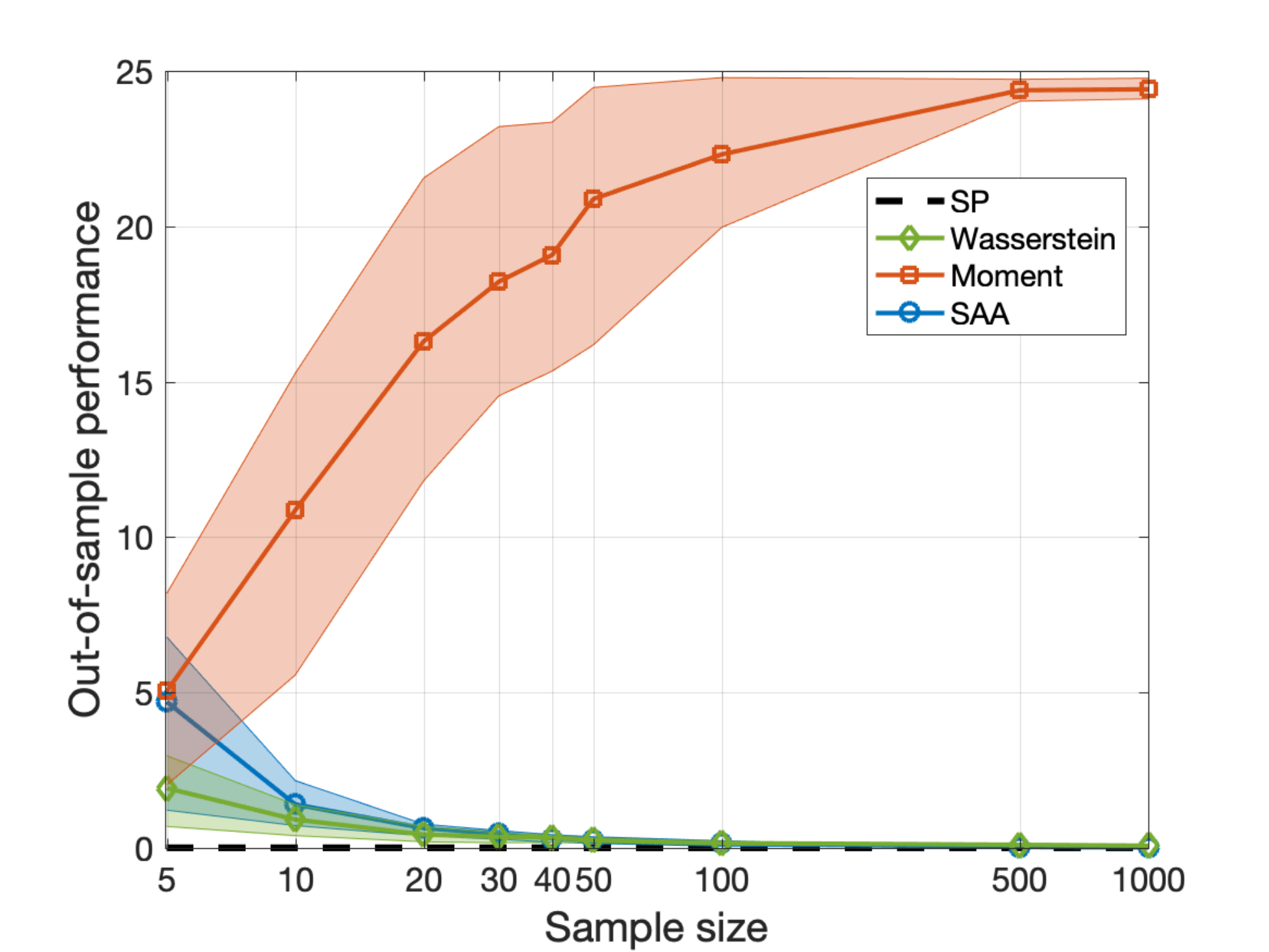}
			\caption{UB, Out-of-sample}
			\label{fig:ubosp}
		\end{subfigure} %
		\hfill
		\begin{subfigure}[t]{0.32\textwidth}
            \includegraphics[width=\linewidth]{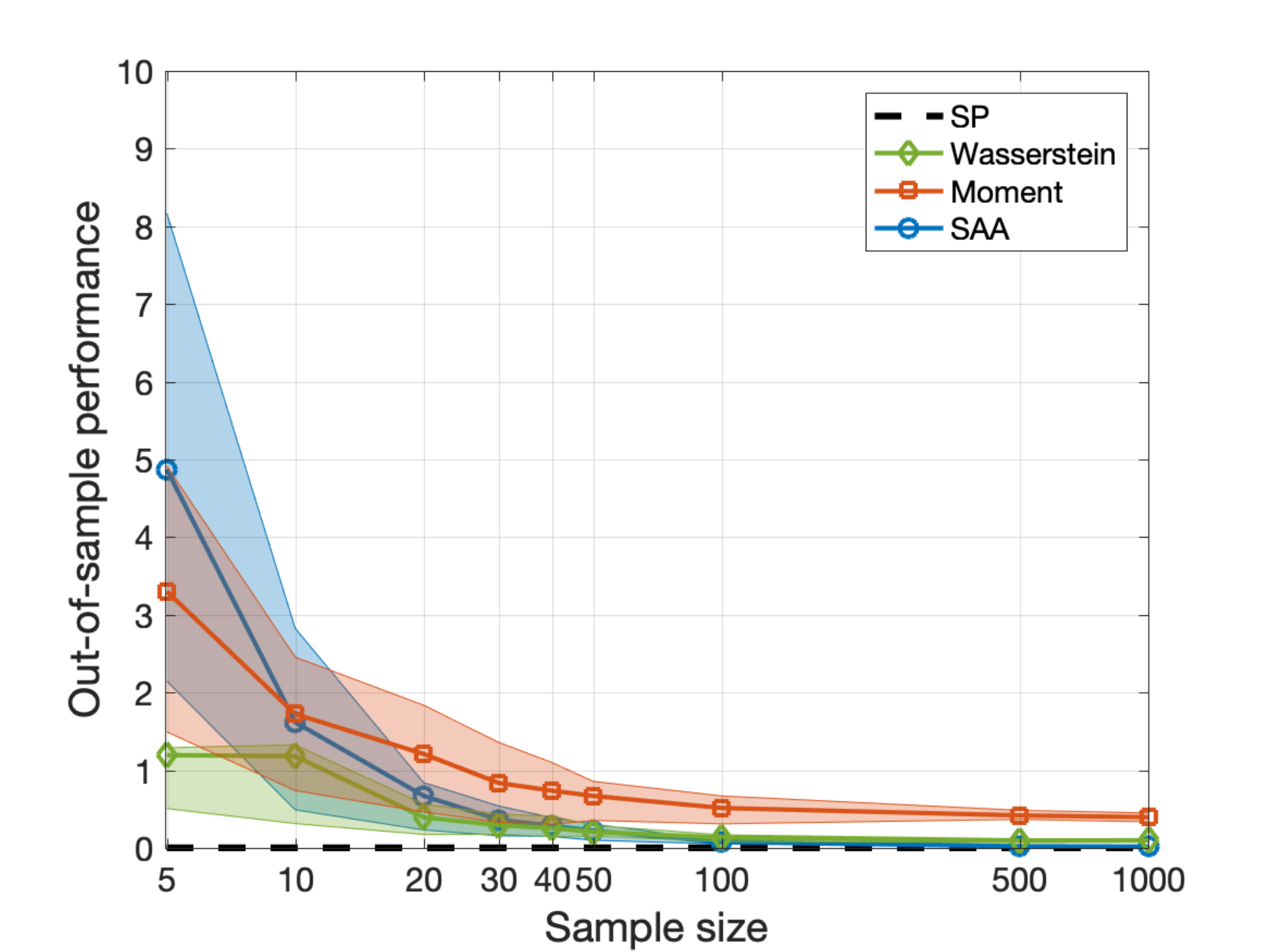}
			\caption{NG, Out-of-sample}
			\label{fig:ngosp}
		\end{subfigure}%
		\hfill
		\quad
		\begin{subfigure}[t]{0.32\textwidth}
            \includegraphics[width=\linewidth]{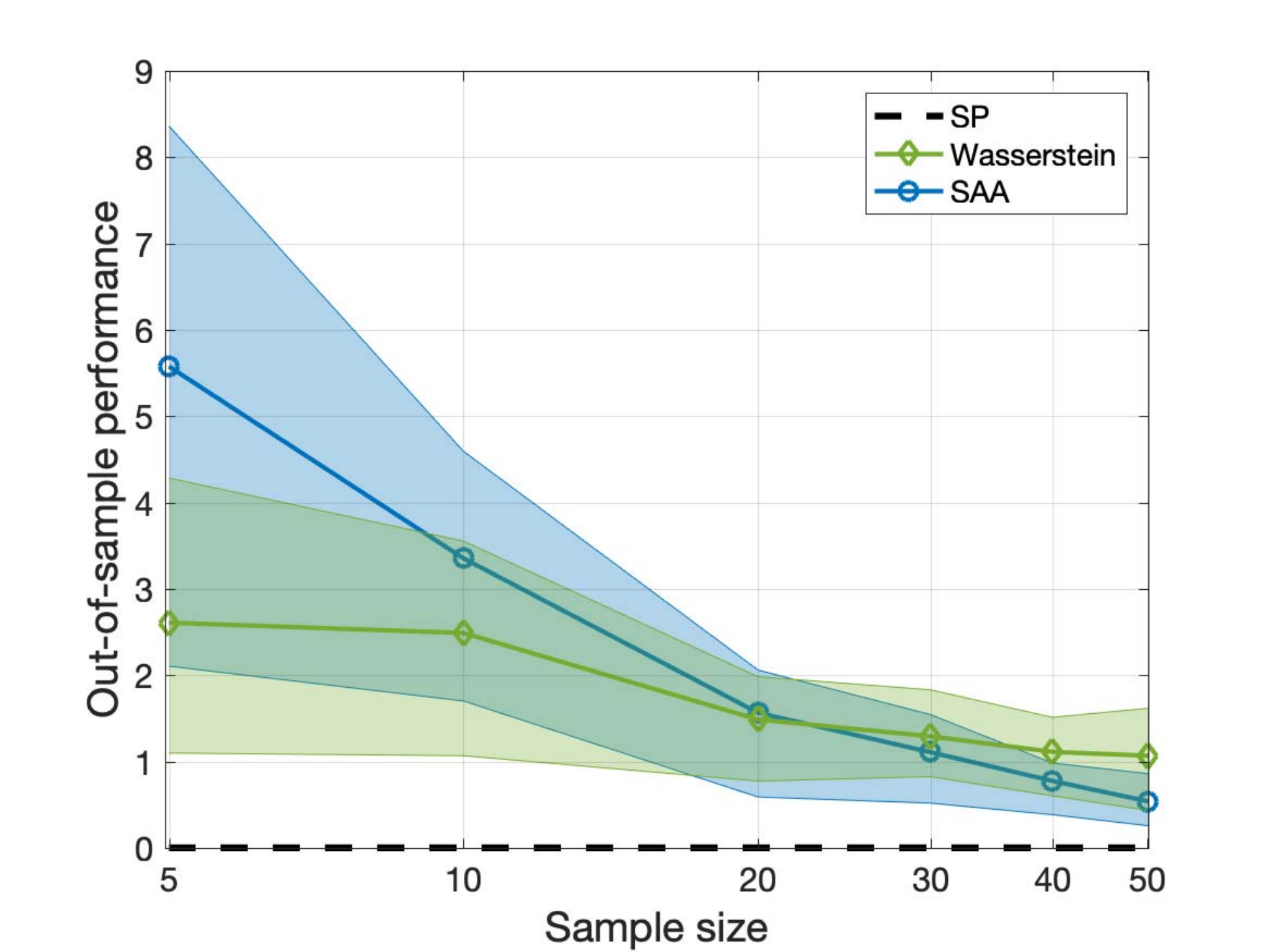}
			\caption{LN, small data}
			\label{fig:lnosp-small}
		\end{subfigure}%
		\hfill
		\begin{subfigure}[t]{0.32\textwidth}
            \includegraphics[width=\linewidth]{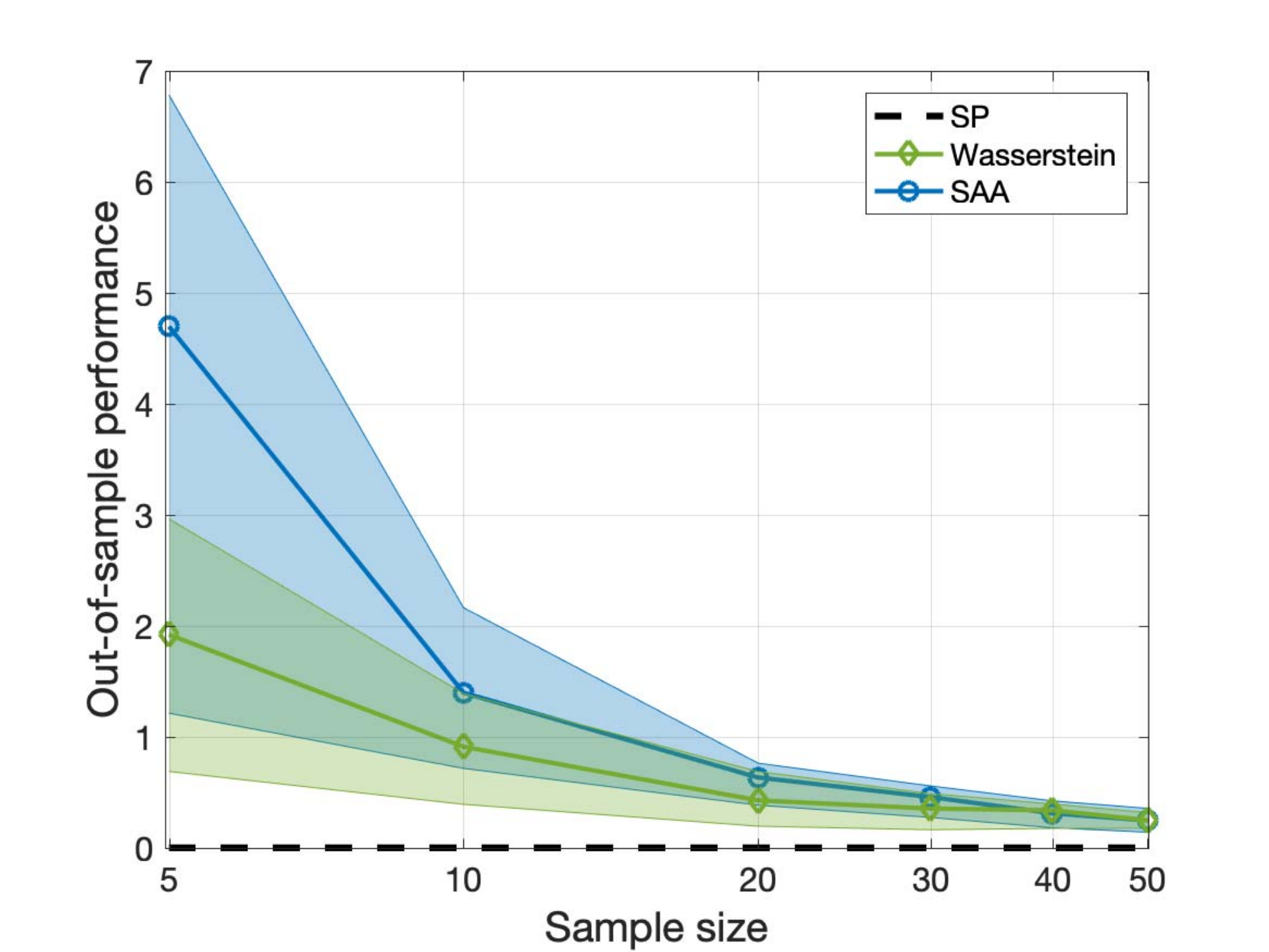}
			\caption{UB, small data}
			\label{fig:ubosp-small}
		\end{subfigure} %
		\hfill
		\begin{subfigure}[t]{0.32\textwidth}
            \includegraphics[width=\linewidth]{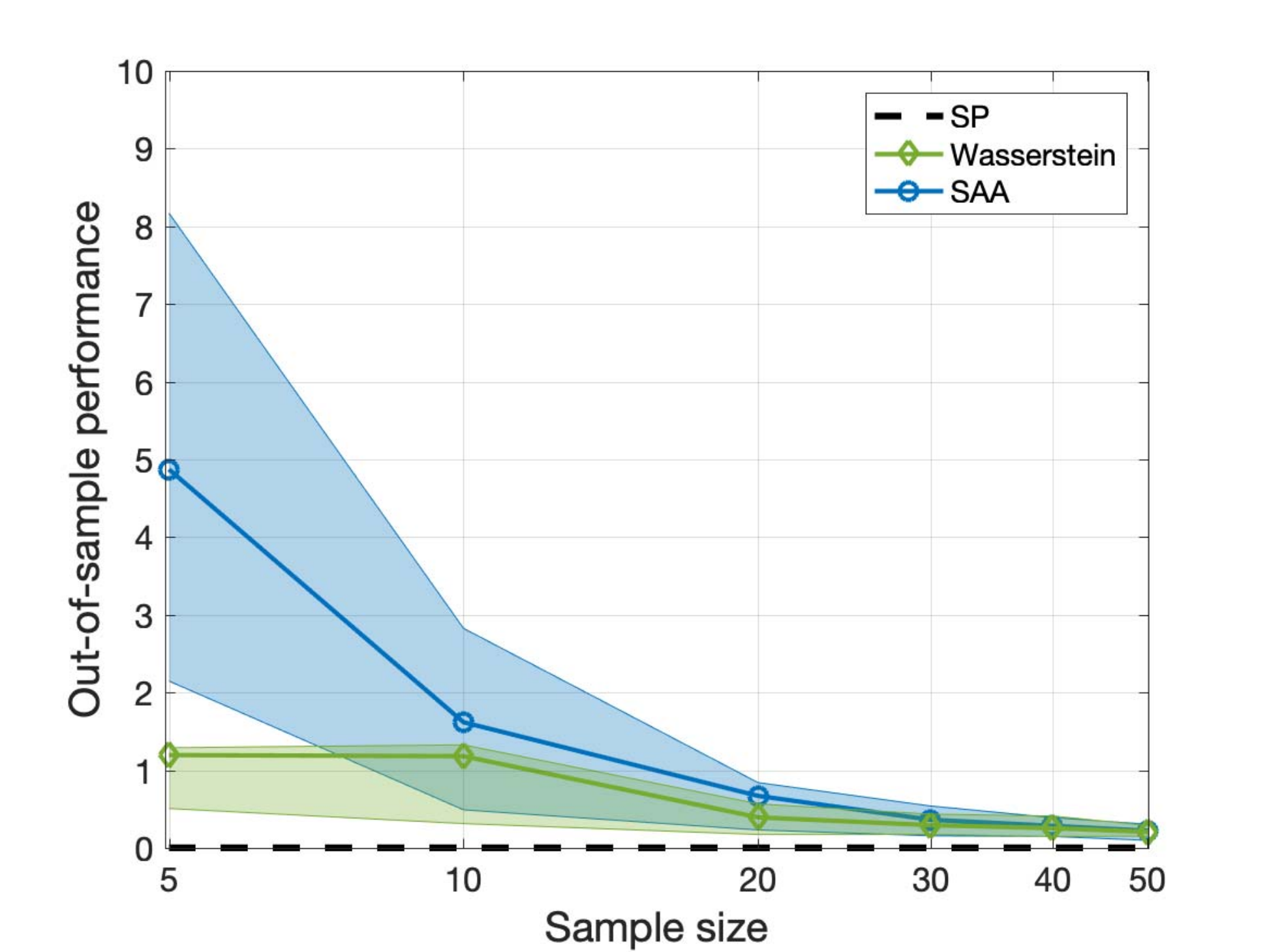}
			\caption{NG, small data}
			\label{fig:ngosp-small}
		\end{subfigure}%
		\hfill
		\quad
		\begin{subfigure}[t]{0.32\textwidth}
            \includegraphics[width=\linewidth]{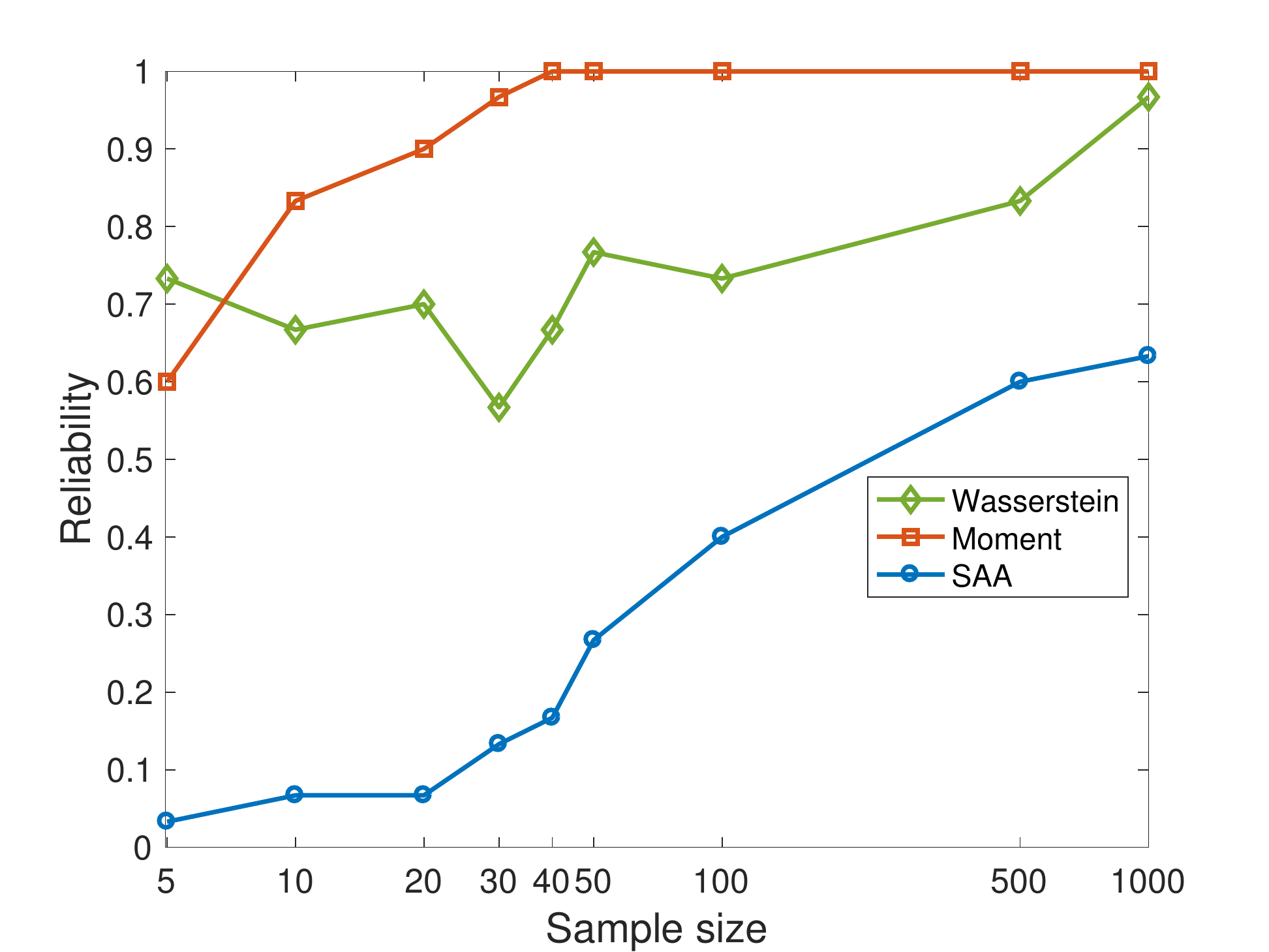}
			\caption{LN, reliability}
			\label{fig:lnrel}
		\end{subfigure}%
		\hfill
		\begin{subfigure}[t]{0.32\textwidth}
            \includegraphics[width=\linewidth]{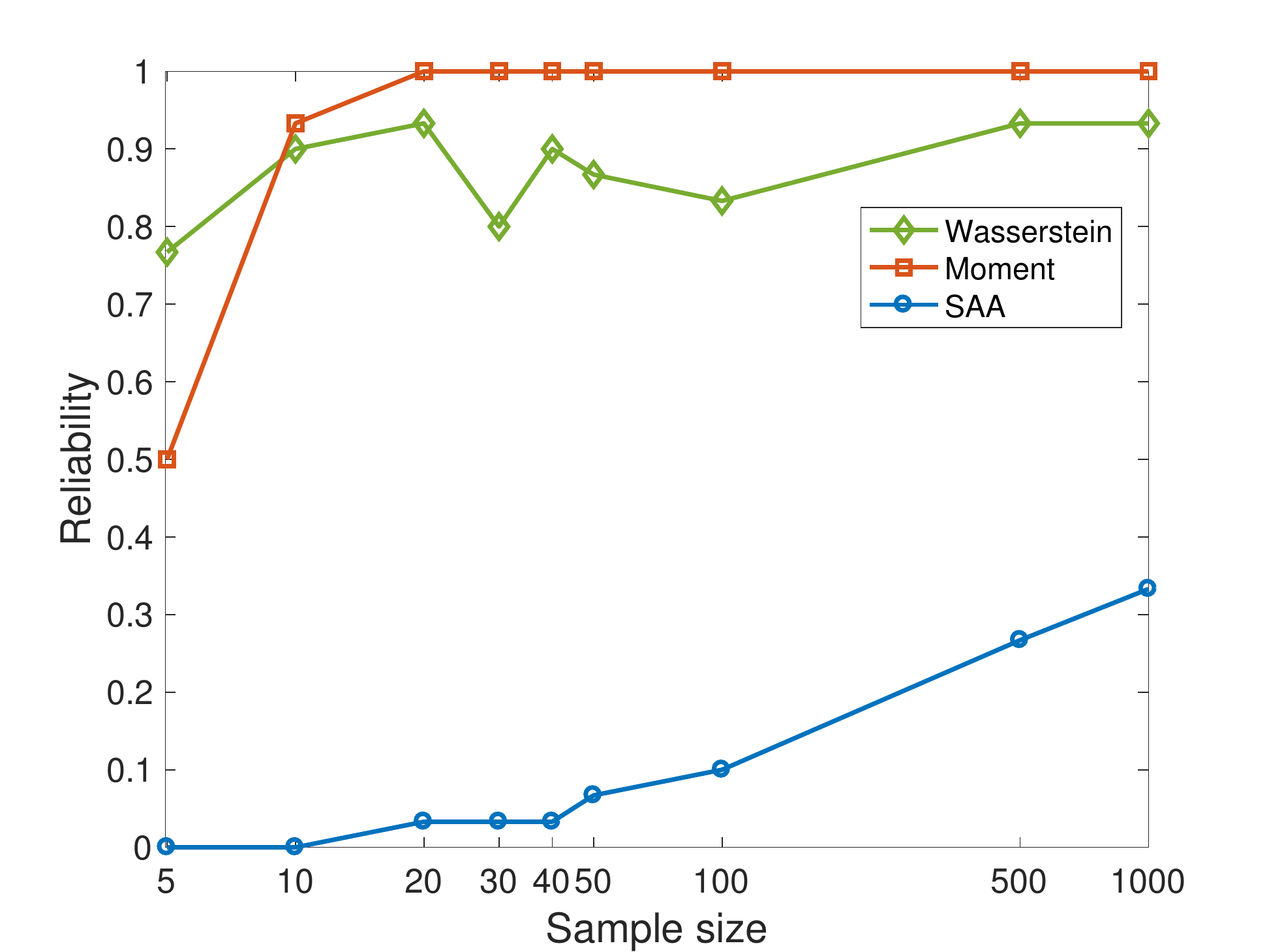}
			\caption{UB, reliability}
			\label{fig:ubrel}
		\end{subfigure}%
		\hfill
		\begin{subfigure}[t]{0.32\textwidth}
			\includegraphics[width=\linewidth]{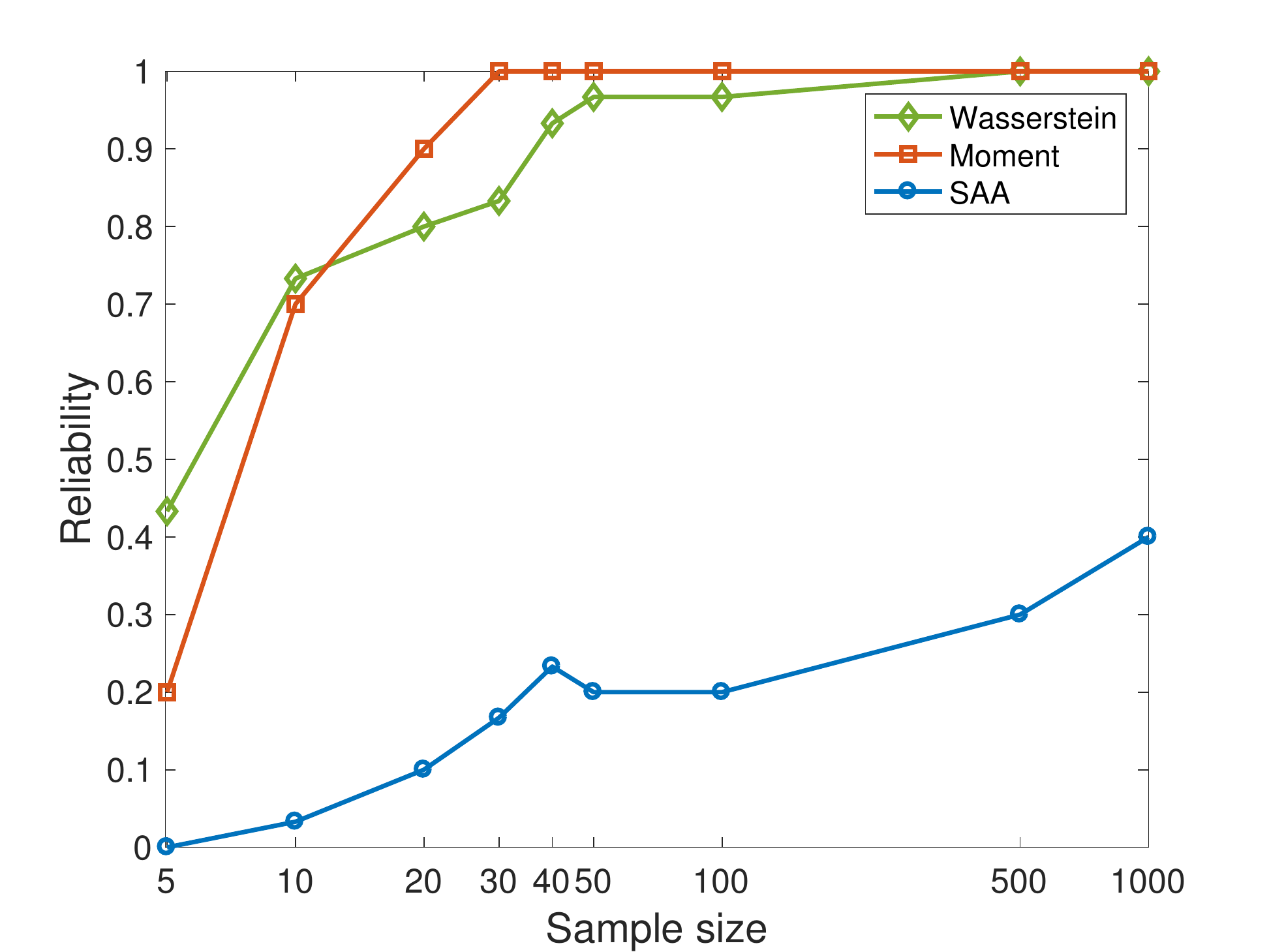}
			\caption{NG, reliability}
			\label{fig:ngrel}
		\end{subfigure}%
		\caption{Out-of-sample performance and reliability of the optimal W-NS, MM, and SAA appointment schedules as a function of data size $N$}
		\label{fig:ns_osp}
	\end{figure}

We report the experiment results in Figure~\ref{fig:ns_osp}. In particular, Figures~\ref{fig:lnosp}--\ref{fig:ngosp} visualize the out-of-sample performance of optimal W-NS, MM, and SAA appointment schedules. From these figures, we observe that the out-of-sample performance of W-NS and SAA converge to the optimal value of the stochastic schedule model, while that of MM does not. This confirms that the \eqref{equ:dras-ns} approach enjoys the asymptotic consistency. In contrast, the MM approach does not have such convergence guarantee because its ambiguity set relies only on the mean and support information. Figures~\ref{fig:lnosp-small}--\ref{fig:ngosp-small} report the out-of-sample performance of the W-NS and SAA approaches when the data size is small. From these figures, we observe that W-NS outperforms SAA. Intuitively, this demonstrates that W-NS is capable of learning the distributional information even from a limited amount of data (e.g., when $N \leq 20$). As a consequence, the proposed W-NS approach is particularly effective in AS systems with scarce no-show and service duration data. Figures~\ref{fig:lnrel}--\ref{fig:ngrel} report the reliability of the three approaches. These figures once again confirm our observations in Section \ref{sec:exp_rs} that the W-NS approach can provide a safe (upper bound) guarantee on the expected total cost even with a small data size.

	\begin{figure}
		\begin{subfigure}[b]{0.32\textwidth}
			\includegraphics[width=\linewidth]{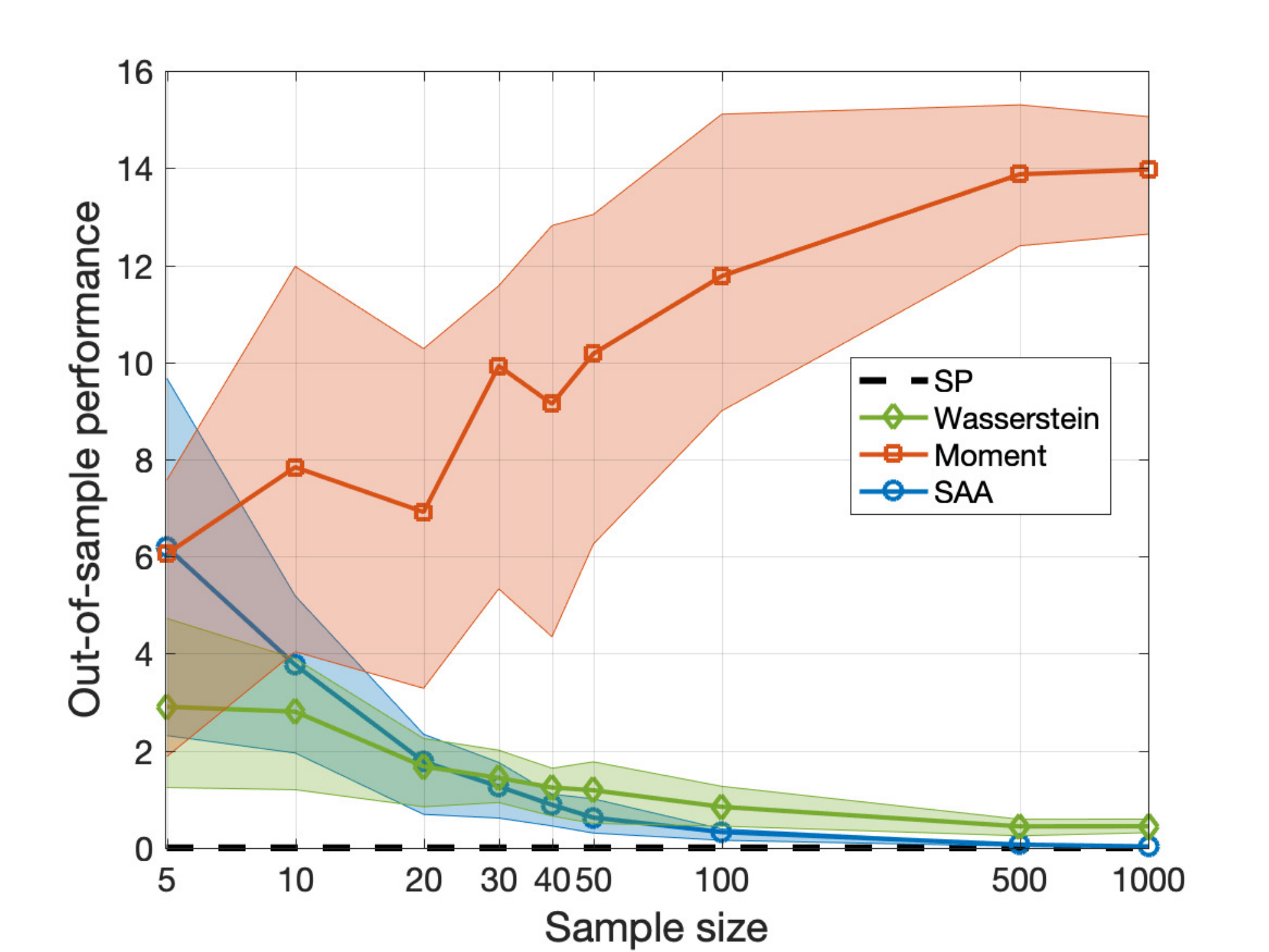}
			\caption{Misspecified LN}
		\end{subfigure}%
		\begin{subfigure}[b]{0.32\textwidth}
			\includegraphics[width=\linewidth]{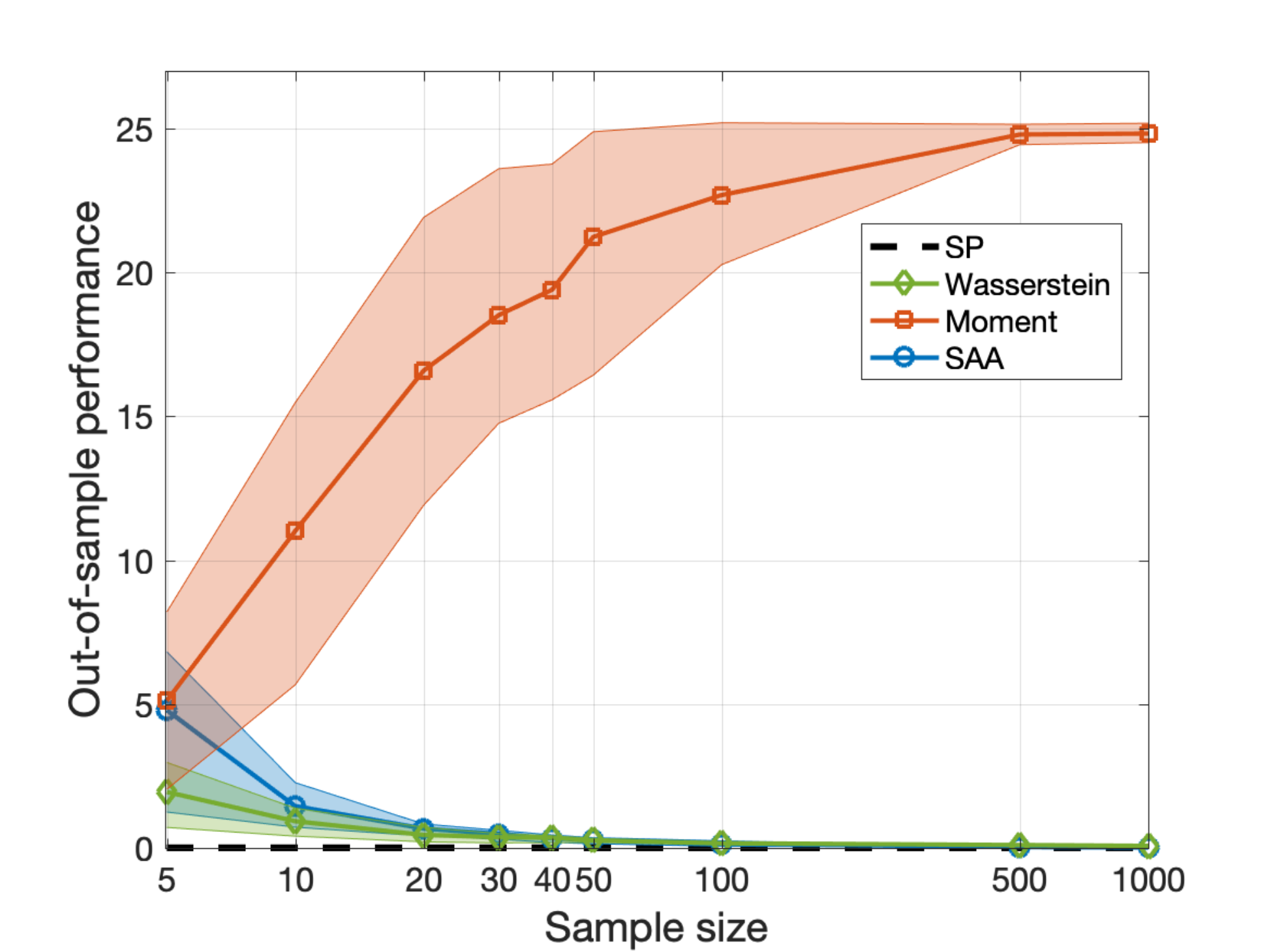}
			\caption{Misspecified UB}
		\end{subfigure}
		\quad
		\begin{subfigure}[b]{0.32\textwidth}
			\includegraphics[width=\linewidth]{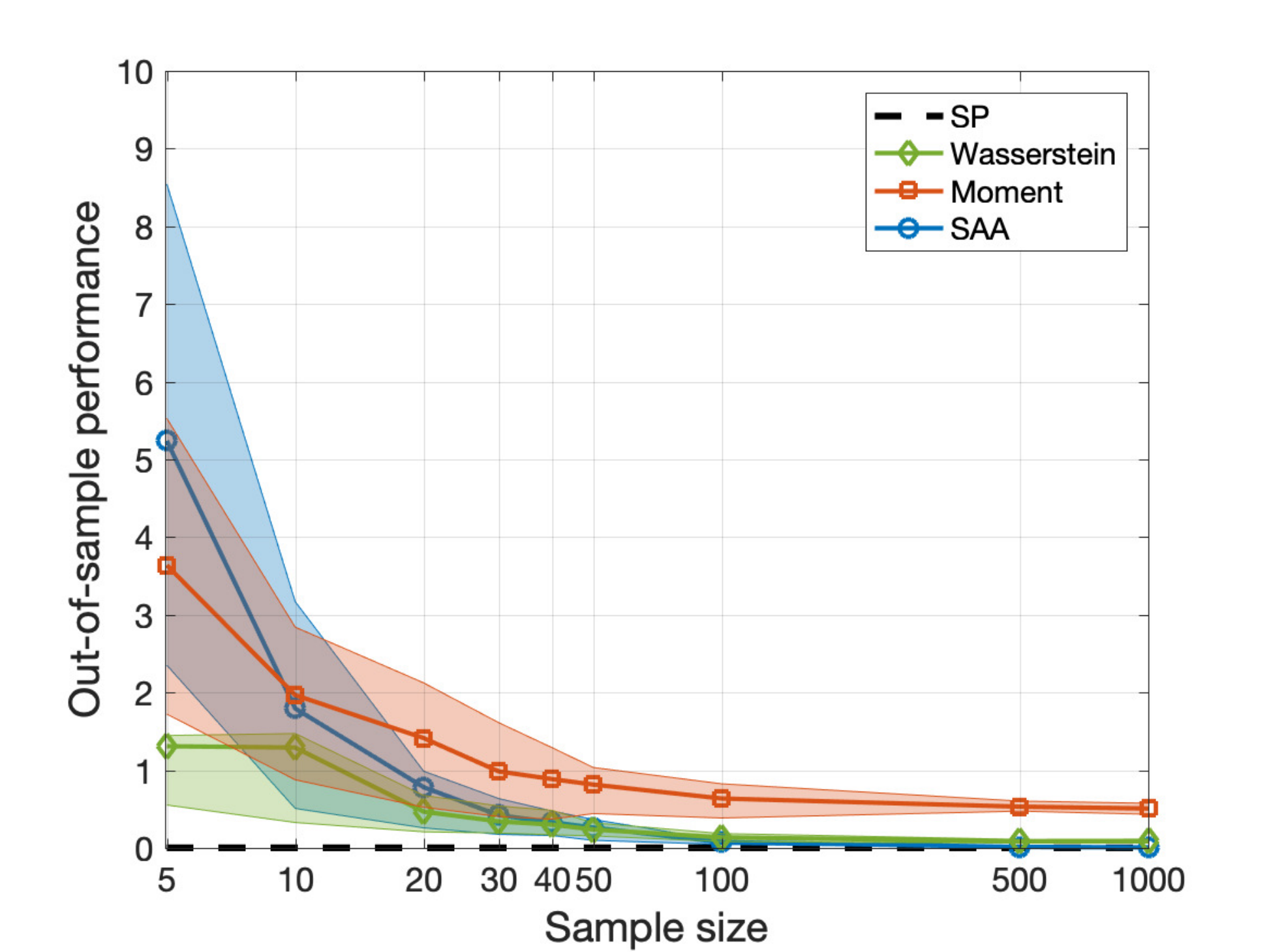}
			\caption{Misspecified NG}
		\end{subfigure}%
		\caption{Out-of-sample performance of the optimal W-NS, MM, and SAA appointment schedules under misspecified distributions}
		\label{fig:ns_mis_osp}
	\end{figure}

We also conduct an experiment to examine out-of-sample performance of the optimal W-NS, MM, and SAA appointment schedules under misspecified distributions. This is particularly motivated by a situation where the no-show behaviors depend on the appointment schedule (see~\cite{Kong.Li.Liu.Teo.Yan.2015}). In that case, if we ignore the impact of the schedule and model the random no-shows by using a schedule-\emph{independent} ambiguity set, then the true, schedule-\emph{dependent} distribution may not belong with the ambiguity set. In this experiment, we use the same types of distributions as those generating the $N$ in-sample data, but increase or decrease their parameters by $\sigma$\% with $\sigma$ uniformly sampled from [5, 10]. Figure~\ref{fig:ns_mis_osp} shows the performance of the W-NS, MM, and SAA appointment schedules under misspecified distributions. We observe that, once again, the W-NS and SAA approaches outperform the MM approach even under misspecified distributions. In addition, the W-NS approach still outperforms SAA when the data size is small. Finally, we observe that the out-of-sample performance of both W-NS and SAA quickly improve (i.e., tend to the optimal value of the true model) as the data size increases. This demonstrates that the proposed W-NS approach remains effective in AS systems under misspecified distributions, and the cost of ignoring the appointment dependency is limited.

\section{Conclusions} \label{sec:conclusions}
In this paper, we studied a distributionally robust appointment scheduling problem over Wasserstein ambiguity sets. We proposed two models, with the first considering random service durations only and the second considering both random no-shows and service durations. We showed that both models can be recast as tractable convex programs under mild conditions on (i) support set of the uncertainties and (ii) penalty costs of the waiting time and server idleness. Through extensive numerical experiments, we demonstrated that the proposed approaches enjoy both asymptotic consistency and finite-data guarantees, and are particularly suitable for the AS systems with scarce data or in quickly varying environments.

\newpage
\section*{Appendix}

\subsection*{Proof of Lemma \ref{lem:concentration}}
We review (a part of) Theorem 2 in~\cite{fournier2015rate} that is directly related to our discussion.
\begin{theorem}[Adapted from Theorem 2 in~\cite{fournier2015rate}] \label{thm:concentration}
Let $\mathbb{P} \in \mathcal{P}(\mathbb{R}^n)$ and $p > 0$, and assume that there exist $\alpha > p$ and $\gamma > 0$ such that $\int_{\mathbb{R}^n}\exp\{\gamma \|\bm u\|_p^{\alpha} \} d\mathbb{P}(\bm u) < \infty$. Then for all $N \geq 1$ and $\epsilon \in (0, \infty)$,
$$
\mathbb{P}^N\left\{ d_p(\mathbb{P}, \widehat{\mathbb{P}}^N) \geq \epsilon^{1/p} \right\} \ \leq \ a(N, \epsilon) \text{\bf 1}_{\{\epsilon \leq 1\}} + b(N, \epsilon),
$$
where
$$
a(N, \epsilon) = C \left\{\begin{array}{ll} \exp\{-c N \epsilon^2\} & \mbox{if $p > n/2$}\\
\exp\{-c N (\epsilon/\log(2 + 1/\epsilon))^2\} & \mbox{if $p = n/2$}\\
\exp\{-c N \epsilon^{n/p}\} & \mbox{if $p \in [1,n/2)$}
\end{array}\right. \mbox{and} \ \ b(N, \epsilon) = C \exp\{-c N \epsilon^{\alpha/p}\} \text{\bf 1}_{\{\epsilon > 1\}}.
$$
The positive constants $C$ and $c$ depend only on $p$, $n$, $\alpha$, and $\gamma$.
\end{theorem}
\begin{proof}[Proof of Lemma \ref{lem:concentration}]
First, for any $\alpha > p$ and $\gamma > 0$, $\mathbb{E}_{\mathbb{P}_{\bm u}}[\exp\{\gamma \|\bm u\|_p^{\alpha} \}] < \infty$ because the support set $\mathcal{U}$ is bounded by Assumption \ref{ass:u_set}. It follows from Theorem \ref{thm:concentration} that $\mathbb{P}_{\bm u}^N\left\{ d_p(\mathbb{P}_{\bm u}, \widehat{\mathbb{P}}^N_{\bm u}) \geq \epsilon^{1/p} \right\} \leq a(N, \epsilon) \text{\bf 1}_{\{\epsilon \leq 1\}} + b(N, \epsilon)$ for all $N \geq 1$ and $\epsilon > 0$.

Second, we discuss the following two cases based on the value of $\epsilon$.
\begin{enumerate}[(1)]
\item If $0 < \epsilon \leq 1$, then $0 < \epsilon \left[\log(2 + 1/\epsilon)\right]^2 \geq \left[\log(3)\right]^2$. It follows that $(\epsilon/\log(2 + 1/\epsilon))^2 \geq \epsilon^3/[\log(3)]^2$ and $\exp\{-c N (\epsilon/\log(2 + 1/\epsilon))^2\} \leq \exp\{-cN\epsilon^3/[\log(3)]^2\}$. Then,
$$
a(N, \epsilon) \leq C \left\{\begin{array}{ll} \exp\{-c N \epsilon^2\} & \mbox{if $p > n/2$}\\
\exp\{-c N \epsilon^3/[\log(3)]^2\} & \mbox{if $p = n/2$}\\
\exp\{-c N \epsilon^{n/p}\} & \mbox{if $p \in [1,n/2)$}
\end{array}\right. \ \leq \ C \exp\left\{- \frac{c}{[\log(3)]^2} N \epsilon^{\max\{3, n/p\}} \right\},
$$
where the second inequality is because $\log(3) > 1$ and $\epsilon^x$ decreases as $x$ increases with $\epsilon \in (0, 1]$.
\item If $\epsilon > 1$, then we set $\alpha = \max\{3p, n\} > p$. It follows that $b(N, \epsilon) \leq C \exp\{-c N \epsilon^{\max\{3, n/p\}}\}$.
\end{enumerate}
Summarizing the above two cases and letting $c_1 = C$ and $c_2 = c/[\log(3)]^2$, we have
$$
\mathbb{P}_{\bm u}^N\left\{ d_p(\mathbb{P}_{\bm u}, \widehat{\mathbb{P}}^N_{\bm u}) \geq \epsilon^{1/p} \right\} \leq c_1 \exp\left\{-c_2 N \epsilon^{\max\{3, n/p\}}\right\}
$$
for all $N \geq 1$ and $\epsilon > 0$. Equating the right-hand side of the above inequality to $\beta$ and solving for $\epsilon^{1/p}$ yields $\mathbb{P}_{\bm u}^N\{ d_p(\mathbb{P}_{\bm u}, \widehat{\mathbb{P}}^N_{\bm u}) \geq \epsilon_N(\beta) \} \leq \beta$, where $\epsilon_N(\beta)$ is defined in the statement of Lemma \ref{lem:concentration}. This completes the proof.
\end{proof}

\subsection*{Proof of Theorem \ref{thm:a-c}}
We first prove the following three lemmas.

\begin{lemma} \label{lem:a-c-1}
The function $f(\bm s, \bm u)$ is bounded on $\mathcal{S}\times \mathcal{U}$. Additionally, $f(\bm s, \bm u)$ is jointly convex in $\bm s$ and $\bm u$.
\end{lemma}
\begin{proof}
First, by the dual formulation of liner program \eqref{equ:qx}, we represent $f(\bm s, \bm u) = \max_{\bm{y} \in \mathcal{Y}} (\bm u - \bm s)^{\top} \bm{y}$, where $\mathcal{Y}$ is defined in \eqref{equ:y-def}. It follows that $f(\bm s, \bm u)$ is bounded on $\mathcal{S}\times \mathcal{U}$ because (i) $\mathcal{Y}$ is nonempty and bounded by Lemma \ref{lem:y_set}, (ii) $\mathcal{S}$ is bounded by definition, and (iii) $\mathcal{U}$ is bounded by Assumption \ref{ass:u_set}.

Second, $f(\bm s, \bm u)$ is jointly convex in $\bm s$ and $\bm u$ because it is represented as the maximum of linear functions of $\bm s$ and $\bm u$.
\end{proof}

\begin{lemma} \label{lem:a-c-2}
Let $p \geq 1$. Then, the function $f(\bm s, \bm u)$ is continuous on $\mathcal{S}\times \mathcal{U}$. Additionally, for any fixed $\bm s \in \mathcal{S}$, $f(\bm s, \bm u)$ is Lipschitz continuous in $\bm{u}$ with Lipschitz constant $L := \max_{\bm y \in \mathcal{Y}}\|\bm y\|_q < \infty$, where $q$ is such that $\frac{1}{p} + \frac{1}{q} = 1$.
\end{lemma}
\begin{proof}
First, as $f(\bm s, \bm u) \equiv \max_{\bm{y} \in \mathcal{Y}} (\bm u - \bm s)^{\top} \bm{y}$ and $\mathcal{Y}$ is polyhedral, nonempty, and bounded (see \eqref{equ:y-def} and Lemma \ref{lem:y_set}), $f(\bm s, \bm u)$ can be represented as the maximum of a \emph{finite} number of linear functions of $\bm s$ and $\bm u$. It follows that $f(\bm s, \bm u)$ is continuous on $\mathcal{S}\times \mathcal{U}$.

Second, pick any $\bm{s} \in \mathcal{S}$ and $\bm{u}^1$, $\bm{u}^2 \in \mathcal{U}$. Then,
\begin{align*}
f(\bm s, \bm u^1) - f(\bm s, \bm u^2) \ = & \ \max_{\bm{y} \in \mathcal{Y}} (\bm u^1 - \bm s)^{\top} \bm{y} - \max_{\bm{y} \in \mathcal{Y}} (\bm u^2 - \bm s)^{\top} \bm{y} \\
\leq & \ \max_{\bm{y} \in \mathcal{Y}} \left\{(\bm u^1 - \bm s)^{\top} \bm{y} - (\bm u^2 - \bm s)^{\top} \bm{y}\right\} \\
= & \ \max_{\bm{y} \in \mathcal{Y}} \bm{y}^{\top} (\bm{u}^1 - \bm{u}^2) \\
\leq & \ \max_{\bm{y} \in \mathcal{Y}} \left\| \bm y \right\|_q \left\| \bm{u}^1 - \bm{u}^2 \right\|_p \\
= & \ L \left\| \bm{u}^1 - \bm{u}^2 \right\|_p,
\end{align*}
where the second inequality follows from the H{\"o}lder's inequality. We note that $L < \infty$ because $\mathcal{Y}$ is bounded and $\|\bm y\|_q$ is continuous in $\bm y$. Similarly, we can show that $f(\bm s, \bm u^2) - f(\bm s, \bm u^1) \leq L \left\| \bm{u}^1 - \bm{u}^2 \right\|_p$. Hence, $|f(\bm s, \bm u^1) - f(\bm s, \bm u^2)| \leq L \left\| \bm{u}^1 - \bm{u}^2 \right\|_p$ for any $\bm{s} \in \mathcal{S}$ and $\bm{u}^1$, $\bm{u}^2 \in \mathcal{U}$, which completes the proof.
\end{proof}

\begin{lemma}[Adapted from Lemma 3.7 in~\cite{Esfahani.Kuhn.2017}] \label{lem:a-c-3}
Let $p \geq 1$. Consider a sequence of confidence levels $\{\beta_N\}_{N \in \mathbb{N}}$ such that $\sum_{N=1}^{\infty} \beta_N < \infty$ and $\lim_{N\rightarrow \infty} \epsilon_N(\beta_N) = 0$. Additionally, let $\{\widehat{\mathbb Q}_{\bm u}^N\}_{N \in \mathbb{N}}$ represent a sequence of probability distributions with each $\widehat{\mathbb Q}_{\bm u}^N \in \mathcal{D}_p(\widehat{\mathbb{P}}^N_{\bm u}, \epsilon_N(\beta_N))$. Then, $\lim_{N\rightarrow \infty} d_1(\mathbb{P}_{\bm u}, \widehat{\mathbb Q}_{\bm u}^N) = 0$ $\mathbb{P}_{\bm u}^{\infty}$-almost surely.
\end{lemma}
\begin{proof}
First, as $\widehat{\mathbb Q}_{\bm u}^N \in \mathcal{D}_p(\widehat{\mathbb{P}}^N_{\bm u}, \epsilon_N(\beta_N))$, by the triangular inequality we have $d_p(\mathbb{P}_{\bm u}, \widehat{\mathbb Q}_{\bm u}^N) \leq d_p(\mathbb{P}_{\bm u}, \widehat{\mathbb P}^N_{\bm u}) + d_p(\widehat{\mathbb P}^N_{\bm u}, \widehat{\mathbb Q}_{\bm u}^N) \leq d_p(\mathbb{P}_{\bm u}, \widehat{\mathbb P}^N_{\bm u}) + \epsilon_N(\beta_N)$. By Lemma \ref{lem:concentration}, we have $\mathbb{P}^N_{\bm u}\left\{d_p(\mathbb{P}_{\bm u}, \widehat \PP^N_{\bm u}) \leq \epsilon_N(\beta_N)\right\} \geq 1 - \beta_N$ and so
$$
\mathbb{P}^N_{\bm u}\left\{d_p(\mathbb{P}_{\bm u}, \widehat{\mathbb Q}_{\bm u}^N) \leq 2\epsilon_N(\beta_N)\right\} \geq 1 - \beta_N.
$$
But as $\sum_{N=1}^{\infty} \beta_N < \infty$, the Borel-Cantelli Lemma implies that $\mathbb{P}^{\infty}_{\bm u}\{d_p(\mathbb{P}_{\bm u}, \widehat{\mathbb Q}_{\bm u}^N) \leq 2\epsilon_N(\beta_N) \ \mbox{eventually}\} = 1$. In addition, as $\lim_{N\rightarrow \infty} \epsilon_N(\beta_N) = 0$, we have $\lim_{N\rightarrow \infty} d_p(\mathbb{P}_{\bm u}, \widehat{\mathbb Q}_{\bm u}^N) = 0$ $\mathbb{P}_{\bm u}^{\infty}$-almost surely.

Second, by the Jensen's inequality we have
\begin{align*}
d_p(\mathbb{P}_{\bm u}, \widehat{\mathbb Q}_{\bm u}^N) \ = & \ \Biggl(\inf_{\Pi \in \mathcal{P}(\mathbb{P}_{\bm u}, \widehat{\mathbb Q}_{\bm u}^N)} \mathbb{E}_{\Pi}\bigl[ \|\bm{u}_{\mathbb{P}} - \bm{u}_{\mathbb{Q}}\|_p^p \bigr] \Biggr)^{1/p} \\
\geq & \ \Biggl\{ \Biggl(\inf_{\Pi \in \mathcal{P}(\mathbb{P}_{\bm u}, \widehat{\mathbb Q}_{\bm u}^N)} \mathbb{E}_{\Pi}\bigl[ \|\bm{u}_{\mathbb{P}} - \bm{u}_{\mathbb{Q}}\|_p^1 \bigr] \Biggr)^p \Biggr\}^{1/p} \\
= & \ d_1(\mathbb{P}_{\bm u}, \widehat{\mathbb Q}^N).
\end{align*}
It follows that $\lim_{N\rightarrow \infty} d_1(\mathbb{P}_{\bm u}, \widehat{\mathbb Q}_{\bm u}^N) = 0$ $\mathbb{P}_{\bm u}^{\infty}$-almost surely.
\end{proof}

Theorem \ref{thm:a-c} follows from Lemmas \ref{lem:a-c-1}--\ref{lem:a-c-3}.
\begin{proof}[Proof of Theorem \ref{thm:a-c}]
By Assumption \ref{ass:u_set} and Lemmas \ref{lem:a-c-1}--\ref{lem:a-c-3}, all conditions of Theorem 3.6 in~\cite{Esfahani.Kuhn.2017} are satisfied. Therefore, the conclusions of Theorem \ref{thm:a-c} hold valid.
\end{proof}

\subsection*{Proof of Theorem \ref{thm:f-d}}
\begin{proof}
By Assumption \ref{ass:u_set} and Lemma \ref{lem:concentration}, all conditions of Theorem 3.5 in~\cite{Esfahani.Kuhn.2017} are satisfied. Therefore, the conclusion of Theorem \ref{thm:f-d} holds valid.
\end{proof}

\subsection*{Proof of Lemma \ref{lem:y_set}}
\begin{proof}
First, $\bm 0 \in \Y$ because $\bm c \geq \bm 0$, $\bm d \geq \bm 0$, and $C \geq 0$. It follows that $\Y$ is nonempty. Next, $\Y$ is closed as it is polyhedral. Finally, $-d_n \leq y_n \leq C$ implies that $y_{n-1} \leq y_n + c_n \leq C + c_n$ and so $-d_{n-1} \leq y_{n-1} \leq C + c_n$. Similarly, we have $-d_i \leq y_i \leq C + \sum_{j=i+1}^n c_j$ for all $i \in [n-2]$. This implies that $\Y$ is bounded and completes the proof.
\end{proof}

\subsection*{Proof of Proposition \ref{prop:ROFormat}}
\begin{proof}
First, by the definition of Wasserstein distance, $\mathbb{Q}_{\bm u} \in \mathcal{D}_p(\widehat{\mathbb{P}}^N_{\bm u}, \epsilon)$ implies that there exists a joint distribution $\Pi$ of $(\bm u, \widehat{\bm u})$ with marginals $\mathbb{Q}_{\bm u}$ and $\widehat{\mathbb{P}}^N_{\bm u}$, such that $\mathbb{E}_{\Pi}\left[\|\bm{u} - \widehat{\bm{u}}\|_p^p\right] \leq \epsilon^p$. As $\widehat{\mathbb{P}}^N_{\bm u} = \frac{1}{N}\sum_{j=1}^N \delta_{\widehat{\bm u}^j}$, there exist conditional distributions $\{\mathbb{Q}_{\bm u}^j\}_{j \in [N]}$ such that $\Pi = \frac{1}{N}\sum_{j=1}^N \mathbb{Q}_{\bm u}^j$, where each $\mathbb{Q}_{\bm u}^j$ represents the distribution of $\bm u$ conditional on that $\widehat{\bm u} = \widehat{\bm u}^j$. It follows that problem \eqref{equ:bound} can be recast as
\begin{equation} \label{equ:moment}
\begin{array}{rl}
                     \widehat{Z}(\bm s) \ := \ \sup  & \dfrac1N  \sum\limits_{j=1}^N  \displaystyle\int_{\mathcal{U}} f(\bm s, \, \bm u) \, \QQ^j_{\bm u}(d\bm u)  \\
                            \st &  \dfrac1N \sum\limits_{j=1}^N  \displaystyle\int_{\mathcal{U}} \|\bm u - \bm{\widehat u}^j\|_p^p \,\QQ^j_{\bm u}(d\bm u) \le \epsilon^p \\
                                 & \QQ^j_{\bm u} \in {\cal P}(\mathcal{U}) \ \ \ \forall \, j \in [N].
\end{array}
\end{equation}
Second, by a standard duality argument and letting $\rho \geq 0$ be the Lagrangian
dual multiplier, we write the dual of (\ref{equ:moment}) as:
\begin{subequations}
\begin{align}
  \widehat{Z}(\bm s) \ = \ & \inf \limits_{\rho \ge 0} \sup \limits_{\QQ^j_{\bm u} \in \mathcal{P}(\mathcal{U})}  \epsilon^p \rho + \frac{1}{N}\sum_{j=1}^N \int_{\cal U} \left(f(\bm s, \, \bm u) - \rho \|\bm u - \bm{\widehat u}^j\|_p^p \right) \,\QQ^j_{\bm u}(d\bm u) \label{equ:sub1+} \\
  = \ & \inf \limits_{\rho \ge 0}  \epsilon^p \rho + \frac{1}{N}\sum_{j=1}^N \sup \limits_{\bm u \in \mathcal{U}} \ \left\{f(\bm s, \, \bm u) - \rho \|\bm u - \bm{\widehat u}^j\|_p^p \right\}, \label{equ:sub2+}
\end{align}
\end{subequations}
where equality \eqref{equ:sub1+} follows from the strong duality
result for the moment problems (see Proposition 3.4 in \cite{Shapiro.2001}, Theorem 1 in \cite{Hanasusanto.Kuhn.2018}, and Lemma 7 in
\cite{Hanasusanto.Roitch.Kuhn.Wiesemann.2017}) and equality \eqref{equ:sub2+} follows from the
fact that ${\cal P}({\cal U})$ contains all the Dirac distributions supported on $\cal U$. Finally, reformulation \eqref{equ:vdp_nonconvex} is obtained by introducing an auxiliary variable $\theta_j$ to represent each supremum in \eqref{equ:sub2+}.
\end{proof}

\subsection*{Proof of Theorem \ref{thm:dro-cop}}

We first establish three technical lemmas. Consider the following quadratic program:
\begin{equation} \label{equ:qp_cons}
\begin{array}{ll}
\displaystyle\sup & \bm{x}^\top \bm{Q}_0 \bm{x} \\
\st    &  \bm{e}_1^\top \bm{x} = 1, \ \bm{x} \in \K  \\
         & \bm{x}^\top \bm{Q}_i \bm{x} = 0 \ \ \  \forall \, i \in [n],
\end{array}
\end{equation}
where $\bm{Q}_i \in \SYM^k \ \forall \, i = 0, \ldots, n$. Its copositive programming relaxation (see \cite{Burer.2009, Burer.2012}) reads:
\begin{equation} \label{equ:cpp_cons}
\begin{array}{ll}
\displaystyle\sup & \bm{Q}_0 \bullet  \bm X \\
\st    & \bm{e}_1\bm{e}_1^\top \bullet \bm X = 1, \ \bm X \in \CP(\K) \\
        & \bm{Q}_i \bullet \bm X = 0 \ \ \ \forall \, i \in [n],
\end{array}
\end{equation}
where $\K$ is a nonempty, closed and convex cone.
Defining ${\cal L} := \left\{ \bm x \in \K : \bm{e}_1^\top \bm{x} = 1 \right\}$,
we review the following result.

\begin{lemma} [\cite{Burer.2012}, Corollary 8.4, Theorem 8.3] \label{lem:burer}
Suppose that $\cal L$ is nonempty and bounded. Then, (\ref{equ:cpp_cons}) is equivalent to
(\ref{equ:qp_cons}), i.e., i) the optimal value of (\ref{equ:cpp_cons}) is equal to that of (\ref{equ:qp_cons});
ii) if $\bm{X}^\star$ is an optimal solution for (\ref{equ:cpp_cons}), then $\bm{X}^\star\mathbf{e}_1$
is in the convex hull of optimal solutions for (\ref{equ:qp_cons}).
\end{lemma}

By the conic programming duality, the dual of (\ref{equ:cpp_cons}) is the following linear program over the cone of copositive matrices with respect to $\K$:
\begin{equation} \label{equ:cop_cons}
\begin{array}{ll}
\inf & \beta_j \\
\st    & \beta_j \, \bm{e}_1\bm{e}_1^\top  + \sum\limits_{i=1}^n \psi_i \bm{Q}_i - \bm{Q}_0  \in \COP(\K) \\
        & \beta_j \in \RR, \, \bm \psi \in \RR^n.
\end{array}
\end{equation}

\noindent We now showcase that strong duality holds between~\eqref{equ:cpp_cons} and~\eqref{equ:cop_cons} in the lemma below.

\begin{lemma} \label{lem:cpp_cop}
Suppose that $\cal L$ is noempty and bounded. Then, strong duality holds between (\ref{equ:cpp_cons}) and (\ref{equ:cop_cons}),
i.e., the optimal value of~\eqref{equ:cpp_cons} equals that of~\eqref{equ:cop_cons}.
\end{lemma}

\begin{proof}
We prove the statement by showing that the dual problem~\eqref{equ:cop_cons} admits a Slater point.
To this end, we seek for a scalar $\beta_j$ such that the following relation holds:
\begin{equation} \label{equ:copositive_expression}
\big (t, \bar{\bm x}^\top \big) \Big( \beta_j \bm e_1 \bm e_1^\top  - \bm{Q}_0 \Big) \big (t, \bar{\bm x}^\top \big)^\top =
\beta_j t^2 -\big (t, \bar{\bm x}^\top \big)   \bm{Q}_0  \big (t, \bar{\bm x}^\top \big)  ^\top > 0
\end{equation}
for all non-zero vector $(t, \bar{\bm x}) \in \K$.

We first show that $t = 0 \ \implies \ \bar{\bm x}  = 0$. We prove the statement by using the contradiction argument. Suppose that $t = 0$ but $\bar{\bm x} \neq 0$. Choose $(1,\bar{\bm x}') \in \K$. Then for any non-negative scalar
$\lambda \geq  0$, we have $\bm z(\lambda) := (1, \bar{\bm x}') + \lambda (0, \bar{\bm x}) \in \K$ because
$\K$ is a closed and convex cone. As $\lambda$ can be arbitrarily large while $\bar{\bm x} \neq 0$, we conclude that $\cal L$ is unbounded, contradicting the statement.

Therefore, it suffices to consider the case of $t > 0$. We then can divide the expression in~\eqref{equ:copositive_expression} by $t^2$, which requires us to show
the following equivalent relation:
\[
\beta_j - \big (1, (\bar{\bm x}/t)^\top \big)  \bm{Q}_0  \big (1, (\bar{\bm x}/t)^\top \big) ^\top > 0
\]
for all $(t, \bar{\bm x}) \in \K$ and $t > 0$. Then, the boundedness of $\cal L$ implies that there exists a constant $\beta_j^\star$
such that $\beta_j^\star > \big (1, (\bar{\bm x}/t)^\top \big)  \bm{Q}_0 \big (1, (\bar{\bm x}/t)^\top \big) ^\top$ for all $\big (1, (\bar{\bm x}/t)^\top \big)  \in \cal L$. The claim thus follows since the point $\beta_j = \beta_j^\star$ constitutes a Slater point for the problem~\eqref{equ:cop_cons}.
\end{proof}

\begin{proof}[Proof of Theorem \ref{thm:dro-cop}]
Using the result from Proposition~\ref{prop:ROFormat}, for any given $\bm s \in {\cal S}$,
we can compute the worst-case expectation value by solving a standard robust optimization
problem shown in~\eqref{equ:vdp_nonconvex}, where each of the semi-infinite constraints
entails a non-convex program. Then, for any fixed $(\bm s, \rho, \theta_j) \in \RR^n \times \RR \times \RR$,
we consider the $j$-th constraint separately:
\begin{equation} \label{equ:jth-cons}
\sup \limits_{\bm u \in \mathcal{U}} \ (f(\bm s, \, \bm u) - \rho \|\bm u - \bm{\widehat u}^j\|_p^p) \le \theta_j.
\end{equation}
First if $p=1$, then the maximization problem on the left-hand side of \eqref{equ:jth-cons} can be reformulated as
\begin{equation} \label{equ:qp_cons-p1-equiv}
\begin{array}{ll}
\sup & (\bm u^+ - \bm u^- + \bm{\widehat u}^j - \bm s)^\top \bm y - \rho \bm e^\top (\bm u^+ + \bm u^- ) \\
\st    &  (\bm u^+, \, \bm u^-, \, \bm y) \in \mathcal{F}_1^j.
\end{array}
\end{equation}

Letting $\bm x := [t, (\bm{u}^+)^\top, (\bm{u}^-)^\top, \bm{y}^\top]^\top \in \RR^{3n+1}$, we rewrite~\eqref{equ:qp_cons-p1-equiv} as
\begin{equation} \label{equ:nqp_cons-p1}
\begin{array}{ll}
\sup & \bm{H}_j^1(\rho, \, \bm s) \bullet \bm{x}\bm{x}^\top  \\
\st    &  \bm{e}_1\bm{e}_1^\top \bullet \bm{x}\bm{x}^\top = 1 \\
         & \bm{x} \in \K^j_1.
\end{array}
\end{equation}
By Lemma~\ref{lem:burer}, \eqref{equ:nqp_cons-p1}
can be reformulated as a linear program over the cone of completely positive matrices:
\begin{equation} \label{equ:cp_cons-p1}
\begin{array}{ll}
\sup & \bm{H}_j^1(\rho, \, \bm s)  \bullet \bm{X} \\
\st    &  \bm{e}_1\bm{e}_1^\top \bullet \bm{X} = 1 \\
        & \bm{X} \in \CP(\K^j_1).
\end{array}
\end{equation}
If Assumption~\ref{ass:u_set} holds, then the set
$\{ \bm{x}  \in \K_1^j : \bm{e}_1^\top\bm{x} =1  \}$ is nonempty and bounded.
Therefore by Lemma~\ref{lem:cpp_cop}, the optimal value of \eqref{equ:cp_cons-p1} is equal to that of its dual problem, shown as:
\begin{equation} \label{equ:cop_cons-p1}
\begin{array}{ll}
\inf & \beta_j \\
\st    &  \beta_j \in \RR, \ \beta_j \, \bm{e}_1 \bm{e}_1^\top  - \bm{H}_j^1(\rho,\, \bm s) \in \COP(\K_1^j).
\end{array}
\end{equation}
Then, the constraint~\eqref{equ:jth-cons}
is satisfied if and only if there exists $\theta_j$ such that
\begin{equation} \label{equ:jth-cons-ref-p1}
\begin{array} {l}
\beta_j \le \theta_j \\
 \beta_j \, \bm{e}_1 \bm{e}_1^\top  - \bm{H}_j^1(\rho,\, \bm s)  \in \COP(\K_1^j).
\end{array}
\end{equation}
Using the same argument for all $N$ constraints yields the finite constraint system
\begin{equation} \label{equ:cons_system-p1}
\begin{array} {l}
\beta_j \le \theta_j  \ \ \ \forall \, j \in [N]\\
  \beta_j \, \bm{e}_1 \bm{e}_1^\top  - \bm{H}_j^1(\rho,\, \bm s) \in \COP(\K_1^j) \ \ \ \forall \, j \in [N].
\end{array}
\end{equation}
Replacing the semi-finite constraints in~\eqref{equ:vdp_nonconvex} with the constraint system
in~\eqref{equ:cons_system-p1}, removing the variables $\theta_j$, and making $\bm s$ as the decision
variables, we end up with a copositive programming reformulation
\eqref{equ:dro-cop-p1} for~\eqref{equ:wdro}.

Second if $p=2$, then the maximization problem on the left-hand side of \eqref{equ:jth-cons} can be written as
\begin{equation} \label{equ:qp_cons-p2-equiv}
\begin{array}{ll}
\sup & (\bm u - \bm s)^\top \bm y - \rho \| \bm u - \bm{\widehat u}^j \|_2^2 \\
\st    &  (\bm u, \,  \bm y) \in \mathcal{F}_2.
\end{array}
\end{equation}
Letting $\bm{x} := (t, \, \bm{u}^\top, \,  \bm{y}^\top)^\top$, we rewrite~\eqref{equ:qp_cons-p2-equiv} as
\begin{equation} \label{equ:nqp_cons-p2}
\begin{array}{ll}
\sup & \bm{H}_j^2(\rho, \, \bm s) \bullet \bm{x}\bm{x}^\top  \\
\st    &  \bm{e}_1\bm{e}_1^\top \bullet \bm{x}\bm{x}^\top = 1 \\
         & \bm{x} \in \K^j_2.
\end{array}
\end{equation}
Using a similar argument, we can reformulate~\eqref{equ:wdro} to the copositive program
in \eqref{equ:dro-cop-p2} for the case of $p=2$.
\end{proof}

\subsection*{Proof of Proposition \ref{prop:tractable}}
\begin{proof}
Given $\rho \geq 0$ and $\bm s \in \mathcal{S}$, the problem for computing $\omega_j(\rho, \bm s)$ can be rewritten as follows:
\begin{subequations}
\begin{align}
\omega_j(\rho, \bm s) \ = & \ \sup \limits_{\bm u \in \mathcal{U}} \ \left\{\sup_{\bm y \in \mathcal{Y}} \Bigl\{\sum_{i=1}^n (u_i - s_i)y_i\Bigr\} - \rho \|\bm u - \bm{\widehat u}^j\|^p_p\right\} \\
= & \ \sup \limits_{\bm u \in \mathcal{U}} \ \sup_{\bm y \in \mathcal{Y}} \ \left\{\sum_{i=1}^n (u_i - s_i)y_i - \rho \sum_{i=1}^n|u_i - \widehat{u}^j_i|^p\right\} \\
= & \ \sup_{\bm y \in \mathcal{Y}} \ \sup_{\bm u \in \U} \ \left\{\sum_{i=1}^n \Bigl[(u_i - s_i)y_i - \rho |u_i - \hat{u}^j_i|^p\Bigr] \right\} \\
= & \ \sup_{\bm y \in \mathcal{Y}} \ \left\{\sum_{i=1}^n \biggl[ - s_iy_i + \sup_{u^{\mbox{\tiny L}}_i \leq u_i \leq u^{\mbox{\tiny U}}_i} \bigl\{y_i u_i - \rho |u_i - \hat{u}^j_i|^p \bigr\} \biggr] \right\}. \label{equ:y-convex}
\end{align}
\end{subequations}
In the maximization problem~\eqref{equ:y-convex}, the objective function is convex in variables $\bm y$ because $p \geq 1$. Hence, it suffices to consider the extreme points of
polytope $\mathcal{Y}$ to obtain $\omega_j(\rho, \bm s)$. In what follows,
we apply a similar exposition to the proof of Proposition 2 in~\cite{Mak.Rong.Zhang.2015}.
In particular, we show that each extreme point of $\mathcal{Y}$ corresponds to
a partition of the set $\{1, \ldots, n+1\}$ into intervals in the form of
$[k, \ell]_{\mathbb{Z}}$ for some $k, \ell \in [n+1]$. To this end, we remark that for any extreme point $\bm y$ of $\Y$, either
$y_n = -d_n$ or $y_n= C$ should hold (see~\cite{Zangwill.1966, Zangwill.1969}).
Moreover, for $i=2,\ldots,n$, we have $(y_{i-1} + d_{i-1})(y_{i} - y_{i-1} + c_i) = 0$,
which indicates that either $y_{i-1} = -d_{i-1}$ or $y_{i-1} = y_i + c_i$ for all $i=2,\ldots,n$.
In the latter case, the value of $y_{i-1}$ is uniquely determined by $y_i$.
Recursively applying this fact, we obtain the following findings. For any $i \in [n]$, let $\ell \in [i, n+1]_{\mathbb{Z}}$ represent the smallest index such that $y_{\ell} = -d_{\ell}$ (for notation convenience, we let $c_{n+1} := 0$, $d_{n+1} := - C$, and $y_{n+1} := C$, so that $y_{n+1} = -d_{n+1}$), then $y_i = y_{\ell} + \sum_{q=i+1}^{\ell} c_q$, i.e., $y_i = \pi_{i\ell}$, which is defined in~\eqref{equ:h-pi}.

This gives rise to a one-to-one correspondence between an extreme point $\bm y$ and a partition of $\{1, \ldots, n+1\}$ into intervals in the form of $[k, \ell]_{\mathbb{Z}}$, and $y_i = \pi_{i\ell}$ if $i \in [k, \ell]_{\mathbb{Z}}$ for some $k, \ell \in [n+1]$. Therefore, formulation \eqref{equ:y-convex} is equivalent to finding an optimal partition of the set $\{1,\ldots,n+1\}$.
To this end, for any $k \leq \ell$, we define a binary variable $t_{k\ell}$ such that $t_{k\ell} = 1$ if any only if $[k, \ell]_{\Bbb Z}$ is a
component of the partition of $\{1,\ldots,n+1\}$. The set $\big \{ t_{k\ell} \in \{0,1\}, \forall \, k \in [n+1], \forall \, \ell \in [k,n+1]_{\Bbb Z} \big \}$ represents a partition of
$\{1,\ldots,n+1\}$ if and only if
\[
\sum\limits_{k=1}^i \sum\limits_{\ell = i}^{n+1} t_{k \ell}= 1 \ \ \ \forall \, i \in [n+1].
\]
Recall that $z_{(n+1)(n+1)j} = 0$ and $z_{i\ell j} = - s_i\pi_{i\ell} + \sup_{u^{\mbox{\tiny L}}_i \leq u_i \leq u^{\mbox{\tiny U}}_i} \{\pi_{i\ell} u_i - \rho |u_i - \hat{u}^j_i|^p \}$. Then, for all $k \in [n+1]$ and $\ell \in [k, n+1]_{\Bbb Z}$, we have when $t_{k\ell} = 1$
\[
\sum\limits_{i=k}^\ell \Bigl[ - s_iy_i + \sup_{u^{\mbox{\tiny L}}_i \leq u_i \leq u^{\mbox{\tiny U}}_i} \bigl\{y_i u_i - \rho |u_i - \hat{u}^j_i|^p \bigr\} \Bigr] = \sum\limits_{i=k}^\ell z_{i\ell j}.
\]

Hence, we can obtain $\omega_j(\rho, \bm s)$ by solving the following binary integer program:
\begin{align*}
\omega_j(\rho, \bm s) \ = \ \max_t \ & \ \sum_{k=1}^{n+1} \sum_{\ell=k}^{n+1} \left( \sum_{i=k}^{\ell} z_{i\ell j} \right) t_{k\ell} \\
\mbox{s.t.} \ & \ \sum_{k=1}^{i} \sum_{\ell=i}^{n+1} t_{k\ell} = 1 \ \ \forall \, i \in [n+1] \\
                   \ & \  \ t_{k\ell} \in \{0, 1\} \ \  \forall \, k  \in [n+1], \ \forall \, \ell \in [k, n+1]_{\Bbb Z}.
\end{align*}
Furthermore, we note that the constraint matrix of this formulation is totally unimodular (see, e.g.,~\cite{Faigle.Kern.2000}). Hence, without loss of optimality we relax the binary restrictions to $t_{k\ell} \geq 0$ for all $k  \in [n+1]$ and $\ell \in [k, n+1]_{\Bbb Z}$. In addition, as $z_{(n+1)(n+1)j} = 0$, we can drop decision variable $t_{(n+1)(n+1)j}$ to obtain
\begin{subequations}
\begin{align}
\omega_j(\rho, \bm s) \ = \ \max_t \ & \ \sum_{k=1}^{n} \sum_{\ell=k}^{n+1} \left( \sum_{i=k}^{\min\{\ell, n\}} z_{i\ell j} \right) t_{k\ell} \nonumber \\
\mbox{s.t.} \ & \ \sum_{k=1}^{i} \sum_{\ell=i}^{n+1} t_{k\ell} = 1 \ \ \forall \, i \in [n] \label{tractable-note-1} \\
& \ \sum_{k=1}^{n} t_{k(n+1)} \leq 1 \label{tractable-note-2} \\
                   \ & \  \ t_{k\ell} \geq 0 \ \  \forall \, k  \in [n], \ \forall \, \ell \in [k, n+1]_{\Bbb Z}. \nonumber
\end{align}
\end{subequations}
But letting $i = n$ in constraint \eqref{tractable-note-1} yields $1 = \sum_{k=1}^{n} \sum_{\ell=n}^{n+1} t_{k\ell} = \sum_{k=1}^n t_{kn} + \sum_{k=1}^nt_{k(n+1)} \geq \sum_{k=1}^n t_{k(n+1)}$, which implies constraint \eqref{tractable-note-2}. Dropping the redundant constraint \eqref{tractable-note-2} leads to the linear programming reformulation \eqref{equ:lp-obj}--\eqref{equ:lp-con-nonnegative} of \eqref{equ:wdro}.
\end{proof}

\subsection*{Proof of Theorem \ref{thm:tractable} and Extension to the Rational $p > 1$}

\begin{proof}
First, if $p = 1$ then
\begin{align*}
z_{i\ell j} \ = & \ - s_i\pi_{i\ell} + \sup_{u^{\mbox{\tiny L}}_i \leq u_i \leq u^{\mbox{\tiny U}}_i} \{\pi_{i\ell} u_i - \rho |u_i - \widehat{u}^j_i| \} \\
= & \ \max\{(u^{\mbox{\tiny L}}_i - s_i)\pi_{i\ell} - \rho|u^{\mbox{\tiny L}}_i - \widehat{u}^j_i|, \ (\widehat{u}^j_i - s_i) \pi_{i\ell}, \ (u^{\mbox{\tiny U}}_i - s_i)\pi_{i\ell} - \rho|u^{\mbox{\tiny U}}_i - \widehat{u}^j_i|\}
\end{align*}
for all $j \in [N]$, $i \in [n]$, and $\ell \in [i,n+1]_{\mathbb{Z}}$.
Taking the dual of the linear program \eqref{equ:lp-obj}--\eqref{equ:lp-con-nonnegative} in Proposition \ref{prop:tractable} yields
\begin{equation*}
\begin{array}{rll}
\omega_j(\rho, \bm s) \ = \ \displaystyle\min_{\bm{\gamma}, \bm{z}} \ & \ \sum\limits_{i=1}^{n} \gamma_{ij} &  \\
\mbox{s.t.} \ & \ \sum\limits_{k=i}^{\min\{\ell,n\}} \gamma_{kj} \geq \sum\limits_{k=i}^{\min\{\ell,n\}} z_{k\ell j}  &  \forall \,  i \in [n], \,  \forall \, \ell \in [i, n+1]_{\Bbb Z} \\
& \ z_{i\ell j} \geq (u^{\mbox{\tiny L}}_i - s_i)\pi_{i\ell} - |u^{\mbox{\tiny L}}_i - \widehat{u}^j_i| \rho &  \forall \,  i \in [n], \,  \forall \, \ell \in [i, n+1]_{\Bbb Z} \\
& \ z_{i\ell j} \geq (\widehat{u}^j_i - s_i) \pi_{i\ell} &  \forall \,  i \in [n], \,  \forall \, \ell \in [i, n+1]_{\Bbb Z} \\
& \ z_{i\ell j} \geq (u^{\mbox{\tiny U}}_i - s_i)\pi_{i\ell} - |u^{\mbox{\tiny U}}_i - \widehat{u}^j_i| \rho &  \forall \,  i \in [n], \,  \forall \, \ell \in [i, n+1]_{\Bbb Z},
\end{array}
\end{equation*}
where dual variables $\gamma_{ij}$ are associated with primal constraints \eqref{equ:lp-con}. The proof is completed by substituting $\omega_j(\rho, \bm s) = \sup \limits_{\bm u \in \U} \ \left\{f ( \bm s, \, \bm u) - \rho \|\bm u - \bm{\widehat u}^j\|^1_1\right\}$ in formulation \eqref{equ:tractable-ref}.

Second, if $p = 2$ then
\begin{align*}
z_{i\ell j} \ = & \ - s_i\pi_{i\ell} + \sup_{u^{\mbox{\tiny L}}_i \leq u_i \leq u^{\mbox{\tiny U}}_i} \left\{\pi_{i\ell} u_i - \rho (u_i - \widehat{u}^j_i)^2 \right\} \\
= & \ - s_i\pi_{i\ell} + \inf_{\beta^{\mbox{\tiny L}}_{i\ell j}, \beta^{\mbox{\tiny U}}_{i\ell j} \geq 0} \sup_{u_i} \left\{\pi_{i\ell} u_i - \rho (u_i - \widehat{u}^j_i)^2 + \beta^{\mbox{\tiny L}}_{i\ell j}(u_i - u^{\mbox{\tiny L}}_i) + \beta^{\mbox{\tiny U}}_{i\ell j}(u^{\mbox{\tiny U}}_i - u_i) \right\} \\
= & \ - s_i\pi_{i\ell} + \inf_{\beta^{\mbox{\tiny L}}_{i\ell j}, \beta^{\mbox{\tiny U}}_{i\ell j} \geq 0} \sup_{u_i} \Bigl\{u^{\mbox{\tiny U}}_i\beta^{\mbox{\tiny U}}_{i\ell j} - u^{\mbox{\tiny L}}_i\beta^{\mbox{\tiny L}}_{i\ell j} + (\pi_{i\ell} + \beta^{\mbox{\tiny L}}_{i\ell j} - \beta^{\mbox{\tiny U}}_{i\ell j}) \widehat{u}^j_i \\
& \ + (\pi_{i\ell} + \beta^{\mbox{\tiny L}}_{i\ell j} - \beta^{\mbox{\tiny U}}_{i\ell j}) (u_i - \widehat{u}^j_i) - \rho (u_i - \widehat{u}^j_i)^2 \Bigr\} \\
= & \ - s_i\pi_{i\ell} + \inf_{\beta^{\mbox{\tiny L}}_{i\ell j}, \beta^{\mbox{\tiny U}}_{i\ell j} \geq 0} \left\{u^{\mbox{\tiny U}}_i\beta^{\mbox{\tiny U}}_{i\ell j} - u^{\mbox{\tiny L}}_i\beta^{\mbox{\tiny L}}_{i\ell j} + (\pi_{i\ell} + \beta^{\mbox{\tiny L}}_{i\ell j} - \beta^{\mbox{\tiny U}}_{i\ell j}) \widehat{u}^j_i + \frac{(\pi_{i\ell} + \beta^{\mbox{\tiny L}}_{i\ell j} - \beta^{\mbox{\tiny U}}_{i\ell j})^2}{4\rho} \right\}
\end{align*}
for all $j \in [N]$, $i \in [n]$, and $\ell \in [i,n+1]_{\mathbb{Z}}$. Taking the dual of the linear program \eqref{equ:lp-obj}--\eqref{equ:lp-con-nonnegative} in Proposition \ref{prop:tractable} yields
\begin{equation*}
\begin{array}{rll}
\omega_j(\rho, \bm s) \ = \ \displaystyle\min_{\bm{\gamma}, \bm{z}, \bm{\beta}, \bm{r}} \ & \ \sum\limits_{i=1}^{n} \gamma_{ij} &  \\
\mbox{s.t.} \ & \ \sum\limits_{k=i}^{\min\{\ell, n\}} \gamma_{kj} \geq \sum\limits_{k=i}^{\min\{\ell, n\}} z_{k\ell j}  &  \forall \,  i \in [n], \,  \forall \, \ell \in [i, n+1]_{\Bbb Z} \\
& \ z_{i\ell j} + \pi_{i\ell} s_i - (u^{\mbox{\tiny U}}_i - \widehat{u}^j_i)\beta^{\mbox{\tiny U}}_{i\ell j} & \\
& \ - (\widehat{u}^j_i - u^{\mbox{\tiny L}}_i)\beta^{\mbox{\tiny L}}_{i\ell j} - r_{i\ell j} \geq \pi_{i\ell} \widehat{u}^j_i &  \forall \,  i \in [n], \,  \forall \, \ell \in [i, n+1]_{\Bbb Z} \\[0.25cm]
& \ r_{i\ell j} \geq \frac{\left(\pi_{i\ell} + \beta^{\mbox{\tiny L}}_{i\ell j} - \beta^{\mbox{\tiny U}}_{i\ell j}\right)^2}{4\rho} & \forall \,  i \in [n], \,  \forall \, \ell \in [i, n+1]_{\Bbb Z}.
\end{array}
\end{equation*}
We represent the last constraint in the following second-order conic form:
$$
\left\| \begin{bmatrix} \pi_{i\ell} + \beta^{\mbox{\tiny L}}_{i\ell j} - \beta^{\mbox{\tiny U}}_{i\ell j}\\[0.1cm] r_{i\ell j} - \rho \end{bmatrix} \right\|_2 \leq r_{i\ell j} + \rho \ \ \forall \,  i \in [n], \,  \forall \, \ell \in [i, n+1]_{\Bbb Z}.
$$
The proof is completed by substituting $\omega_j(\rho, \bm s) = \sup \limits_{\bm u \in \U} \ \left\{f ( \bm s, \, \bm u) - \rho \|\bm u - \bm{\widehat u}^j\|^2_2\right\}$ in formulation \eqref{equ:tractable-ref}.
\end{proof}

We extend Theorem \ref{thm:tractable} and derive a second-order cone reformulation of \eqref{equ:wdro} for any rational $p > 1$. To this end, we define $q := p/(p-1)$, $q_1, q_2 \in \mathbb{N}$ such that $q = q_1/q_2$ and $q_1 > q_2$, $M := \lceil \log_2 q_1 \rceil$, and $q_3:= 2^M - q_1$. We summarize the reformulation in the following theorem and note that it involves $\mathcal{O}(N^2nq_1)$ variables and $\mathcal{O}(N^2nq_1)$ constraints.
\begin{theorem} \label{thm:p_rational}
Suppose that Assumption \ref{ass:tractable-support} holds and $p > 1$ is rational. Then, \eqref{equ:wdro} yields the same optimal value and the same set of optimal solutions as the following second-order cone program:
\begin{equation*}
\begin{array} {lll}
\min \ & \ \epsilon^p \rho + \dfrac{1}{N} \sum\limits_{j=1}^N \sum\limits_{i=1}^{n} \gamma_{ij} & \\
\mbox{s.t.} \ & \ \sum\limits_{k=i}^{\min\{\ell,n\}} \gamma_{kj} \geq \sum\limits_{k=i}^{\min\{\ell,n\}} z_{k\ell j} &   \forall \, i \in [n], \ \forall \, \ell \in [i, n+1]_{\Bbb Z}, \ \forall \, j \in [N] \\
& \ z_{i\ell j} + \pi_{i\ell} s_i - u^{\mbox{\tiny U}}_i \beta^{\mbox{\tiny U}}_{i\ell j} + u^{\mbox{\tiny L}}_i \beta^{\mbox{\tiny L}}_{i\ell j} & \\
& \ - \widehat{u}^j_i(\alpha^{\mbox{\tiny U}}_{i\ell j} - ^{\mbox{\tiny L}}_{i\ell j}) - (p^{1-q} - p^{-q}) r_{i\ell j} \geq 0 & \forall  \,  i \in [n],  \ \forall \,  \ell \in [i, n+1]_{\Bbb Z}, \  \forall \, j \in [N] \\
& \ \varphi^{i\ell j}_{1k} = r_{i\ell j} & \forall \, k \in [q_2], \ \forall  \,  i \in [n],  \ \forall \,  \ell \in [i, n+1]_{\Bbb Z}, \  \forall \, j \in [N] \\
& \ \varphi^{i\ell j}_{1k} = \rho & \forall \, k \in [q_2+1, q_1]_{\mathbb{Z}}, \ \forall  \,  i \in [n],  \ \forall \,  \ell \in [i, n+1]_{\Bbb Z}, \\
& & \forall \, j \in [N] \\
& \ \varphi^{i\ell j}_{1k} = \alpha^{\mbox{\tiny U}}_{i\ell j} + \alpha^{\mbox{\tiny L}}_{i\ell j} & \forall \, k \in [q_1+1, 2^M]_{\mathbb{Z}}, \ \forall  \,  i \in [n],  \ \forall \,  \ell \in [i, n+1]_{\Bbb Z}, \\
& & \forall \, j \in [N] \\
& \ \Biggl\|\begin{bmatrix} 2(\alpha^{\mbox{\tiny U}}_{i\ell j} + \alpha^{\mbox{\tiny L}}_{i\ell j}) \\[0.1cm] \varphi^{i\ell j}_{M1} - \varphi^{i\ell j}_{M2} \end{bmatrix}\Biggr\|_2 \leq \varphi^{i\ell j}_{M1} + \varphi^{i\ell j}_{M2} & \forall  \, i \in [n], \ \forall \,  \ell \in [i, n+1]_{\Bbb Z}, \  \forall \, j \in [N] \\[0.8cm]
& \ \left\|\begin{bmatrix} 2\varphi^{i\ell j}_{(m+1)k} \\[0.1cm] \varphi^{i\ell j}_{m(2k-1)} - \varphi^{i\ell j}_{m(2k)} \end{bmatrix}\right\|_2 \leq \varphi^{i\ell j}_{m(2k-1)} + \varphi^{i\ell j}_{m(2k)} & \forall \, m \in [M-1], \ \forall \, k \in [2^{M-m}], \ \forall  \, i \in [n], \\
& & \forall \,  \ell \in [i, n+1]_{\Bbb Z}, \  \forall \, j \in [N] \\
& \ \rho \geq 0, \ \bm s \in {\cal S}, \ \varphi^{i\ell j}_{mk} \geq 0 & \forall \, m \in [M], \ \forall \, k \in [2^{M-m+1}], \ \forall  \, i \in [n], \\
& & \forall \,  \ell \in [i, n+1]_{\Bbb Z}, \ \forall \,  j \in [N].
\end{array}
\end{equation*}
\end{theorem}
\begin{proof}
As $p > 1$, we have
\begin{align*}
z_{i\ell j} \ = & \ - s_i\pi_{i\ell} + \sup_{u^{\mbox{\tiny L}}_i \leq u_i \leq u^{\mbox{\tiny U}}_i} \left\{\pi_{i\ell} u_i - \rho |u_i - \widehat{u}^j_i|^p \right\} \\
= & \ - s_i\pi_{i\ell} + \inf_{\substack{\beta^{\mbox{\tiny L}}_{i\ell j}, \beta^{\mbox{\tiny U}}_{i\ell j}\\ \alpha^{\mbox{\tiny L}}_{i\ell j}, \alpha^{\mbox{\tiny U}}_{i\ell j} \geq 0}} \sup_{u_i, v_i} \Bigl\{\pi_{i\ell} u_i - \rho v_i^p + \beta^{\mbox{\tiny U}}_{i\ell j}(u^{\mbox{\tiny U}}_i - u_i) + \beta^{\mbox{\tiny L}}_{i\ell j}(u_i - u^{\mbox{\tiny L}}_i) \\
& \ + \alpha^{\mbox{\tiny U}}_{i\ell j}(v_i - u_i + \widehat{u}^j_i) + \alpha^{\mbox{\tiny L}}_{i\ell j}(v_i + u_i - \widehat{u}^j_i) \Bigr\} \\
= & \ - s_i\pi_{i\ell} + \inf_{\substack{\beta^{\mbox{\tiny L}}_{i\ell j}, \beta^{\mbox{\tiny U}}_{i\ell j}\\ \alpha^{\mbox{\tiny L}}_{i\ell j}, \alpha^{\mbox{\tiny U}}_{i\ell j} \geq 0}} \sup_{u_i, v_i} \Bigl\{u^{\mbox{\tiny U}}_i\beta^{\mbox{\tiny U}}_{i\ell j} - u^{\mbox{\tiny L}}_i\beta^{\mbox{\tiny L}}_{i\ell j} + \widehat{u}^j_i(\alpha^{\mbox{\tiny U}}_{i\ell j} - \alpha^{\mbox{\tiny L}}_{i\ell j}) - \rho v_i^p + (\alpha^{\mbox{\tiny U}}_{i\ell j} + \alpha^{\mbox{\tiny L}}_{i\ell j}) v_i \\
& \ + (\pi_{i\ell} + \beta^{\mbox{\tiny L}}_{i\ell j} - \beta^{\mbox{\tiny U}}_{i\ell j} - \alpha^{\mbox{\tiny U}}_{i\ell j} + \alpha^{\mbox{\tiny L}}_{i\ell j}) u_i \Bigr\} \\
= & \ - s_i\pi_{i\ell} + \inf_{\substack{\beta^{\mbox{\tiny L}}_{i\ell j}, \beta^{\mbox{\tiny U}}_{i\ell j}\\ \alpha^{\mbox{\tiny L}}_{i\ell j}, \alpha^{\mbox{\tiny U}}_{i\ell j} \geq 0}} \Bigl\{u^{\mbox{\tiny U}}_i\beta^{\mbox{\tiny U}}_{i\ell j} - u^{\mbox{\tiny L}}_i\beta^{\mbox{\tiny L}}_{i\ell j} + \widehat{u}^j_i(\alpha^{\mbox{\tiny U}}_{i\ell j} - \alpha^{\mbox{\tiny L}}_{i\ell j}) + (p^{1-q} - p^{-q})\rho^{1-q}(\alpha^{\mbox{\tiny U}}_{i\ell j} + \alpha^{\mbox{\tiny L}}_{i\ell j})^{q} \Bigr\} \\
\mbox{s.t.} \ & \ \pi_{i\ell} + \beta^{\mbox{\tiny L}}_{i\ell j} - \alpha^{\mbox{\tiny U}}_{i\ell j} - \alpha^{\mbox{\tiny U}}_{i\ell j} + \alpha^{\mbox{\tiny L}}_{i\ell j} = 0
\end{align*}
for all $j \in [N]$, $i \in [n]$, and $\ell \in [i,n+1]_{\mathbb{Z}}$. Taking the dual of the linear program \eqref{equ:lp-obj}--\eqref{equ:lp-con-nonnegative} in Proposition \ref{prop:tractable} yields
\begin{equation*}
\begin{array}{rll}
\omega_j(\rho, \bm s) \ = \ \min \ & \ \sum\limits_{i=1}^{n} \gamma_{ij} &  \\
\mbox{s.t.} \ & \ \sum\limits_{k=i}^{\min\{\ell,n\}} \gamma_{kj} \geq \sum\limits_{k=i}^{\min\{\ell,n\}} z_{k\ell j}  &  \forall \,  i \in [n], \,  \forall \, \ell \in [i, n+1]_{\Bbb Z} \\
& \ z_{i\ell j} + \pi_{i\ell} s_i - u^{\mbox{\tiny U}}_i \beta^{\mbox{\tiny U}}_{i\ell j} + u^{\mbox{\tiny L}}_i \beta^{\mbox{\tiny L}}_{i\ell j} & \\
& \ - \widehat{u}^j_i(\alpha^{\mbox{\tiny U}}_{i\ell j} - ^{\mbox{\tiny L}}_{i\ell j}) - (p^{1-q} - p^{-q}) r_{i\ell j} \geq 0 & \forall  \,  i \in [n],  \ \forall \,  \ell \in [i, n+1]_{\Bbb Z}, \  \forall \, j \in [N] \\
& \ r_{i\ell j} \geq \rho^{1-q}(\alpha^{\mbox{\tiny U}}_{i\ell j} + \alpha^{\mbox{\tiny L}}_{i\ell j})^{q} & \forall \,  i \in [n], \,  \forall \, \ell \in [i, n+1]_{\Bbb Z},
\end{array}
\end{equation*}
where dual variables $\gamma_{ij}$ are associated with primal constraints \eqref{equ:lp-con}. In addition, we rewrite the last constraint in the above formulation as
\begin{align*}
& \ \alpha^{\mbox{\tiny U}}_{i\ell j} + \alpha^{\mbox{\tiny L}}_{i\ell j} \leq r_{i\ell j}^{1/q} \rho^{(q-1)/q} \\
\Leftrightarrow \ & \ \alpha^{\mbox{\tiny U}}_{i\ell j} + \alpha^{\mbox{\tiny L}}_{i\ell j} \leq \Biggl[\underbrace{r_{i\ell j} \cdots r_{i\ell j}}_{q_2} \underbrace{\rho \cdots \rho}_{q_1 - q_2} \underbrace{(\alpha^{\mbox{\tiny U}}_{i\ell j} + \alpha^{\mbox{\tiny L}}_{i\ell j}) \cdots (\alpha^{\mbox{\tiny U}}_{i\ell j} + \alpha^{\mbox{\tiny L}}_{i\ell j})}_{q_3} \Biggr]^{\frac{1}{2^M}} \\
\Leftrightarrow \ & \ \alpha^{\mbox{\tiny U}}_{i\ell j} + \alpha^{\mbox{\tiny L}}_{i\ell j} \leq \left( \varphi^{i\ell j}_{11} \cdots \varphi^{i\ell j}_{1(2^M)} \right)^{\frac{1}{2^M}},
\end{align*}
where $\varphi^{i\ell j}_{mk}$ are defined by $\varphi^{i\ell j}_{1k} := r_{i\ell j}$ for all $k \in [q_2]$, $\varphi^{i\ell j}_{1k} := \rho$ for all $k \in [q_2+1, q_1]_{\mathbb{Z}}$, and $\varphi^{i\ell j}_{1k} := \alpha^{\mbox{\tiny U}}_{i\ell j} + \alpha^{\mbox{\tiny L}}_{i\ell j}$ for all $k \in [q_1+1, 2^M]_{\mathbb{Z}}$. Based on a seminal work on second-order conic representability (see Example 11 in Section 3.3 of~\cite{ben2001lectures}), the last inequality holds if and only if there exist $\{\varphi^{i\ell j}_{mk} \geq 0: m \in [2, M]_{\mathbb{Z}}, k \in [2^{M+1-m}]\}$ such that $\varphi^{i\ell j}_{(m+1)k} \leq \sqrt{\varphi^{i\ell j}_{m(2k-1)}\varphi^{i\ell j}_{m(2k)}}$ for all $m \in [M-1]$ and $k \in [2^{M-m}]$. This requirement can be represented in the following second-order conic form:
$$
\left\|\begin{bmatrix} 2\varphi^{i\ell j}_{(m+1)k} \\[0.1cm] \varphi^{i\ell j}_{m(2k-1)} - \varphi^{i\ell j}_{m(2k)} \end{bmatrix}\right\|_2 \leq \varphi^{i\ell j}_{m(2k-1)} + \varphi^{i\ell j}_{m(2k)}.
$$
The proof is completed by substituting $\omega_j(\rho, \bm s) = \sup \limits_{\bm u \in \U} \ \left\{f ( \bm s, \, \bm u) - \rho \|\bm u - \bm{\widehat u}^j\|^p_p\right\}$ in formulation \eqref{equ:tractable-ref}.
\end{proof}

\subsection*{Proof of Theorem \ref{thm:wc_dist}}

\begin{proof}
First, for fixed $\bar{\bm s} \in \mathcal{S}$ and $\epsilon \geq 0$, the dual problem of formulation \eqref{equ:lpdro} is
\begin{equation*} 
\begin{array}{ll}
\sup\limits_{\QQ \in \D_1(\widehat \PP^N_{\bm u}, \, \epsilon)} & \Ep_{\QQ} [f(\bar{\bm s}, \bm u)] \\
= \ \max\limits_{\bm p, \bm q, \bm r, \bm w} &  \sum\limits_{j=1}^N \sum\limits_{i=1}^{n}\sum\limits_{\ell=i}^{n+1} \pi_{i\ell} \left [ (u_i^{\mbox{\tiny L}} - \bar s_i) q_{i\ell j} + (u_i^{\mbox{\tiny U}} - \bar s_i)r_{i\ell j} + (\widehat u_i^j - \bar s_i)w_{i\ell j} \right] \\
\ \ \ \ \ \ \ \ \st & \sum\limits_{j=1}^N \sum\limits_{i=1}^{n}\sum\limits_{\ell=i}^{n+1}  \left[  (\widehat u_i^j - u_i^{\mbox{\tiny L}}) q_{i\ell j} + (u_i^{\mbox{\tiny U}} - \widehat u_i^j) r_{i\ell j}\right] \leq \epsilon  \\
 & \sum \limits_{\ell = i}^{n+1}\sum\limits_{k=1}^i p_{k\ell j} = \frac1N  \ \ \  \forall \, i \in [n], \  \forall \, j \in [N]  \\
 & \sum\limits_{k=1}^i p_{k\ell j} = q_{i\ell j} + w_{i\ell j} + r_{i\ell j} \ \ \ \forall \, i \in [n], \ \forall \, \ell \in [i, \, n+1]_{\Bbb Z}, \ \forall \,  j \in [N]  \\
& p_{i\ell j} \geq 0, \, q_{i\ell j} \geq 0, \, r_{i\ell j} \geq 0, \, w_{i\ell j} \geq 0 \ \ \ \forall \, i \in [n], \  \forall \, \ell \in [i, \, n+1]_{\Bbb Z}, \ \forall \, j \in [N],
\end{array}
\end{equation*}
where dual variables $p_{k\ell j}$, $q_{i\ell j}$, $w_{i\ell j}$, and $r_{i\ell j}$ are associated with the constraints in \eqref{equ:lpdro}. The strong duality holds because the dual feasible region is non-empty (e.g., we can set $p_{1(n+1)j} = 1/N$, $\forall j \in [N]$, all other $p_{k\ell j}$ to be zero, all $q_{i\ell j}$ and $r_{i\ell j}$ to be zero, and all $w_{i\ell j}$ to be $\sum_{k=1}^i p_{k\ell j}$). Replacing $\{p_{i\ell j}, q_{i\ell j}, r_{i\ell j}, w_{i\ell j}\}$ with $\frac1N\{p_{i\ell j}, q_{i\ell j}, r_{i\ell j}, w_{i\ell j}\}$ and substituting variables $w_{i\ell j}$ with $\sum_{k=1}^i p_{k\ell j} - q_{i\ell j} - r_{i\ell j}$ yields formulation \eqref{equ:lpdro-x-wc}.

Second, to see the existence of $\mathbb{P}^j_{\bm t}$ for all $j \in [N]$, we note that the constraint matrix defining set $\mathcal{T}$ is totally unimodular. Hence, $\mbox{conv}(\mathcal{T}) = \{\bm{t}\in [0, 1]^{(n+1)(n+2)/2}: \ \sum_{k=1}^i \sum_{\ell=i}^{n+1} t_{k\ell} = 1, \ \forall i \in [n+1]\}$ and so $\bm{p}^\star_j := [p^\star_{k\ell j}: k \in [n+1], \ell \in [k, n+1]_{\mathbb{Z}}] \in \mbox{conv}(\mathcal{T})$. It follows that there exists a finite number of points $\bm{t}^1, \ldots, \bm{t}^I \in \mathcal{T}$ and weights $\lambda^1, \ldots, \lambda^I \in [0, 1]$ such that
$$
\bm{p}^\star_j \ = \ \sum_{i=1}^I \lambda^i \bm{t}^i, \ \ \sum_{i=1}^I \lambda^i = 1.
$$
Hence, the distribution $\mathbb{P}^j_{\bm t}$ constructed by setting $\mathbb{P}^j_{\bm t}\{\bm{t} = \bm{t}^i\} = \lambda^i$ for all $i \in [I]$ fulfills the claim.

Third, to prove that $\mathbb{Q}^\star_{\bm u} \in \mathcal{D}_1(\widehat{\mathbb{P}}^N_{\bm u}, \epsilon)$, we note that $u_{i\ell j} \in [u_i^{\mbox{\tiny L}}, u_i^{\mbox{\tiny U}}]$. Indeed, on the one hand, if $\sum_{k=1}^i p^\star_{k\ell j} = 0$ then $q^\star_{i\ell j} = r^\star_{i\ell j} = 0$ by formulation \eqref{equ:lpdro-x-wc}. Thus, by the extended arithmetics $0/0=0$ we have $u_{i\ell j}^\star = \widehat u_i^j \in [u_i^{\mbox{\tiny L}}, u_i^{\mbox{\tiny U}}]$. On the other hand, if $\sum_{k=1}^i p^\star_{k\ell j} > 0$ then
$$
u_{i\ell j} = \left(\frac{\sum_{k=1}^i p^\star_{k\ell j} - q^\star_{i\ell j} - r^\star_{i\ell j}}{\sum_{k=1}^i p^\star_{k\ell j}}\right) \widehat{u}^j_i + \left(\frac{q^\star_{i\ell j}}{\sum_{k=1}^i p^\star_{k\ell j}}\right)u_i^{\mbox{\tiny L}} + \left(\frac{r^\star_{i\ell j}}{\sum_{k=1}^i p^\star_{k\ell j}}\right)u_i^{\mbox{\tiny U}}.
$$
It follows that $u_{i\ell j}$ is a convex combination of $\widehat{u}^j_i$, $u_i^{\mbox{\tiny L}}$, and $u_i^{\mbox{\tiny U}}$ and so $u_{i\ell j}^\star \in [u_i^{\mbox{\tiny L}}, u_i^{\mbox{\tiny U}}]$. In addition, we define a joint probability distribution $\mathbb{Q}_{\widehat{\bm u}, \bm u, \bm t}$ of $(\hat{\bm u}, \bm u, \bm t)$ through $\mathbb{Q}_{\widehat{\bm u}, \bm u, \bm t} = \frac{1}{N}\sum_{j=1}^N \sum_{\bm{\tau} \in \mathcal{T}} \mathbb{P}^j_{\bm t}\{\bm t = \bm \tau\} \delta_{\hat{\bm{u}}^j, \bm{u}^j(\bm \tau), \bm \tau}$. Note that the projection of $\mathbb{Q}_{\widehat{\bm u}, \bm u, \bm t}$ on $\bm u$ is $\mathbb{Q}^\star_{\bm u}$. It follows that
\begin{subequations}
\begin{align}
d_1(\mathbb{Q}^\star_{\bm u}, \widehat{\mathbb{P}}^N_{\bm u}) \ = & \ \inf \limits_{\Pi \in {\cal P}(\mathbb{Q}^\star_{\bm u}, \widehat{\mathbb{P}}^N_{\bm u})} \Ep_{\Pi} \Big[ \|\bm u - \widehat{\bm u}\|_1\Big] \nonumber \\
\leq & \ \mathbb{E}_{\mathbb{Q}_{\widehat{\bm u}, \bm u, \bm t}} \bigl[ \|\bm u - \widehat{\bm u}\|_1\bigr] \label{equ:wc-dist-note-1} \\
= & \ \sum_{i=1}^n \mathbb{E}_{\mathbb{Q}_{\widehat{\bm u}, \bm u, \bm t}} \bigl[ |u_i - \widehat{u}_i|\bigr] \nonumber \\
= & \ \sum_{i=1}^n \left(\frac{1}{N}\right) \sum_{j=1}^N \sum_{\bm \tau \in \mathcal{T}} \mathbb{P}^j_{\bm t}\{\bm t = \bm \tau\} \left| \sum_{\ell=i}^{n+1} \left( \sum_{k=1}^i \tau_{k\ell} \right) u_{i\ell j} - \widehat{u}^j_i \right| \label{equ:wc-dist-note-2} \\
= & \ \sum_{i=1}^n \left(\frac{1}{N}\right) \sum_{j=1}^N \sum_{\ell = i}^{n+1} \sum_{k=1}^i \sum_{\substack{\bm \tau \in \mathcal{T}:\\ \tau_{k\ell} = 1}} \mathbb{P}^j_{\bm t}\{\bm t = \bm \tau\} \left| u_{i\ell j} - \widehat{u}^j_i \right| \label{equ:wc-dist-note-3} \\
= & \ \sum_{i=1}^n \left(\frac{1}{N}\right) \sum_{j=1}^N \sum_{\ell = i}^{n+1} \left( \sum_{k=1}^i p^\star_{k\ell j} \right) \left| u_{i\ell j} - \widehat{u}^j_i \right| \label{equ:wc-dist-note-4} \\
= & \ \sum_{i=1}^n \left(\frac{1}{N}\right) \sum_{j=1}^N \sum_{\ell = i}^{n+1} \left( \sum_{k=1}^i p^\star_{k\ell j} \right) \left| \frac{q^\star_{i\ell j}(u_i^{\mbox{\tiny L}} - \widehat u^j_i )} {\sum_{k=1}^i p^\star_{k\ell j}} + \frac{r^\star_{i\ell j}(u_i^{\mbox{\tiny U}} - \widehat u^j_i )} {\sum_{k=1}^ip^\star_{k\ell j}} \right| \nonumber \\
\leq & \ \sum_{i=1}^n \left(\frac{1}{N}\right) \sum_{j=1}^N \sum_{\ell = i}^{n+1} \left[ q^\star_{i\ell j}\left|u_i^{\mbox{\tiny L}} - \widehat u^j_i\right| + r^\star_{i\ell j}\left|u_i^{\mbox{\tiny U}} - \widehat u^j_i\right| \right] \nonumber \\
= & \ \frac{1}{N} \sum_{j=1}^N \sum_{i=1}^n \sum_{\ell = i}^{n+1} \left[ q^\star_{i\ell j}\left(\widehat u^j_i - u_i^{\mbox{\tiny L}}\right) + r^\star_{i\ell j}\left(u_i^{\mbox{\tiny U}} - \widehat u^j_i\right) \right] \ \leq \ \epsilon, \label{equ:wc-dist-note-5}
\end{align}
\end{subequations}
where inequality \eqref{equ:wc-dist-note-1} is because the projection of $\mathbb{Q}_{\widehat{\bm u}, \bm u, \bm t}$ on $\bm u$ and $\bm \hat{\bm u}$ are $\mathbb{Q}^\star_{\bm u}$ and $\widehat{\mathbb{P}}^N_{\bm u}$, respectively, equality \eqref{equ:wc-dist-note-2} follows from the definition of $\mathbb{Q}_{\widehat{\bm u}, \bm u, \bm t}$, equality \eqref{equ:wc-dist-note-3} is because, for each $\bm \tau \in \mathcal{T}$ and $i \in [n]$, there is one and only one pair of indices $(k, \ell) \in [i] \times [i, n+1]_{\mathbb{Z}}$ such that $\tau_{k\ell} = 1$, equality \eqref{equ:wc-dist-note-4} is because $\sum_{\substack{\bm \tau \in \mathcal{T}:\\ \tau_{k\ell} = 1}} \mathbb{P}^j_{\bm t}\{\bm t = \bm \tau\} = \mathbb{P}^j_{\bm t}\{t_{k\ell} = 1\} = p^\star_{k\ell j}$, and the inequality in \eqref{equ:wc-dist-note-5} follows from the first constraint in formulation \eqref{equ:lpdro-x-wc}.

Finally, we have
\begin{subequations}
\begin{align}
\sup\limits_{\mathbb{Q}_{\bm u} \in \D_1(\widehat \PP^N_{\bm u}, \, \epsilon)} \Ep_{\QQ_{\bm u}} [f(\bar{\bm s}, \bm u)] & \geq \Ep_{\QQ^\star_{\bm u}}[f(\bar{\bm s},\bm u)] \label{equ:wc-dist-note-6} \\
& = \Ep_{\QQ^\star_{\bm u}}\left[ \max_{\bm y \in \mathcal{Y}} \left\{\sum_{i=1}^n y_i(u_i - \bar{s}_i)\right\} \right] \nonumber \\
& = \Ep_{\QQ^\star_{\bm u}}\left[ \max_{\bm t \in \mathcal{T}} \left\{\sum_{i=1}^n \sum_{\ell=i}^{n+1} \sum_{k=1}^i t_{k\ell} \pi_{i\ell} (u_i - \bar{s}_i)\right\} \right] \label{equ:wc-dist-note-7} \\
& \geq \Ep_{\QQ^\star_{\hat{\bm u}, \bm u, \bm t}}\left[ \sum_{i=1}^n \sum_{\ell=i}^{n+1} \sum_{k=1}^i t_{k\ell} \pi_{i\ell} (u_i - \bar{s}_i) \right] \label{equ:wc-dist-note-8} \\
& = \dfrac1N \sum\limits_{j=1}^N \sum_{\bm \tau \in \mathcal{T}} \mathbb{P}^j_{\bm t}\{\bm t = \bm \tau\} \sum\limits_{i=1}^{n} \sum\limits_{\ell=i}^{n+1} \sum_{k=1}^i t_{k\ell} \pi_{i\ell} ( u_{i\ell j} - \bar{s}_i) \nonumber \\
& = \dfrac1N \sum\limits_{j=1}^N \sum\limits_{i=1}^{n} \sum\limits_{\ell=i}^{n+1} \sum_{k=1}^i \sum_{\bm \tau \in \mathcal{T}} \mathbb{P}^j_{\bm t}\{\bm t = \bm \tau\} t_{k\ell} \pi_{i\ell} ( u_{i\ell j} - \bar{s}_i) \nonumber \\
& = \dfrac1N \sum\limits_{j=1}^N \sum\limits_{i=1}^{n} \sum\limits_{\ell=i}^{n+1} \sum_{k=1}^i \sum_{\substack{\bm \tau \in \mathcal{T}:\\ \tau_{k\ell} = 1}} \mathbb{P}^j_{\bm t}\{\bm t = \bm \tau\} \pi_{i\ell} ( u_{i\ell j} - \bar{s}_i) \nonumber \\
& = \dfrac1N \sum\limits_{j=1}^N \sum\limits_{i=1}^{n} \sum\limits_{\ell=i}^{n+1} \left(\sum_{k=1}^i p^\star_{k\ell j} \right) \pi_{i\ell} ( u_{i\ell j} - \bar{s}_i) \label{equ:wc-dist-note-10} \\
& = \dfrac1N \sum\limits_{j=1}^N \sum\limits_{i=1}^{n} \sum\limits_{\ell=i}^{n+1} \left(\sum_{k=1}^i p^\star_{k\ell j} \right) \pi_{i\ell} \left( \widehat{u}^j_i - \bar{s}_i + \frac{q^\star_{i\ell j}(u_i^{\mbox{\tiny L}} - \widehat u^j_i )} {\sum_{k=1}^i p^\star_{k\ell j}} + \frac{r^\star_{i\ell j}(u_i^{\mbox{\tiny U}} - \widehat u^j_i )} {\sum_{k=1}^ip^\star_{k\ell j}}\right) \nonumber \\
& = \sup\limits_{\QQ_{\bm u} \in \D_1(\widehat \PP^N_{\bm u}, \, \epsilon)} \Ep_{\QQ_{\bm u}} [f( \bar{\bm s}, \, \bm u)], \label{equ:wc-dist-note-11}
\end{align}
\end{subequations}
where inequality \eqref{equ:wc-dist-note-6} is because $\QQ_{\bm u} \in \D_1(\widehat{\PP}^N_{\bm u}, \epsilon)$, equality \eqref{equ:wc-dist-note-7} is because $\mathcal{T}$ consists of all extreme points of polytope $\mathcal{Y}$, inequality \eqref{equ:wc-dist-note-8} is because we replace the maximization over variables $\bm t$ with a distribution of $\bm t$, equality \eqref{equ:wc-dist-note-10} is because $\sum_{\substack{\bm \tau \in \mathcal{T}:\\ \tau_{k\ell} = 1}} \mathbb{P}^j_{\bm t}\{\bm t = \bm \tau\} = \mathbb{P}^j_{\bm t}\{t_{k\ell} = 1\} = p^\star_{k\ell j}$, and equality \eqref{equ:wc-dist-note-11} is because $\sup_{\mathbb{Q}_{\bm u} \in \mathcal{D}_1(\widehat{\mathbb{P}}^N_{\bm u}, \epsilon)} \mathbb{E}_{\mathbb{Q}_{\bm u}}[f(\bar{\bm s}, \bm u)]$ equals the optimal value of formulation \eqref{equ:lpdro-x-wc}. This completes the proof.
\end{proof}

\subsection*{Proof of Theorem \ref{thm:dras-ns}}
\begin{proof}
If $p=1$, we rewrite the maximization problem embedded in
 \eqref{equ:wc-dras-ns-nonconvex} for any $j \in [N]$ as
\begin{equation} \label{equ:qp_cons_noshow-0}
\begin{array}{rl}
\sup & (\bm \mu - \bm s)^\top \bm y - \rho \|\bm \mu - \bm{\widehat \mu}^j\|^1_1 - \rho \|\bm \lambda - \bm{\widehat \lambda}^j\|^1_1 \\
\st    &  (\bm \mu, \bm \lambda) \in \Xi, \, \bm y \in \Y(\bm \lambda),
\end{array}
\end{equation}
where we use the dual formulation \eqref{equ:qxsl-max} of $g(\bm s, \bm \xi)$ and the fact that $\bm \xi = [\bm \mu^\top, \bm \lambda^\top]^\top$.

Then,~\eqref{equ:qp_cons_noshow-0} can be reformulated to
\begin{equation} \label{equ:qp_cons_noshow-1}
\begin{array}{rl}
\sup & (\bm{\mu}^+ - \bm{\mu}^- + \bm{\widehat \mu}^j - \bm s)^\top \bm y - \rho \bm e^\top(\bm \mu^+ + \bm \mu^-) - \rho \bm e^\top (\bm \lambda^+ + \bm \lambda^-) \\
\st    &  (\bm \mu^+, \bm \mu^-, \bm \lambda^+, \bm \lambda^-, \bm y) \in {\cal F}_{\NS,1}^j \\
       & \bm \lambda^+ \in \{0,1\}^n, \  \bm \lambda^- \in \{0,1\}^n.
\end{array}
\end{equation}
Note that $\lambda_i^+ \in \{0,1\} \Leftrightarrow (\lambda_i^+)^2 = \lambda_i^+$. Therefore, these binary constraints can be enforced by
using $2n$ quadratic equality constraints.
Letting $\bm x := [t, (\bm \mu^+)^\top, (\bm \mu^-)^\top, (\bm \lambda^+)^\top, (\bm \lambda^-)^\top, (\bm y)^\top]^\top$, we rewrite~\eqref{equ:qp_cons_noshow-1} as
\begin{equation} \label{equ:nqp_cons_noshow-p1}
\begin{array}{ll}
\sup & \bm{G}_j^1(\rho, \, \bm s) \bullet \bm{x}\bm{x}^\top  \\
\st    &  \bm{e_1}\bm{e_1}^\top \bullet \bm{x}\bm{x}^\top = 1 \\
        &  {\bm M}_i \bullet \bm{x}\bm{x}^\top = 0  \ \ \ \forall \, i \in [n]  \\
        & {\bm J}_i \bullet \bm{x}\bm{x}^\top = 0 \ \ \ \forall \, i \in [n] \\
        &  \bm{x} \in \K_{\NS,1}^j ,
\end{array}
\end{equation}
which, by Lemma~\ref{lem:burer}, can be reformulated as the following copositive program
\begin{equation} \label{equ:cpp_cons_noshow-p1}
\begin{array}{ll}
\sup & \bm{G}_j^1(\rho, \, \bm s) \bullet \bm{X}  \\
\st    &  \bm{e}_1\bm{e}_1^\top \bullet \bm{X} = 1 \\
        &  {\bm M}_i \bullet \bm{X} = 0  \ \ \ \forall \, i \in [n] \\
        &  {\bm J}_i \bullet \bm{X} = 0  \ \ \ \forall \, i \in [n] \\
        & \bm{X} \in \CP(\K_{\NS,1}^j).
\end{array}
\end{equation}
It can be easily verified that $\{\bm x \in \K_{\NS,1}^j : \bm e^\top \bm x = 1\}$ is nonempty and bounded. Therefore, by Lemma~\ref{lem:cpp_cop}, the optimal value of~\eqref{equ:cpp_cons_noshow-p1} is equal to that of its dual problem:
\begin{equation} \label{equ:cop_cons_noshow-p1}
\begin{array}{ll}
\inf & \beta_j \\
\st    &  \beta_j \in \RR, \ \psi_{ij} \in \RR, \ \phi_{ij} \in \RR \ \forall \, i \in [n] \\
    & \beta_j \, \bm{e}_1 \bm{e}_1^\top + \sum\limits_{i=1}^n \psi_{ij} {\bm M}_i + \sum\limits_{i=1}^n \phi_{ij} {\bm J}_i - \bm{G}_j^1(\rho,\, \bm s) \in \COP(\K_{\NS,1}^j).
\end{array}
\end{equation}
Using the same argument for all $N$ maximization problems yields the copositive programming reformulation in~\eqref{equ:dro_cop2_p1} for~\eqref{equ:dras-ns}.

Second, if $p=2$, then the maximization problem embedded in
 \eqref{equ:wc-dras-ns-nonconvex} for any $j \in [N]$ as
\begin{equation} \label{equ:qp_cons_noshow-p2}
\begin{array}{rl}
\sup & (\bm \mu - \bm s)^\top \bm y - \rho \|\bm \mu - \bm{\hat \mu}^j\|^2_2 - \rho \|\bm \lambda - \bm{\hat \lambda}^j\|^2_2 \\
\st    &  (\bm \mu, \bm \lambda) \in \Xi, \, \bm y \in \Y(\bm \lambda),
\end{array}
\end{equation}
which can be reformulated to
\begin{equation} \label{equ:nqp_cons_noshow-p2}
\begin{array}{ll}
\sup & \bm{G}_j^2(\rho, \, \bm s) \bullet \bm{x}\bm{x}^\top  \\
\st    &  \bm{e_1}\bm{e_1}^\top \bullet \bm{x}\bm{x}^\top = 1, \  \bm{x} \in \K_{\NS,2}^j \\
        &  {\bm N}_i \bullet \bm{x}\bm{x}^\top = 0 \ \ \ \forall \, i \in [n],
\end{array}
\end{equation}
where $\bm x := (t, \bm \mu^\top, \bm \lambda^\top, \bm y^\top)^\top$.
Using a similar argument, we can reformulate~\eqref{equ:dras-ns} to the copositive program in~\eqref{equ:dro_cop2} for the case of $p = 2$.
\end{proof}

\subsection*{Proof of Proposition \ref{prop:ns-oc}}
\begin{proof}
First, we rewrite $\omega'_j(\rho, \bm s)$ as follows:
\begin{subequations}
\begin{align}
\omega'_j(\rho, \bm s) \ = & \ \sup \limits_{(\bm \mu, \bm \lambda) \in \Xi} \ \left\{\sup_{\bm y \in \mathcal{Y}(\bm \lambda)} \Bigl\{\sum_{i=1}^n (\mu_i - s_i)y_i\Bigr\} - \rho \| (\bm \mu^\top, \bm \lambda^\top)^{\top} - \widehat{\bm \xi}^j\|^p_p\right\} \nonumber \\
= & \ \sup \limits_{(\bm \mu, \bm \lambda) \in \Xi} \ \sup_{\bm y \in \mathcal{Y}(\bm \lambda)} \ \left\{ \sum_{i=1}^n \left(y_i(\mu_i - s_i) - \rho|\mu_i - \widehat{\mu}^j_i|^p - \rho|\lambda_i - \widehat{\lambda}^j_i|^p\right) \right\} \label{equ:tractable-no-show-note-1} \\
= & \ \sup \limits_{\bm \lambda \in \Lambda, \ \bm y \in \mathcal{Y}(\bm \lambda)} \ \left\{\sum_{i=1}^n \sup\limits_{u^{\mbox{\tiny L}}_i \lambda_i \leq \mu_i \leq u^{\mbox{\tiny U}}_i \lambda_i} \Bigl\{ \sum_{i=1}^n y_i(\mu_i - s_i) - \rho|\mu_i - \widehat{\mu}^j_i|^p - \rho|\lambda_i - \widehat{\lambda}^j_i|^p \Bigr\} \right\} \label{equ:tractable-no-show-note-2} \\
= & \ \sup \limits_{\bm \lambda \in \Lambda, \ \bm y \in \mathcal{Y}(\bm \lambda)} \ \left\{\sum_{i=1}^n f_{ij}(\lambda_i, y_i)\right\}, \label{equ:tractable-no-show-note-3}
\end{align}
\end{subequations}
where equality \eqref{equ:tractable-no-show-note-2} is because, for fixed $\bm \lambda$ and $\bm y$, the objective function in \eqref{equ:tractable-no-show-note-1} is separable in the index $i$ and in each $\mu_i$.

Second, for any fixed $\bm \lambda \in \Lambda$, $\sum_{i=1}^n f_{ij}(\lambda_i, y_i)$ is convex in $\bm y$. It follows that there exists an optimal $\bm y^*$ to problem \eqref{equ:tractable-no-show-note-3} such that $\bm y^*$ lies in an extreme point of the polytope $\mathcal{Y}(\bm \lambda)$. Hence, without loss of optimality, variables $\bm y$ satisfy the first two conditions in \eqref{equ:oc}, i.e., $y_n = C$ or $y_n = -d_0$, and $y_i = -d_0$ or $y_i = y_{i+1} + \lambda_{i+1}$ for all $i \in [n-1]$. As a result, $y_n \in \{-d_0\}\cup\{C\} \equiv \mathcal{Y}_n$. As $y_{n-1} = -d_0$ or $y_{n-1} = y_n + \lambda_n$, we have $y_{n-1} \in \{-d_0, -d_0 + 1\}\cup\{C, C+1\}$ because $\lambda_n \in \{0, 1\}$. Backward recursion of this analysis yields that $y_i \in \mathcal{Y}_i$ for all $i \in [n]$, which is the final condition in \eqref{equ:oc}. This completes the proof.
\end{proof}

\subsection*{Proof of Proposition \ref{prop:tractable-no-show-bounded-support}}
\begin{proof}
By construction, there exists a one-to-one mapping between a $(\bm \lambda, \bm y) \in \Lambda \times \mathcal{Y}(\bm \lambda)$ and an \texttt{S}--\texttt{E} path of the network $(\mathcal{G}, \mathcal{E})$. In addition, the length of the \texttt{S}--\texttt{E} path generating $(\bm \lambda, \bm y)$ coincides with $\sum_{i=1}^n f_{ij}(\lambda_i, y_i)$ by definition of $g_{k\ell j}$. Therefore, the optimal value of problem~\eqref{equ:tractable-no-show-note-3}, and so $\omega'_j(\rho, \bm s)$, equals the length of the longest path in the network $(\mathcal{G}, \mathcal{E})$, i.e., the optimal value of problem~\eqref{equ:lp-no-show-bounded-obj}--\eqref{equ:lp-no-show-bounded-con-2}.
\end{proof}

\subsection*{Proof of Theorem \ref{thm:tractable-ns-bounded}}
\begin{proof}
First, when $p = 1$, we have $f_{ij}(\lambda_i, y_i) = \displaystyle\sup_{u^{\mbox{\tiny L}}_i \lambda_i \leq \mu_i \leq u^{\mbox{\tiny U}}_i \lambda_i} \left\{y_i(\mu_i - s_i) - \rho|\mu_i - \widehat{\mu}^j_i| - \rho|\lambda_i - \widehat{\lambda}^j_i|\right\}$ for all $i \in [n]$ and $j \in [N]$. As $\lambda_i \in \{0, 1\}$,
$$
f_{ij}(\lambda_i, y_i) = \left\{ \begin{array}{ll} -y_is_i - \rho\widehat{\mu}^j_i - \rho\widehat{\lambda}^j_i & \mbox{if $\lambda_i = 0$} \\[0.5cm]
\begin{array}{l} \max\Bigl\{y_i(u^{\mbox{\tiny L}}_i - s_i) - \rho(\widehat{\mu}^j_i - u^{\mbox{\tiny L}}_i) - \rho(1 - \widehat{\lambda}^j_i), \ y_i(\widehat{\mu}^j_i - s_i) - \rho(1 - \widehat{\lambda}^j_i), \\
\ \ \ \ \ \ \ y_i(u^{\mbox{\tiny U}}_i - s_i) - \rho(u^{\mbox{\tiny U}}_i - \widehat{\mu}^j_i) - \rho(1 - \widehat{\lambda}^j_i)\Bigr\}
\end{array} & \mbox{if $\lambda_i = 1$}.
\end{array} \right.
$$

Second, taking the dual of the linear program \eqref{equ:lp-no-show-bounded-obj}--\eqref{equ:lp-no-show-bounded-con-2} yields
\begin{align*}
\omega'_j(\rho, \, \bm s) \ = \ \min_{\bm{\alpha}^j} \ & \ \alpha^j_{\texttt{S}} - \alpha^j_{\texttt{E}} \\
\mbox{s.t.} \ & \ \alpha^j_k - \alpha^j_{\ell} \geq g_{k\ell j} \ \ \forall \, (k, \ell) \in \mathcal{E},
\end{align*}
where dual variables $\alpha^j_k$, $k \in \mathcal{G}$ are associated with primal constraints~\eqref{equ:lp-no-show-bounded-con-1}. By definition, $g_{k\ell j} = 0$ if $(k, \ell) \in \mathcal{E}_{\texttt{E}}$. In addition, for all $i \in [n]$ and $j \in [N]$, if $(k,\ell) \in \mathcal{E}^0_i$ then $g_{k\ell j} = f_{ij}(0, y_i)$; and if $(k,\ell) \in \mathcal{E}^1_i$ then $g_{k\ell j} = f_{ij}(1, y_i)$. It follows that
\begin{align*}
\omega'_j(\rho, \, \bm s) \ = \ \min_{\bm{\alpha}^j} \ & \ \alpha^j_{\texttt{S}} - \alpha^j_{\texttt{E}} \\
\mbox{s.t.} \ & \ \alpha^j_k - \alpha^j_{\ell} \geq 0 \ \ \forall (k, \ell) \in \mathcal{E}_{\texttt{E}} \\
& \ \alpha^j_k - \alpha^j_{\ell} \geq - y_i s_i - (\widehat{\mu}^j_i + \widehat{\lambda}^j_i) \rho \ \ \forall \, i \in [n], \ \forall \, (k, \ell) \in \mathcal{E}^0_i \\
& \hspace{-0.15cm} \left.\begin{array}{l}
\alpha^j_k - \alpha^j_{\ell} \geq - y_i s_i - (1 - \widehat{\lambda}^j_i + \widehat{\mu}^j_i - u^{\mbox{\tiny L}}_i) \rho + u^{\mbox{\tiny L}}_i y_i \\[0.2cm]
\alpha^j_k - \alpha^j_{\ell} \geq - y_i s_i - (1 - \widehat{\lambda}^j_i) \rho + \widehat{\mu}^j_i y_i \\[0.2cm]
\alpha^j_k - \alpha^j_{\ell} \geq - y_i s_i - (1 - \widehat{\lambda}^j_i + u^{\mbox{\tiny U}}_i - \widehat{\mu}^j_i) \rho + u^{\mbox{\tiny U}}_i y_i
\end{array}\right\}
\forall \, i \in [n], \ \forall \, (k, \ell) \in \mathcal{E}^1_i.
\end{align*}
The claimed reformulation follows from substituting $\omega'_j(\rho, \bm s)$ back into formulation~\eqref{equ:wc-dras-ns-nonconvex} with the above linear program representation.

Third, when $p = 2$, we have $f_{ij}(\lambda_i, y_i) = \displaystyle\sup_{u^{\mbox{\tiny L}}_i \lambda_i \leq \mu_i \leq u^{\mbox{\tiny U}}_i \lambda_i} \left\{y_i(\mu_i - s_i) - \rho(\mu_i - \widehat{\mu}^j_i)^2 - \rho(\lambda_i - \widehat{\lambda}^j_i)^2\right\}$ for all $i \in [n]$ and $j \in [N]$. Hence, if $\lambda_i = 0$ then $f_{ij}(0, y_i) = -y_is_i - \rho \bigl(\widehat{\mu}^j_i\bigr)^2 - \rho\widehat{\lambda}^j_i$; and if $\lambda_i = 1$ then
\begin{align}
& \ f_{ij}(1, y_i) \nonumber \\
= & \ \sup_{u^{\mbox{\tiny L}}_i \leq \mu_i \leq u^{\mbox{\tiny U}}_i} \left\{y_i(\mu_i - s_i) - \rho(\mu_i - \widehat{\mu}^j_i)^2 - \rho(1 - \widehat{\lambda}^j_i)^2\right\} \nonumber \\
= & \ \inf_{\beta^{\mbox{\tiny L}} \geq 0, \ \beta^{\mbox{\tiny U}} \geq 0} \ \sup_{\mu_i \in \mathbb{R}} \left\{y_i(\mu_i - s_i) - \rho(\mu_i - \widehat{\mu}^j_i)^2 - \rho(1 - \widehat{\lambda}^j_i) + \beta^{\mbox{\tiny L}} (\mu_i - u^{\mbox{\tiny L}}_i) + \beta^{\mbox{\tiny U}} (u^{\mbox{\tiny U}}_i - \mu_i)\right\} \label{equ:ns-ref-note-1} \\
= & \ \inf_{\beta^{\mbox{\tiny L}} \geq 0, \ \beta^{\mbox{\tiny U}} \geq 0} \left\{\frac{(y_i + \beta^{\mbox{\tiny L}} - \beta^{\mbox{\tiny U}})^2}{4\rho} + y_i(\widehat{\mu}^j_i - s_i) + \beta^{\mbox{\tiny L}} (\widehat{\mu}^j_i - u^{\mbox{\tiny L}}_i) + \beta^{\mbox{\tiny U}} (u^{\mbox{\tiny U}}_i - \widehat{\mu}^j_i) - \rho(1 - \widehat{\lambda}^j_i)\right\} \nonumber \\
= & \ \inf_{\beta^{\mbox{\tiny L}} \geq 0, \ \beta^{\mbox{\tiny U}} \geq 0, \ \varphi} \left\{\varphi + y_i(\widehat{\mu}^j_i - s_i) + \beta^{\mbox{\tiny L}} (\widehat{\mu}^j_i - u^{\mbox{\tiny L}}_i) + \beta^{\mbox{\tiny U}} (u^{\mbox{\tiny U}}_i - \widehat{\mu}^j_i) - \rho(1 - \widehat{\lambda}^j_i)\right\} \nonumber \\
& \ \ \ \ \ \ \ \ \mbox{s.t.} \ \ \ \ \ \ \left\|\begin{bmatrix} \beta^{\mbox{\tiny L}} - \beta^{\mbox{\tiny U}} + y_i \\ \varphi - \rho \end{bmatrix}\right\|_2 \leq \varphi + \rho, \nonumber
\end{align}
where the strong (Lagrangian) duality holds in equality \eqref{equ:ns-ref-note-1} because the primal formulation is strictly feasible. It follows that
\begin{align*}
\omega'_j(\rho, \, \bm s) \ = \ \min_{\bm{\alpha}^j, \bm{\beta}^{\mbox{\tiny L}}, \bm{\beta}^{\mbox{\tiny U}}, \bm{\varphi}} \ & \ \alpha^j_{\texttt{S}} - \alpha^j_{\texttt{E}} \\
\mbox{s.t.} \ & \ \alpha^j_k - \alpha^j_{\ell} \geq 0 \ \ \forall (k, \ell) \in \mathcal{E}_{\texttt{E}} \\
& \ \alpha^j_k - \alpha^j_{\ell} \geq - y_i s_i - \left(\bigl(\widehat{\mu}^j_i\bigr)^2 + \widehat{\lambda}^j_i\right) \rho \ \ \forall \, i \in [n], \ \forall \, (k, \ell) \in \mathcal{E}^0_i \\
& \ \alpha^j_k - \alpha^j_{\ell} \geq - y_i s_i - (1 - \widehat{\lambda}^j_i)\rho + \varphi_{k\ell j} + (\widehat{\mu}^j_i - u^{\mbox{\tiny L}}_i) \beta^{\mbox{\tiny L}}_{k\ell j} + (u^{\mbox{\tiny U}}_i - \widehat{\mu}^j_i) \beta^{\mbox{\tiny U}}_{k\ell j} + y_i \widehat{\mu}^j_i \\
& \ \forall \, i \in [n], \ \forall \, (k, \ell) \in \mathcal{E}^1_i \\
& \ \left\|\begin{bmatrix} \beta^{\mbox{\tiny L}}_{k\ell j} - \beta^{\mbox{\tiny U}}_{k\ell j} + y_i \\ \varphi_{k\ell j} - \rho \end{bmatrix}\right\|_2 \leq r_{k\ell j} + \rho \ \ \forall \, i \in [n], \ \forall \, (k, \ell) \in \mathcal{E}^1_i.
\end{align*}
The claimed reformulation follows from substituting $\omega'_j(\rho, \bm s)$ back into formulation~\eqref{equ:wc-dras-ns-nonconvex} with the above second-order conic program representation. The proof is completed.
\end{proof}

\subsection*{Proof of Theorem \ref{thm:no-show-wc-dist}}
\begin{proof}
First, for fixed $\bar{\bm s} \in \mathcal{S}$ and $\epsilon \geq 0$, the dual problem of formulation \eqref{equ:ns-ref-1} is
\begin{align*}
\max\limits_{\bm o, \bm p, \bm q, \bm w, \bm r} \ & \ \sum\limits_{j=1}^N \sum\limits_{i=1}^{n} y_i \left[ \sum\limits_{(k, \ell) \in \mathcal{E}^1_i} \left( (u_i^{\tiny L} - \bar{s}_i) q_{k\ell j} + (\widehat{\mu}^j_i - \bar{s}_i) w_{k\ell j} + (u_i^{\tiny U} - \bar{s}_i) r_{i\ell j} \right) - \sum\limits_{(k, \ell) \in \mathcal{E}^0_i} \bar{s}_i p_{k\ell j} \right] \nonumber \\
\st \ & \ \sum\limits_{j=1}^N \sum\limits_{i=1}^{n} \Biggl[ \sum\limits_{(k, \ell) \in \mathcal{E}^0_i} (\widehat{\mu}^j_i + \widehat{\lambda}^j_i) p_{k\ell j} + \sum\limits_{(k, \ell) \in \mathcal{E}^1_i} \Biggl( \bigl(1 - \widehat{\lambda}^j_i + \widehat{\mu}^j_i - u^{\mbox{\tiny L}}_i \bigr) q_{k\ell j} \nonumber \\
& + \bigl(1 - \widehat{\lambda}^j_i\bigr) w_{k\ell j} + \bigl(1 - \widehat{\lambda}^j_i + u_i^{\tiny U} - \widehat{\mu}^j_i\bigr) r_{k\ell j} \Biggr) \Biggr] \leq \epsilon \nonumber \\
& \sum\limits_{\ell:(k, \ell) \in \mathcal{E}^0} p_{k\ell j} + \sum\limits_{\ell:(k, \ell) \in \mathcal{E}^1} (q_{k\ell j} + w_{k\ell j} + r_{k\ell j}) - \sum\limits_{\ell:(\ell, k) \in \mathcal{E}^0} p_{\ell kj} \nonumber \\
& - \sum\limits_{\ell:(\ell, k) \in \mathcal{E}^1} (q_{\ell kj} + w_{\ell kj} + r_{\ell kj}) = \left\{ \begin{array}{ll} \frac{1}{N} & \mbox{if $k = \texttt{S}$}\\ 0 & \mbox{if $k \neq \texttt{S}$} \end{array} \right. \ \ \forall \, k \in \mathcal{E}\setminus\bigl(\mathcal{L}(n)\cup\texttt{E}\bigr), \ \forall \, j \in [N] \nonumber \\
& o_{k\texttt{E} j} - \sum\limits_{\ell:(\ell, k) \in \mathcal{E}^0} p_{\ell kj} - \sum\limits_{\ell:(\ell, k) \in \mathcal{E}^1} (q_{\ell kj} + w_{\ell kj} + r_{\ell kj}) = 0 \ \ \forall \, k \in \mathcal{L}(n), \ \forall \, j \in [N] \nonumber \\
& \sum\limits_{k:(k, \texttt{E}) \in \mathcal{E}_{\texttt{E}}} o_{k\texttt{E} j} = \frac{1}{N} \ \ \forall \, j \in [N] \nonumber \\
& o_{k\ell j}, \ p_{k\ell j}, \ q_{k\ell j}, \ w_{k\ell j}, \ r_{k\ell j} \geq 0 \ \ \forall \, (k, \ell) \in \mathcal{E}, \ \forall \, j \in [N], \nonumber
\end{align*}
where dual variables $o_{k\ell j}$, $p_{k\ell j}$, $q_{k\ell j}$, $w_{k\ell j}$, and $r_{k\ell j}$ are associated with the constraints in \eqref{equ:ns-ref-1}. The strong duality holds because the primal feasible region is non-empty. Indeed, if needed, we can always increase the values of $\alpha^j_k$ for some $j \in [N]$ and $k \in \mathcal{G}$ to satisfy the constraints in \eqref{equ:ns-ref-1} (note that this can be done arbitrarily because there are no directed cycles in the network $(\mathcal{G}, \mathcal{E})$). Replacing $\{o_{k\ell j}, p_{k\ell j}, q_{k\ell j}, w_{k\ell j}, r_{k\ell j}\}$ with $\frac{1}{N}\{o_{k\ell j}, p_{k\ell j}, q_{k\ell j}, w_{k\ell j}, r_{k\ell j}\}$ yields formulation \eqref{equ:no-show-wc-lp-obj}--\eqref{equ:no-show-wc-lp-con-5}.

Second, to see the existence of $\mathbb{P}^j_{\bm z}$ for all $j \in [N]$, we note that the set $\mathcal{P}$ is defined through constraints \eqref{equ:lp-no-show-bounded-con-1}, which are network balance constraints and so yield a totally unimodular constraint matrix. Hence, $\mbox{conv}(\mathcal{P}) = \{\bm{z}\in [0, 1]^{|\mathcal{E}|}: \ \mbox{\eqref{equ:lp-no-show-bounded-con-1}}\}$. As constraints \eqref{equ:no-show-wc-lp-con-2}--\eqref{equ:no-show-wc-lp-con-5} possess the structure of network balance constraints with respect to $[\bm{p}^\star_j, \bm{\eta}^\star_j, \bm{o}^\star_j] := [p^\star_{k\ell j}: (k, \ell) \in \mathcal{E}^0; \ (q^\star_{k\ell j} + w^\star_{k\ell j} + r^\star_{k\ell j}): (k, \ell) \in \mathcal{E}^1; \ o^\star_{k\ell j}: (k, \ell) \in \mathcal{E}_{\texttt{E}}]$, we have $[\bm{p}^\star_j, \bm{\eta}^\star_j, \bm{o}^\star_j] \in \mbox{conv}(\mathcal{P})$. It follows that there exists a finite number of points $\bm{z}^1, \ldots, \bm{z}^I \in \mathcal{P}$ and weights $\theta^1, \ldots, \theta^I \in [0, 1]$ such that
$$
[\bm{p}^\star_j, \bm{\eta}^\star_j, \bm{o}^\star_j] \ = \ \sum_{i=1}^I \theta^i \bm{z}^i, \ \ \sum_{i=1}^I \theta^i = 1.
$$
Hence, the distribution $\mathbb{P}^j_{\bm z}$ constructed by setting $\mathbb{P}^j_{\bm z}\{\bm{z} = \bm{z}^i\} = \theta^i$ for all $i \in [I]$ fulfills the claim.

Third, to prove that $\mathbb{Q}^\star_{\bm \xi} \in \mathcal{D}_1(\widehat{\mathbb{P}}^N_{\bm \xi}, \epsilon)$, we note that $\mu_{k\ell j} \in [u_i^{\mbox{\tiny L}}, u_i^{\mbox{\tiny U}}]$ whenever $(k, \ell) \in \mathcal{E}^1_i$. Indeed, on the one hand, if $q^\star_{k\ell j} + w^\star_{k\ell j} + r^\star_{k\ell j} = 0$ then $q^\star_{k\ell j} = r^\star_{k\ell j} = 0$ because $q^\star_{k\ell j}, w^\star_{k\ell j}, r^\star_{k\ell j} \geq 0$. Thus, by the extended arithmetics $0/0=0$ we have $\mu_{k\ell j}^\star = \widehat{\mu}_i^j \in [u_i^{\mbox{\tiny L}}, u_i^{\mbox{\tiny U}}]$. On the other hand, if $q^\star_{k\ell j} + w^\star_{k\ell j} + r^\star_{k\ell j} > 0$ then
$$
\mu_{k\ell j} = \left(\frac{w^\star_{k\ell j}}{q^\star_{k\ell j} + w^\star_{k\ell j} + r^\star_{k\ell j}}\right) \widehat{\mu}^j_i + \left(\frac{q^\star_{k\ell j}}{q^\star_{k\ell j} + w^\star_{k\ell j} + r^\star_{k\ell j}}\right)u_i^{\mbox{\tiny L}} + \left(\frac{r^\star_{k\ell j}}{q^\star_{k\ell j} + w^\star_{k\ell j} + r^\star_{k\ell j}}\right)u_i^{\mbox{\tiny U}}.
$$
It follows that $\mu_{k\ell j}$ is a convex combination of $\widehat{\mu}^j_i$, $u_i^{\mbox{\tiny L}}$, and $u_i^{\mbox{\tiny U}}$ and so $\mu_{k\ell j}^\star \in [u_i^{\mbox{\tiny L}}, u_i^{\mbox{\tiny U}}]$. In addition, we define a joint probability distribution $\mathbb{Q}_{\widehat{\bm \xi}, \bm \xi, \bm z}$ of $(\widehat{\bm \xi}, \bm \xi, \bm z)$ through $\mathbb{Q}_{\widehat{\bm \xi}, \bm \xi, \bm z} = \frac{1}{N}\sum_{j=1}^N \sum_{\bm{\zeta} \in \mathcal{P}} \mathbb{P}^j_{\bm z}\{\bm z = \bm \zeta\} \delta_{\widehat{\bm{\xi}}^j, \bm{\xi}^j(\bm \zeta), \bm \zeta}$. Note that the projection of $\mathbb{Q}_{\widehat{\bm \xi}, \bm \xi, \bm z}$ on $\bm \xi$ is $\mathbb{Q}^\star_{\bm \xi}$. It follows that
\begin{subequations}
\begin{align}
& \ d_1(\mathbb{Q}^\star_{\bm \xi}, \widehat{\mathbb{P}}^N_{\bm \xi}) \nonumber \\
= \ & \ \inf \limits_{\Pi \in {\cal P}(\mathbb{Q}^\star_{\bm \xi}, \widehat{\mathbb{P}}^N_{\bm \xi})} \Ep_{\Pi} \bigl[ \|\bm \xi - \widehat{\bm \xi}\|_1\bigr] \nonumber \\
\leq \ & \ \mathbb{E}_{\mathbb{Q}_{\widehat{\bm \xi}, \bm \xi, \bm z}} \bigl[ \|\bm \xi - \widehat{\bm \xi}\|_1\bigr] \label{equ:ns-wc-dist-note-1} \\
= \ & \ \sum_{i=1}^n \mathbb{E}_{\mathbb{Q}_{\widehat{\bm \xi}, \bm \xi, \bm z}} \bigl[ |\mu_i - \widehat{\mu}_i| + |\lambda_i - \hat{\lambda}_i|\bigr] \nonumber \\
= \ & \ \sum_{i=1}^n \left(\frac{1}{N}\right) \sum_{j=1}^N \sum_{\bm \zeta \in \mathcal{P}} \mathbb{P}^j_{\bm z}\{\bm z = \bm \zeta\} \Biggl[\sum_{(k, \ell) \in \mathcal{E}^0_i} \zeta_{k\ell} (\widehat{\mu}^j_i + \widehat{\lambda}^j_i) + \sum_{(k, \ell) \in \mathcal{E}^1_i} \zeta_{k\ell} \bigl(|\mu_{k\ell j} - \widehat{\mu}^j_i| + 1 - \widehat{\lambda}^j_i\bigr) \Biggr] \label{equ:ns-wc-dist-note-2} \\
= \ & \ \sum_{i=1}^n \left(\frac{1}{N}\right) \sum_{j=1}^N \Biggl[ \sum_{(k, \ell) \in \mathcal{E}^0_i} \sum_{\zeta \in \mathcal{P}: \zeta_{k\ell}=1} \mathbb{P}^j_{\bm z}\{\bm z = \bm \zeta\} (\widehat{\mu}^j_i + \widehat{\lambda}^j_i) \nonumber \\
& \ + \sum_{(k, \ell) \in \mathcal{E}^1_i} \sum_{\zeta \in \mathcal{P}: \zeta_{k\ell}=1} \mathbb{P}^j_{\bm z}\{\bm z = \bm \zeta\} \bigl(|\mu_{k\ell j} - \widehat{\mu}^j_i| + 1 - \widehat{\lambda}^j_i\bigr) \Biggr] \nonumber \\
= \ & \ \sum_{i=1}^n \left(\frac{1}{N}\right) \sum_{j=1}^N \Biggl[ \sum_{(k, \ell) \in \mathcal{E}^0_i} p^\star_{k\ell j} (\widehat{\mu}^j_i + \widehat{\lambda}^j_i) \nonumber \\
& \ + \sum_{(k, \ell) \in \mathcal{E}^1_i} (q^\star_{k\ell j} + w^\star_{k\ell j} + r^\star_{k\ell j}) \left(\frac{|q^\star_{k\ell j}(u_i^{\tiny L} - \widehat{\mu}^j_i) + r^\star_{k\ell j} (u_i^{\tiny U} - \widehat{\mu}^j_i)|}{q^\star_{k\ell j} + w^\star_{k\ell j} + r^\star_{k\ell j}} + 1 - \widehat{\lambda}^j_i \right) \Biggr] \label{equ:ns-wc-dist-note-3} \\
\leq \ & \ \sum_{i=1}^n \left(\frac{1}{N}\right) \sum_{j=1}^N \Biggl[ \sum_{(k, \ell) \in \mathcal{E}^0_i} p^\star_{k\ell j} (\widehat{\mu}^j_i + \widehat{\lambda}^j_i) \nonumber \\
& \ + \sum_{(k, \ell) \in \mathcal{E}^1_i} (q^\star_{k\ell j} + w^\star_{k\ell j} + r^\star_{k\ell j}) \left(\frac{q^\star_{k\ell j}(\widehat{\mu}^j_i - u_i^{\tiny L}) + r^\star_{k\ell j} (u_i^{\tiny U} - \widehat{\mu}^j_i)}{q^\star_{k\ell j} + w^\star_{k\ell j} + r^\star_{k\ell j}} + 1 - \widehat{\lambda}^j_i \right) \Biggr] \nonumber \\
= \ & \ \frac{1}{N} \sum_{i=1}^n \sum_{j=1}^N \Biggl[ \sum_{(k, \ell) \in \mathcal{E}^0_i} p^\star_{k\ell j} (\widehat{\mu}^j_i + \widehat{\lambda}^j_i) + \sum_{(k, \ell) \in \mathcal{E}^1_i} \Bigl(q^\star_{k\ell j}(\widehat{\mu}^j_i - u_i^{\tiny L} + 1 - \widehat{\lambda}^j_i) \nonumber \\
& \ + w^\star_{k\ell j}(1 - \widehat{\lambda}^j_i) + r^\star_{k\ell j} (u_i^{\tiny U} - \widehat{\mu}^j_i + 1 - \widehat{\lambda}^j_i) \Bigr) \Biggr] \ \leq \ \epsilon, \label{equ:ns-wc-dist-note-4}
\end{align}
\end{subequations}
where inequality \eqref{equ:ns-wc-dist-note-1} is because the projections of $\mathbb{Q}_{\widehat{\bm \xi}, \bm \xi, \bm z}$ on $\bm \xi$ and $\widehat{\bm \xi}$ are $\mathbb{Q}^\star_{\bm \xi}$ and $\widehat{\mathbb{P}}^N_{\bm \xi}$, respectively, equality \eqref{equ:ns-wc-dist-note-2} follows from the definition of $\mathbb{Q}_{\widehat{\bm \xi}, \bm \xi, \bm z}$, equality \eqref{equ:ns-wc-dist-note-3} follows from the facts that, for all $j \in [N]$, $\sum_{\zeta \in \mathcal{P}: \zeta_{k\ell}=1} \mathbb{P}^j_{\bm z}\{\bm z = \bm \zeta\} = p^\star_{k\ell j}$ whenever $(k, \ell) \in \mathcal{E}^0$ and $\sum_{\zeta \in \mathcal{P}: \zeta_{k\ell}=1} \mathbb{P}^j_{\bm z}\{\bm z = \bm \zeta\} = q^\star_{k\ell j} + w^\star_{k\ell j} + r^\star_{k\ell j}$ whenever $(k, \ell) \in \mathcal{E}^1$, and the inequality in \eqref{equ:ns-wc-dist-note-4} follows from constraint \eqref{equ:no-show-wc-lp-con-1}.

Finally, we have
\begin{subequations}
\begin{align}
& \ \sup\limits_{\mathbb{Q}_{\bm \xi} \in \D_1(\widehat \PP^N_{\bm \xi}, \, \epsilon)} \Ep_{\QQ_{\bm \xi}} [g(\bar{\bm s}, \bm \xi)] \nonumber \\
\geq \ & \ \Ep_{\QQ^\star_{\bm \xi}}[g(\bar{\bm s},\bm \xi)] \label{equ:ns-wc-dist-note-5} \\
= \ & \ \Ep_{\QQ^\star_{\bm \xi}}\left[ \max_{\bm y \in \mathcal{Y}(\bm \lambda)} \left\{\sum_{i=1}^n y_i(\mu_i - \bar{s}_i)\right\} \right] \nonumber \\
\geq \ & \ \Ep_{\QQ^\star_{\widehat{\bm \xi}, \bm \xi, \bm z}}\left[ \sum_{i=1}^n y_i(\mu_i - \bar{s}_i) \right] \label{equ:ns-wc-dist-note-6} \\
= \ & \ \dfrac1N \sum\limits_{j=1}^N \sum_{\bm \zeta \in \mathcal{P}} \mathbb{P}^j_{\bm z}\{\bm z = \bm \zeta\} \sum\limits_{i=1}^{n} \Biggl[ \sum_{(k, \ell) \in \mathcal{E}^0_i} \zeta_{k\ell} \ y_i \ (0 - \bar{s}_i) + \sum_{(k, \ell) \in \mathcal{E}^1_i} \zeta_{k\ell} \ y_i \ (\mu_{k\ell j} - \bar{s}_i) \Biggr] \label{equ:ns-wc-dist-note-7} \\
= \ & \ \dfrac1N \sum\limits_{j=1}^N \sum\limits_{i=1}^{n} \Biggl[ \sum_{(k, \ell) \in \mathcal{E}^0_i} \sum_{\bm \zeta \in \mathcal{P}:\zeta_{k\ell}=1} \mathbb{P}^j_{\bm z}\{\bm z = \bm \zeta\} (- y_i \bar{s}_i) + \sum_{(k, \ell) \in \mathcal{E}^1_i} \sum_{\bm \zeta \in \mathcal{P}:\zeta_{k\ell}=1} \mathbb{P}^j_{\bm z}\{\bm z = \bm \zeta\} y_i(\mu_{k\ell j} - \bar{s}_i) \Biggr] \nonumber \\
= \ & \ \dfrac1N \sum\limits_{j=1}^N \sum\limits_{i=1}^{n} \Biggl[ \sum_{(k, \ell) \in \mathcal{E}^0_i} p^\star_{k\ell j} (- y_i \bar{s}_i) \nonumber \\
& \ + \sum_{(k, \ell) \in \mathcal{E}^1_i} (q^\star_{k\ell j} + w^\star_{k\ell j} + r^\star_{k\ell j}) y_i \Biggl(\widehat{\mu}^j_i + \frac{q^\star_{k\ell j}(u_i^{\tiny L} - \widehat{\mu}^j_i )} {q^\star_{k\ell j} + w^\star_{k\ell j} + r^\star_{k\ell j}} + \frac{r^\star_{k\ell j}(u_i^{\tiny U} - \widehat{\mu}^j_i )} {q^\star_{k\ell j} + w^\star_{k\ell j} + r^\star_{k\ell j}} - \bar{s}_i\Biggr) \Biggr] \label{equ:ns-wc-dist-note-8} \\
= \ & \ \dfrac1N \sum\limits_{j=1}^N \sum\limits_{i=1}^{n} y_i \Biggl[ - \sum_{(k, \ell) \in \mathcal{E}^0_i} p^\star_{k\ell j} \bar{s}_i + \sum_{(k, \ell) \in \mathcal{E}^1_i} \Bigl( q^\star_{k\ell j}(u_i^{\tiny L} - \bar{s}_i) + w^\star_{k\ell j}(\widehat{\mu}^j_i - \bar{s}_i) + r^\star_{k\ell j}(u_i^{\tiny U} - \bar{s}_i) \Bigr) \Biggr] \nonumber \\
= \ & \ \sup_{\mathbb{Q}_{\bm \xi} \in \mathcal{D}_1(\widehat{\mathbb{P}}^N_{\bm \xi}, \epsilon)} \mathbb{E}_{\mathbb{Q}_{\bm \xi}}[g(\bar{\bm s}, \bm{\xi})], \label{equ:ns-wc-dist-note-9}
\end{align}
\end{subequations}
where inequality \eqref{equ:ns-wc-dist-note-5} is because $\QQ^\star_{\bm \xi} \in \D_1(\widehat{\PP}^N_{\bm \xi}, \epsilon)$, inequality \eqref{equ:ns-wc-dist-note-6} is because we replace the maximization over variables $\bm y$ with a distribution of $\bm y$ pushed forward by the random path $\bm z$, equality \eqref{equ:ns-wc-dist-note-7} follows from the definition of $\QQ^\star_{\widehat{\bm \xi}, \bm \xi, \bm z}$, equality \eqref{equ:ns-wc-dist-note-8} follows from the facts that, for all $j \in [N]$, $\sum_{\zeta \in \mathcal{P}: \zeta_{k\ell}=1} \mathbb{P}^j_{\bm z}\{\bm z = \bm \zeta\} = p^\star_{k\ell j}$ whenever $(k, \ell) \in \mathcal{E}^0$ and $\sum_{\zeta \in \mathcal{P}: \zeta_{k\ell}=1} \mathbb{P}^j_{\bm z}\{\bm z = \bm \zeta\} = q^\star_{k\ell j} + w^\star_{k\ell j} + r^\star_{k\ell j}$ whenever $(k, \ell) \in \mathcal{E}^1$, and equality \eqref{equ:ns-wc-dist-note-9} is because $\sup_{\mathbb{Q}_{\bm \xi} \in \mathcal{D}_1(\widehat{\mathbb{P}}^N_{\bm \xi}, \epsilon)} \mathbb{E}_{\mathbb{Q}_{\bm \xi}}[g(\bar{\bm s}, \bm \xi)]$ equals the optimal value of formulation \eqref{equ:no-show-wc-lp-obj}--\eqref{equ:no-show-wc-lp-con-5}. This completes the proof.
\end{proof}

\bibliographystyle{plain}

\bibliography{scheduling}

\end{onehalfspace}
\end{document}